\documentclass[11pt]{amsart}
\usepackage[left=3cm,right=3cm,top=3.5cm,bottom=3cm, marginpar=2.3cm]{geometry}
 \usepackage[colorlinks]{hyperref}
\usepackage{amsthm,amssymb,amsmath,amsfonts,amscd, dsfont}
\usepackage{stmaryrd}
\usepackage{tikz-cd}
\usepackage{tikz}
\usepackage[latin1]{inputenc}
\usepackage{graphicx}
\usepackage{hyperref}
 \usepackage{latexsym}
 \usepackage{longtable,bbold}
 \usepackage[all]{xy}

\usepackage{tikz}
\usetikzlibrary{matrix,arrows}
\usepackage{mathabx,epsfig}

\newfont{\cyr}{wncyr10 scaled 1100}
\newfont{\cyrr}{wncyr9 scaled 1000}

\theoremstyle{plain}
\newtheorem{thm}{Theorem}[subsection]
\newtheorem{prop}[thm]{Proposition}
\newtheorem{cor}[thm]{Corollary}
\newtheorem{lem}[thm]{Lemma}

\theoremstyle{definition}
\newtheorem{defi}[thm]{Definition}

\newtheorem{cond}[thm]{Condition}

\theoremstyle{remark}
\newtheorem{rmk}[thm]{Remark}

\newcommand{\imaginary}{\mathrm{Im}}
\newcommand{\diffop}{\mathcal{D}}
\newcommand{\Diffop}{\diffop}
\newcommand{\IC}{\C}
\newcommand{\ci}{C^{\infty}}
\newcommand{\tr}{\mathrm{tr}}
\newcommand{\ZZ}{\mathbb{Z}}

\newcommand{\adeles}{\mathbb{A}}
\newcommand{\holoprojop}{H}

\newcommand{\closure}{\overline}

\newcommand{\HH}{{\mathcal H}}
\newcommand{\CH}{\HH}

\newcommand{\dR}{\mathrm{dR}}
\makeindex

\newcommand{\Sp}{{\operatorname{Sp}}}

\newcommand{\Sym}{{\operatorname{Sym }}}

\newcommand{\val}{{\operatorname{val}}}

\newcommand{\Spin}{\mathrm{Spin}}
\newcommand{\GSpin}{\mathrm{GSpin}}
\newcommand{\spin}{\Spin}

\newcommand{\GL}{{\operatorname{GL}}}

\newcommand{\GSp}{{\operatorname{GSp}}}

\newcommand{\Gal}{{\operatorname{Gal}}}

\newcommand{\G}{{\Gamma}}

\def\C{\mathbb{C}}

\def\Q{\mathbb{Q}}

\newcommand{\set}[1]{\left\lbrace #1 \right\rbrace}
\newcommand{\m}[1]{\mathbb{#1}}
\newcommand{\mc}[1]{\mathcal{#1}}
\newcommand{\mr}[1]{\mathrm{#1}}

\newcommand{\mb}[1]{\mathbf{#1}}

\newcommand{\IR}{\m R}
\newcommand{\IQ}{\m Q}

\begin{document}

\title[Algebraicity of Spin $L$-functions for $\mr{GSp}_6$]{Algebraicity 
of Spin $L$-functions for $\mr{GSp}_6$}
\author{Ellen Eischen}\thanks{E.E.'s research was partially supported by NSF Grants DMS-1751281, DMS-1559609, and DMS-2302011.  This material is based upon work supported by the National Science Foundation under Grant No.\ DMS-1926686.  This paper is also partly based upon work supported by the National Science Foundation under Grant No. 1440140, while this author was in residence at the Simons Laufer Mathematical Sciences Institute in Berkeley, California, during Fall 2022 and Spring 2023.  }
\author{Giovanni Rosso}\thanks{G.R.'s research was partly funded by the FRQNT grant 2019-NC-254031 and the NSERC grant RGPIN-2018-04392 }
\author{Shrenik Shah}

 \begin{abstract}
We prove algebraicity of critical values of certain Spin $L$-functions.   More precisely, our results concern $L(s, \pi \otimes \chi, \Spin)$ for cuspidal automorphic representations $\pi$ associated to a holomorphic Siegel eigenform on $\GSp_6$, real Dirichlet characters $\chi$, and critical points $s$ to the right of the center of symmetry.  We use the strategy of relating the $L$-values to properties of Eisenstein series, and a significant portion of the paper concerns the Fourier coefficients of these Eisenstein series.  Unlike in prior algebraicity results following this strategy, our Eisenstein series are on a group $G$ that has no known moduli problem, and the $L$-functions are related to the Eisenstein series through a non-unique model.   \end{abstract}

\subjclass[2010]{11F33, 11F52, 11G09}
\keywords{ }

\address{E.E.: University of Oregon, Department of Mathematics,
Eugene, Oregon, USA}
\email{\href{mailto:eeischen@uoregon.edu}{eeischen@uoregon.edu}}
\urladdr{\url{http://www.elleneischen.com}}

\address{G.R.: Concordia University, Department of Mathematics and Statistics,
Montr\'eal, Qu\'ebec, Canada}
\email{\href{mailto:giovanni.rosso@concordia.ca}{giovanni.rosso@concordia.ca}}
\urladdr{\url{https://sites.google.com/site/gvnros/}}

\address{S.S.: Center for Communications Research - La Jolla,
San Diego, California, USA}
\email{\href{mailto:snshah@ccr-lajolla.org}{snshah@ccr-lajolla.org}}

\date{version of \today}
\maketitle
{\hypersetup{linkcolor=black}
\tableofcontents
}

\section{Introduction}

We prove the algebraicity of critical values of Spin (degree $8$) $L$-functions of cuspidal automorphic representations attached to Siegel eigenforms on $\GSp_6$.  Our setup also lays a foundation for future work on $p$-adic interpolation.

\subsection{Context and relevance}\label{sec:context}
Vast conjectures about algebraic aspects of $L$-functions motivate the development of techniques to study their values.  A powerful principle that has enabled progress, including in this paper, is:
\\

\noindent {\em Principle:``Properties of Eisenstein series imply corresponding properties of certain $L$-functions.''}
\\

\noindent In spite of this principle's power and elegance, it is often challenging to implement strategies relying on it, due to sensitivity to the data to which a given $L$-function is attached.  

This principle plays a crucial role in key examples.  It is first seen in Klingen and Siegel's proof of the rationality of critical values of Dedekind zeta functions of abelian totally real fields \cite{klingen, siegel1, siegel2}, which built on Hecke's observation that the rationality of constant terms of Eisenstein series is implied by the rationality of higher order Fourier coefficients.  Shimura later proved algebraicity results for critical values of the Rankin--Selberg convolution of elliptic modular forms $f$ and $g$ \cite{shimura-RS}, by building on an observation of Rankin \cite{rankin} that the Rankin--Selberg convolution can be realized as the Petersson inner product of modular forms built from $f$, $g$, and an Eisenstein series.  

Under suitable conditions and with substantial work, Shimura's approach has been adapted to prove analogous algebraicity results for certain automorphic $L$-functions $L(s, \pi\otimes \chi, r)$, where $\pi$ is a cuspidal automorphic representation of a group $H$, $r$ is a representation of the Langlands dual group ${ }^LH$, and $\chi$ is a Dirichlet character.  For $r$ the standard representation, groups for which this has been achieved include $H$ the general symplectic group $\GSp_{2g}$ of any rank \cite{sturm-critical, hasv, shar, kozima, bouganissymplectic, PSS, HPSS} or a unitary group of any signature \cite{harrisannals, harriscrelle, shar, harrisbirkhauser, bouganis, guerberoff, guerberofflin}.

This paper is concerned with critical values of Spin $L$-functions, i.e.\ the case where $H=\GSp_{2g}$ and $r$ is the Spin representation $\Spin$.  This is a particularly natural case to consider, since ${ }^L\GSp_{2g}\cong \GSpin$ comes equipped with the $2^g$-dimensional Spin representation; yet progress on algebraicity of critical values in this setting has been scarce.  At present, the complete set of results on algebraicity of Spin $L$-functions obtained via the above principle is:
\begin{itemize}
\item{$g=1$:  In this special case, the $L$-function is a shift of the $L$-function $L(s, f\otimes\chi)$ of a modular form $f$ on $\GL_2\cong \GSp_2$, so algebraicity follows immediately from the corresponding results of Manin and Shimura for $L(s, f\otimes\chi)$ \cite{maninmfalg, shimura-RS}.}
\item{$g=2$: Harris proved certain algebraicity results, via a strategy that, once again, relies on a formulation in terms of a Petersson product of forms built from Eisenstein series and cusp forms \cite{harrisspin}.}
\item{$g=3$: Our main result, Theorem \ref{thm:algprecise}, accomplishes this.}\end{itemize}
The cases $g\geq 4$ remain open.  As noted in Section \ref{sec:connectionsconjs}, results for $\GL_n$ imply some algebraicity results for symplectic groups, but in a form that appears unsuited to certain applications.

Theorem \ref{thm:algprecise} concerns algebraicity of critical values of the Spin $L$-function for the representation $\pi\otimes\chi$ of $\GSp_6$  whose local Euler factor is, for almost all primes, of degree $8$, defined in terms of the Satake parameters of $\pi\otimes\chi$ as in \cite[Lemma 4]{AsgariSchmidt}.  (We view $\chi$ as a character of $\GSp_6$ via the multiplier, i.e. the similitude character.)  We normalize our $L$-function so that  the center of symmetry for the $L$-function in Theorem \ref{thm:algprecise} is $1/2$.  The full set of critical points is $\left\{3-r, \ldots, r-2\right\}$.  Explicitly, if the Satake parameters for $\pi$ at a prime $q$ for which $\pi$ is an unramified principal series are $b_0, b_1, b_2, b_3$, then the Satake parameters of $\pi\otimes\chi$ at $q$ are $\chi(q)b_0, b_1, b_2, b_3$, and the Euler product we consider is
\begin{align*}
L^{(M)}(s, \pi\otimes\chi, \Spin) = \prod_{q\ndivides M}\prod_{J\subseteq\left\{1, 2, 3\right\}}\left(1-\chi(q)b_0\left(\prod_{j\in J}b_j\right)q^{-s}\right)^{-1},
\end{align*}
for $M$ an integer determined by the conductor of $\chi$ and the level of $\pi$.
Our approach requires that $\pi$ be associated to a Siegel cusp form $\phi$ on $\GSp_6$ meeting the following condition.
\begin{cond}\label{cond:Tmaxorder}
In the Fourier expansion $\sum_T a(T)q^T$ of $\phi$, $a(T)\neq 0$ for some $T$ corresponding to a maximal order in a quaternion algebra over $\IQ$.
\end{cond}
Condition \ref{cond:Tmaxorder} holds for all $\phi$ of level $1$ on $\GSp_6$ \cite{boechererdas}.   It seems likely that the proof in \cite{boechererdas} can be extended to ``newforms'' of other levels, thus removing the necessity of stating this condition; but for now, how broadly Condition \ref{cond:Tmaxorder} holds remains open.  
Our proof of Theorem \ref{thm:algprecise} uses (an extension of) a result from \cite{pollack1, Evdo} that requires Condition \ref{cond:Tmaxorder}.  Although \cite{pollack1} only addresses level $1$, our extension of his result allows other levels $\Gamma^0(M)$, as well as twists by real Dirichlet characters $\chi$ (i.e. $\chi^2$ is identically $1$ on $(\ZZ/c\ZZ)^\times$ for $c$ the modulus of $\chi$).  In this case, the square of the Gauss sum $g(\chi)$ is an integer, and we denote it by $c^\ast_\chi$, as explained in Remark \ref{rmk:Gausssumvalue}.

\begin{thm}[Main theorem about algebraicity of Spin $L$-functions]\label{thm:algprecise}
Let $M$ be an integer.  Let $\pi$ be a cuspidal automorphic representation associated to a holomorphic cuspidal Siegel eigenform $\phi$ of scalar weight $2r\geq 12$ and level $\Gamma^0(M)$ on $\GSp_6$ that meets Condition \ref{cond:Tmaxorder}.  Let $s_0\in\ZZ$ be such that $4\leq s_0\leq r-2$.  Then
\begin{align*}
\frac{L^{(M)}(s_0, \pi, \Spin)}{\pi^{4s_0+6r-6} \langle \phi^\natural, \phi\rangle}\in \IQ\left(\phi\right),
\end{align*}
and more generally, if $\chi$ is a real Dirichlet character whose conductor divides $M$, then
\begin{align*}
\frac{L^{(M)}(s_0, \pi \otimes \chi, \Spin)}{\pi^{4s_0+6r-6} \langle \phi^\natural, \phi\rangle}\in \IQ\left(\phi, \sqrt{c^\ast_\chi}\right).
\end{align*}
Furthermore, for each $\sigma\in\Gal(\IC/\IQ)$,
\begin{align}\label{equ:Lfcnsigmaequivariant}
\sigma\left(\frac{L^{(M)}(s_0, \pi \otimes \chi, \Spin)}{\pi^{4s_0+6r-6} \langle\phi^\natural, \phi\rangle}\right)=\frac{L^{(M)}(s_0, \pi^\sigma \otimes \chi, \Spin)}{\pi^{4s_0+6r-6} \langle(\phi^\sigma)^\natural, \phi^\sigma\rangle}.
\end{align}
Moreover, if there is a constant $c$ so that the Fourier coefficients of $c\phi$ lie in a CM field, then 
\begin{align*}
\frac{L^{(M)}(s_0, \pi, \Spin)}{\pi^{4s_0+6r-6} \langle \phi, \phi\rangle}&\in \IQ\left(\phi\right),\\
 \frac{L^{(M)}(s_0, \pi \otimes \chi, \Spin)}{\pi^{4s_0+6r-6} \langle \phi, \phi\rangle}&\in \IQ\left(\phi, \sqrt{c^\ast_\chi}\right),
\end{align*}
and
\begin{align}
\sigma\left(\frac{L^{(M)}(s_0, \pi \otimes \chi, \Spin)}{\pi^{4s_0+6r-6} \langle\phi, \phi\rangle}\right)&=\frac{L^{(M)}(s_0, \pi^\sigma \otimes \chi, \Spin)}{\pi^{4s_0+6r-6} \langle\phi^\sigma, \phi^\sigma\rangle}\label{equ:LfcnsigmaequivariantCM} 
\end{align}
for each $\sigma\in\Gal(\IC/\IQ)$.
\end{thm}

In Theorem \ref{thm:algprecise}, the superscript $\sigma$ denotes the Galois-action defined in Section \ref{sec:petersson}, and $\phi^\natural$ is the form obtained by applying complex conjugation to the Fourier coefficients of $\phi$.  Following the conventions of \cite[Remark 1.6]{milneCM}, we denote by $\IQ^{\mathrm{cm}}$ the union of all CM subfields of $\overline{\IQ}$.  As an immediate consequence of Theorem \ref{thm:algprecise}, we obtain Corollary \ref{cor:algprecise}.
\begin{cor}\label{cor:algprecise} Let $M$ be an integer.  Let $\pi$ be a cuspidal automorphic representation associated to a holomorphic cuspidal Siegel eigenform $\phi$ of scalar weight $2r\geq 12$ and level $\Gamma^0(M)$ on $\GSp_6$ that meets Condition \ref{cond:Tmaxorder}.  Let $\chi$ be a real Dirichlet character whose conductor divides $M$.  Let $s_0\in\ZZ$ be such that $4\leq s_0\leq r-2$.
If $\phi$ has algebraic Fourier coefficients, then
\begin{align*}
\frac{L^{(M)}(s_0, \pi \otimes \chi, \Spin)}{\pi^{4s_0+6r-6} \langle\phi^\natural, \phi \rangle}\in\overline{\IQ}.
\end{align*}
Moreover, if the Fourier coefficients of $\phi$ lie in a CM field, then
\begin{align*}
\frac{L^{(M)}(s_0, \pi \otimes \chi, \Spin)}{\pi^{4s_0+6r-6} \langle\phi, \phi\rangle}\in \IQ^{\mathrm{cm}}.
\end{align*}
\end{cor}

Each of the aforementioned algebraicity results was followed years later by a proof of $p$-adic interpolation of (a suitable modification of) the critical values.  Serre constructed $p$-adic zeta functions by $p$-adically interpolating $p$-stabilizations of the Eisenstein series occurring in Klingen and Siegel's proof of algebraicity of Dedekind zeta functions \cite{serre}.  Hida constructed $p$-adic Rankin--Selberg $L$-functions by $p$-adically interpolating suitable modifications of the pairings against Eisenstein series occurring in Shimura's work \cite{hi85}.  Hida's approach inspired additional proofs of $p$-adic interpolation.  Building on the aforementioned work on algebraicity for standard automorphic $L$-functions, construction of $p$-adic standard $L$-functions has been achieved for symplectic groups \cite{BS, CouPan, liuJussieu} and unitary groups \cite{EW, HELS}.  

Since $p$-adic interpolation only makes sense for numbers known to be algebraic, the corresponding discussion for critical values of Spin $L$-functions is limited.  For $g=1$, i.e.\ exploiting the coincidence with the $L$-function of an elliptic modular form, we have Hida's method relying on Eisenstein series and a separate strategy of Mazur--Tate--Teitelbaum \cite{MTT}.  Under certain conditions, the case $g=2$ was addressed by Loeffler--Pilloni--Skinner--Zerbes by building on Harris's approach to algebraicity \cite{LPSZ}.  In this vein, we expect our results to form the foundation for a similar construction for $\GSp_6$, as addressed in Section \ref{sec:future}.

\subsection{Connections with conjectures}\label{sec:connectionsconjs}
Much of the interest in algebraic aspects of critical values stems from conjectures about $L$-functions of motivic Galois representations, e.g.\ the conjectures of Deligne about periods in \cite{Deligne}, the conjectures of Bloch--Kato relating special values of these $L$-functions to Selmer groups, and the conjectures of Coates--Perrin-Riou, Greenberg, and Iwasawa about $p$-adic behavior \cite{CoPR, green1, green2, green3}. 
At present, most  proofs of algebraicity and $p$-adic interpolation use automorphic forms as a crucial intermediary.   The Langlands program predicts that associated to certain cuspidal automorphic representations $\pi$, there is a motivic Galois representation $\rho_\pi$ whose $L$-function coincides with the $L$-function associated to $\pi$.  Recently, Kret and Shin constructed $\rho_\pi$ for our setting \cite{KretShin}.  A proof of Deligne's conjectures currently remains out of reach for most of the aforementioned $L$-functions, due to the lack of the corresponding motive.   Appendix \ref{appendix:Deligne} elaborates on the relationship between our results and the motivic setting arising in Deligne's conjectures.

At present, geometric formulations appear to be essential for interpreting arithmetic meanings of values of $L$-functions, e.g.\ concerning Galois representations and their Selmer groups in the aforementioned conjectures of Deligne and Greenberg.  Via Langlands functoriality, results about algebraicity or $p$-adic interpolation for one group may imply analogous results for other groups.  As discussed in \cite[Section 1.4]{LPSZ}, the results for $\GL_n$, $n\geq3$, in \cite{DJR, HaRa, gehrmann} imply certain algebraic (and $p$-adic) results for symplectic groups, but absent an appropriate geometric formulation, are not amenable to certain applications, e.g.\ to Euler systems.

\subsection{The approach in this paper}
Our proof of Theorem \ref{thm:algprecise} relies on a close relationship between our $L$-functions and certain Eisenstein series.  Like in the aforementioned examples, the starting point is a {\em Rankin--Selberg-style integral representation} for $L(s, \pi \otimes \chi, \spin)$, i.e.\ an expression of our $L$-function as an integral of an Eisenstein series against a cusp form.  Our challenge then, like in the examples above, is to establish algebraic properties of the integral and the Eisenstein series, which are particular to our setting.

We start with a Rankin--Selberg-style integral that Pollack introduced in his study of analytic behavior in \cite{pollack1}, and we adapt it to higher levels and twists by Dirichlet characters.  With $\phi$ as in Theorem \ref{thm:algprecise}, this integral is
\begin{align}\label{equ:I2r}
I_{2r, \chi}(\phi, s):=\int_{\GSp_6(\IQ)Z(\adeles)\backslash \GSp_6(\adeles)}\phi(g)E_{2r, \chi}^\ast(g, s)dg,
\end{align}
where $E_{2r, \chi}^\ast$ is an Eisenstein series of weight $2r$ on a particular group $G$ containing $\GSp_6$ defined in Section \ref{sec:AdelicEisensteinseries} and $Z$ denotes the center of $\GSp_6$.  
The hypothesis that $\chi$ is trivial or quadratic is a technical one to ensure that this integral is well-defined, as it ensures that $\chi$ is trivial on the center.
This gives a zeta function $I_{2r, \chi}(\phi, s)$ associated with $\phi$, similarly to the relationship between zeta functions of Siegel modular forms and $L$-functions of associated cuspidal automorphic representations discussed particularly clearly in \cite{AsgariSchmidt}.  It is quick to verify that $I_{2r, \chi}(\phi, s)$ can be expressed as a Petersson inner product for $\GSp_6$.

The mere existence of a Rankin--Selberg-style integral representation is not sufficient for proofs of algebraicity or $p$-adic interpolation, and much work is necessary even with an integral representation in hand.  Integral representations, including the one we employ from \cite{pollack1}, are typically constructed for analytic applications and without attention to algebraic consequences, hence the need for this paper, the aforementioned ones, and others.  The principle from the beginning not only powers proofs of algebraicity and $p$-adic interpolation but also enables proofs of functional equations and analytic continuation via corresponding properties of the Eisenstein series occurring in integral representations.

Thus, our principal task is to establish algebraic properties of Eisenstein series on $G$ and their restriction to $\GSp_6$.  The group $G$ presents some novel challenges for us that are not seen in prior work.  In all prior examples, the Eisenstein series were on groups that had associated Shimura varieties of PEL type, which come with moduli problems and rich geometric structures playing essential roles in establishing algebraic and $p$-adic behavior.  In contrast, the group $G$ currently has no known moduli problem and lacks associated properties, like a $q$-expansion principle, that play key roles in prior proofs of algebraicity and $p$-adic interpolation.  The group $G$ is a half-Spin group of type $D_6$ and bears some similarities to exceptional groups, a marked departure from the structures of the groups arising in proofs of algebraicity and $p$-adic interpolation for other $L$-functions.

Although our main theorem about $L$-functions requires that $\chi$ be of order dividing $2$, we place no restrictions on $\chi$ in the Eisenstein series.  The conditions on $\chi$ in our theorem about $L$-functions are purely an artifact of the integral model we have adapted from \cite{pollack1}.  We develop the Eisenstein series in full generality, since we anticipate it will be useful for applications, for example in Iwasawa theory, to have the formulation in terms of $\chi$ of arbitrarily large order.

In addition to the new challenges presented by the group $G$, the integral representation in \cite{pollack1} and Equation \eqref{equ:I2r} is a nonunique model, unlike in prior cases.  This has no major consequences for our approach to algebraicity.  It presents a challenge, though, to adapting prior approaches to calculating the modified Euler factor at $p$ arising in the associated $p$-adic $L$-function, which is important for future applications to Iwasawa theory.

\subsubsection{Organization of the paper}\label{sec:organization}

We begin by introducing properties of the objects with which we work.  Section \ref{sec:setting} introduces the group $G$ and its symmetric space.  The Eisenstein series that play a key role in our work are defined on this group.
Although $G$ is a classical group, it bears resemblance to the exceptional group $E_7$.  The group $G$ contains $\GSp_6$.  Incidentally, there is another group $\tilde{G}$ that is a double-cover of $G$ and of a particular quaternionic unitary group, but this fact turns out to be a red herring as far as our work goes.  Section \ref{sec:mforms} introduces key facts about Siegel modular forms and modular forms on $G$.

Section \ref{sec:Eseries} is devoted to the study of the Eisenstein series on $G$, upon which algebraicity properties of the $L$-function depend.  In this section, we place no restrictions on the Dirichlet character $\chi$, as we expect this general formulation will be useful for future applications.  (As noted above, the conditions on $\chi$ in our main theorem about values of $L$-functions are purely an artifact of the integral model we have adapted from \cite{pollack1} and do not impact our work with the Eisenstein series.) The main results are Theorem \ref{thm:Eseriescoeffs} and Corollary \ref{coro:Eseriescoeffs}, which state where the Fourier coefficients of these Eisenstein series lie.   In addition, Corollary \ref{cor:restrictioncor} shows that the Fourier coefficients of the restriction of these Eisenstein series from $G$ to $\GSp_6$ have similar algebraicity properties.  (Although we do not need it in the present paper, the Fourier coefficients are also expressed in a format amenable to $p$-adic interpolation, a focus of future work.)  This section begins with computation of Fourier coefficients of holomorphic Eisenstein series on $G$.  The intricate computations in this section involve local computations and require computations of both singular and nonsingular coefficients.  While the Archimedean computations are similar to those computed by Shimura in \cite{ShimConfluent}, the rest of the computations build on results for modular forms on exceptional groups whose symmetric spaces have similar properties.  In particular, the computation of singular coefficients (i.e.\ Fourier coefficients at singular matrices) build on Kim's work in \cite{KimExceptional} for modular forms on exceptional groups, and the remainder of the computations build on work of Karel and Tsao \cite{KarelFourier, Tsao} for exceptional groups.  We conclude with results about the action of analogues of Maa\ss{}--Shimura operators on modular forms on $G$, which are necessary for this paper but were not previously developed in the literature.
 
Theorem \ref{thm:algprecise} is proved in Section \ref{sec:mainthm}.  Theorem \ref{thm:intrepn} of Section \ref{sec:integralrepn} extends Pollack's integral representation for the Spin $L$-function to higher levels.  In Section \ref{sec:proofofalgthm}, we prove Theorem \ref{thm:algprecise}.  The key ingredients are the algebraicity properties of the Eisenstein series established in Section \ref{sec:Eseries}, the integral representation $I_{2r, \chi}$ from Equation \eqref{equ:I2r} whose connection to the $L$-function is established in Theorem \ref{thm:intrepn}, and properties of the Petersson inner product discussed in Section \ref{sec:mforms}.

We conclude in Appendix \ref{appendix:Deligne} by summarizing the state of the art concerning connections with the motivic setting.  Similarly to the results mentioned in Section \ref{sec:context}, our proof of Theorem \ref{thm:algprecise} employees automorphic methods, but there are conjectural applications to Deligne's conjecture. Thus, we highlight those links here and we explain why Delingne's conjectures imply that for $\mathrm{GSp}_6$ the motivic periods of the Standard representation and of the Spin representation coincide.

\subsubsection{Directions not pursued here}\label{sec:future}
There are several natural directions we have not pursued here.
  We now comment on their feasibility.  In each case, the limitations reflect similar limitations in the formulation of the integral representation we have employed from \cite{pollack1}.

Pollack only considers base field $\IQ$ and Siegel modular forms $\phi$ of scalar weight in his construction of the integral representation in \cite{pollack1}.  Consequently, we only address the case of $\GSp_6$ over $\IQ$, but it is natural to ask about working with other totally real base fields (i.e.\ about working with Hilbert--Siegel modular forms).  Likewise, we only consider $\phi$ of scalar weight, but it is natural to ask about vector-weight forms.  We expect that the construction of the integral representation could be extended to handle these more general cases, in which case we expect our methods would be straightforward to generalize to higher degree totally real fields and vector weights.  The nontrivial cost of such generalization would likely be not to us but instead to the analytic results in \cite{pollack1} concerning functional equations, which could be harder to prove under more general conditions.

Theorem \ref{thm:algprecise} does not address algebraicity for the first three points to the right of the center of symmetry.  That is, we do not include an algebraicity statement for the values $s=1,2,3$.  The Eisenstein series is not holomorphic at these values, and it is possible for the Spin $L$-function of an automorphic representation of $\GSp_6$ to have a pole as well.  For instance, Langlands functoriality suggests that representations on $\GSp_6$ arising from the transfer of a generic representation from $\mathrm{G}_2$ should have a pole at $s=1$ in their Spin $L$-function.  (See, for example, \cite{gansavin}.) There are other automorphic transfers to $\GSp_6$ that may lead to different poles as well.  A full analysis of these values would require understanding the automorphic periods that detect those transfers and the relation of these periods to the residues of the Eisenstein series, a completely different focus from this paper. 

Theorem \ref{thm:algprecise} requires that the Dirichlet character $\chi$ be real-valued.  This is an artifact of the integral representation we have adapted from \cite{pollack1}.  That integral representation is only provided for level $1$.  We have extended it to other levels.  When we twist by a Dirichlet character, that integral is easily seen to be well-defined when $\chi$ is real-valued.  It is not clear, though, that this generalization of that integral representation remains well-defined when $\chi$ is not real-valued.  More generally, we expect that the values of the twisted Spin $L$-function are not merely in $\IQ(\phi, g(\chi))$ but rather in $\IQ(\phi, \chi)g(\chi)^4$.  In anticipation of future applications, we have constructed the Eisenstein series and computed their Fourier coefficients in full generality, without conditions on the Dirichlet character $\chi$.

The final limitation concerns the $p$-adic interpolation of our special values.  The explicit Fourier coefficients of our Eisenstein series are well-suited to $p$-adic interpolation, at least after $p$-depleting the Eisenstein series $E_{2r, \chi}^\ast$.  Working with a non-unique model presents special challenges, however, for computing the precise form of the Euler factor at $p$ that results from substituting the Eisenstein series $E_{2r, \chi}^\ast$ in Equation \eqref{equ:I2r} with its $p$-depletion. 
Still, we expect these could be resolved and that the local Euler factor at $p$ will be as predicted by Coates--Perrin-Riou \cite{CoPR}, but that is its own substantial project that will involve a different set of tools from those invoked in this paper.

\subsection{Relation to other developments for Spin $L$-functions}
Several other recent results also concern algebraic and $p$-adic aspects of Spin $L$-functions.  The $p$-adic Spin $L$-functions for $\GSp_4$ constructed in \cite{LPSZ} concern critical values for representations associated to nonholomorphic forms, while our work concerns holomorphic forms and does not specialize to their setting.  
Recent work by Burgos Gil, Cauchi, Lemma, and Rodrigues Jacinto for $\GSp_6$ concerns noncritical values \cite{CLRJ1, CLRJ3, CRJDocumenta}, while the present paper handles critical values.  As explained in \cite[Appendix B]{pollack1}, the different integral representations employed in their work and ours could be related to each other and to the aforementioned integral representations, possibly via some theta correspondence. Precise details of the connection remain mysterious but are worth investigating.

There are also other integral representations for Spin $L$-functions that appear to be related, yet appear to be ill-suited to proving algebraicity at critical points.  For $\GSp_6$, the construction of Pollack and the third-named author \cite{PollackShah} is employed in the work of Cauchi, Lemma, and Rodrigues Jacinto at non-critical values but vanishes in our setting.  For $\GSp_{2g}$, $g=3, 4, 5$, the construction of Bump and Ginzburg \cite{BumpGinzburg, BuGi} enables the proof of analytic results, but if it has applications to algebraicity, they remain undiscovered.

\subsection{Acknowledgements}
We are grateful to many people--especially Siegfried B\"ocherer, Antonio Cauchi, Soumya Das, Paul Garrett, David Loeffler, Aaron Pollack, Raghuram, and Abhishek Saha--for helpful discussions, clarifications, and encouragement.  We began this project in 2017 during EPFL's special program on Euler systems and special values of $L$-functions, and we are grateful to EPFL for its hospitality.  The first named author is also grateful for an ideal working environment at the Institute for Advanced Study, where part of this paper was completed.  She is also grateful for insightful conversations with participants at the SLMath/MSRI conference ``Shimura varieties and $L$-functions'' in March 2023 and the remote RIMS conference ``Automorphic forms, Automorphic representations, Galois representations, and its related topics" in January 2021.

\section{Our setting}\label{sec:setting}
We introduce conventions for the spaces where we work, including for Hermitian matrices over quaternion algebras (\S\ref{sec:hermconv}), the group denoted $G$ that plays a crucial role in this paper (\S\ref{sec:gpG}), and the Hermitian symmetric space $\mathcal{H}$ for $G$ (\S\ref{sec:Hss}).  We adhere to the conventions established in \cite{pollack1} whenever reasonable in the context of the present paper.

Throughout the paper, we fix an algebraic closure $\closure{\IQ}$ of $\IQ$ and an embedding $\closure{\IQ}\hookrightarrow \IC$, and we identify $\closure{\IQ}$ with its image under this embedding.  We denote by $\adeles$ the adeles over $\IQ$ and by $\adeles_f$ the finite adeles.  Given a ring $R$ and $R$-algebras $S$ and $T$, we let $S_T:=S\otimes_R T$.

\subsection{Hermitian matrices over quaternion algebras}\label{sec:hermconv}
Our main construction relies on a group $G$ introduced in Section \ref{sec:gpG}.  In the present section, we introduce conventions for Hermitian matrices over quaternion algebras, which are necessary in order for us to define the group $G$.

Let $B$\index{$B$} be a quaternion algebra over a characteristic $0$ field $F$.  We denote by $D_B$\index{$D_B$} the discriminant of $B$.  For fields $L$ over $F$, we also consider $B_L=B\otimes_F L$.  When the extension to $L$ is clear from context, we write simply $B$ in place of $B_L$. For use in what follows, we shall fix a basis $\set{\mb{1}, \mb{i} ,\mb{j}, \mb{k}}$\index{$\set{\mb{1}, \mb{i} ,\mb{j}, \mb{k}}$} of $B_{\mathbb{Q}}$ over $\mathbb{Q}$
\[
B_{\mathbb{Q}} \cong  \mathbb{Q} \mb{1}\oplus \mathbb{Q} \mb{i} \oplus \mathbb{Q} \mb{j} \oplus \mathbb{Q}\mb{k}. 
\]
  We denote the conjugate of each element $a\in B$ by $a^\ast$.\index{$\ast$}  The quaternion algebra $B$ comes equipped with a trace $\tr$\index{$\tr$, trace of an element of a quaternion algebra} and a norm \index{$n$, norm of an element of a quaternion algebra}$n$, given for each $a\in B$ by
\begin{align*}
\tr(a) &=a+a^\ast\\
n(a) &= a a^\ast.
\end{align*}  
\begin{rmk}
To be consistent with pre-existing conventions, we will also use the notation $\tr$ for the trace of a Hermitian matrix introduced in Equation \eqref{equ:matrixtrace}.  Likewise, we will also denote by $n$ a particular embedding of the space of $3\times 3$ Hermitian matrices into the group $G$ introduced in Section \ref{sec:gpG}.  Since they are defined on different domains, the meaning of $\tr$ and $n$ will be  clear from context, like in \cite{pollack1}.
\end{rmk}

We work with the set $H_3(B)$\index{$H_3(B)$} of $3 \times 3$ Hermitian matrices over $B$. Following the conventions of \cite[Section 2.1]{pollack1}, we write each element $h\in H_3(B)$ as 
\begin{align}\label{equ:hermitianmxformat}
h= \left( \begin{array}{ccc}
c_1 & a_3 & a_2^*\\
a_3^* & c_2 & a_1\\
a_2 &a_1^* & c_3
\end{array} \right),
\end{align}
with $c_i \in F$ and $a_i \in B$. For such $h$, we define
\begin{align}
N(h)&:= c_1c_2c_3-c_1n (a_1)-c_2n (a_2)-c_3n(a_3)+\mr{tr}(a_1a_2a_3)\index{$N$, norm of matrix over quaternion algebra}\nonumber\\
\tr(h)&:= c_1 +c_2+c_3;\index{$\tr$, trace of a Hermitian matrix}\label{equ:matrixtrace}\\
h^{\#}&= 
\begin{pmatrix}
c_2c_3-n(a_1) & a_2^*a_1^* -c_3a_3& a_3a_1-c_2a_2^*\\
a_1a_2-c_3 a_3^* & c_1c_2-n(a_2) & a_3^*a_2^*-c_1a_1\\
a^*_1a^*_3 -c_2a_2& a_2a_3-c_1a_1^* & c_1c_2-n(a_3)
\end{pmatrix}\index{$\#$}\label{equ:hhashdefn}
\end{align}
We have the relations $h h^{\#}=h^{\#}h= N(h) 1_3$ and $(h^\#)^{\#}=N(h) h$.  For each pair of elements $x, y\in H_3(B)$, we also define
\begin{align*}
\mr{tr}(x, y)&:=\frac{1}{2}\mr{tr}(xy+yx)\\
x\times y&:=(x+y)^{\#}-x^\#-y^\#.
\end{align*}
Following the conventions of \cite[Definition 2.1]{pollack1}, we say that an element $h\in H_3(B)$ is of {\em rank 3} if $N(h)\neq 0$, of {\em rank 2} if $N(h)=0$ but $h^\#\neq 0$, and of {\em rank 1} if $h^\#=0$ but $h\neq 0$. Following the usual conventions for matrices, $h$ is of {\em rank 0} if an only if $h=0$. 
For $h$ as in Equation \eqref{equ:hermitianmxformat}, we set $N_2(h):=c_1c_2-n\left(a_3\right)$,\index{$N_2$} which is the norm of the $2\times 2$ Hermitian matrix $\begin{pmatrix}c_1 & a_3 \\ a_3^\ast & c_2\end{pmatrix}$ arising as the upper left minor of $h$, and we set $N_1(h)=c_1$\index{$N_1$}, which is norm of the $1 \times 1$ Hermitian matrix arising as the upper left minor of $h$.  We also set $N_3(h):=N(h)$.\index{$N_3$}  
 
\begin{rmk}\label{rmk:traceinvariance}
Some of the familiar conventions for matrices whose entries lie in a commutative ring do not extend in general to quaternionic matrices (i.e. matrices whose entries lie in a quaternion algebra).  For example, there is not a natural way to define quaternion-valued determinants of quaternionic matrices.  Consequently, determinants do not factor into our discussion when working with these matrices.  Likewise, traces behave differently in this setting.  For example, for arbitrary quaternionic matrices $A$ and $B$, the trace of $AB$ need not be the same as the trace of $BA$, in contrast to the case for matrices whose entries lie in a commutative ring.  That said, traces of Hermitian matrices (such as those in $H_3(B)$) are invariant under conjugation by unitary matrices.  For more on trace invariance, see \cite{traceinvariance}. \end{rmk}

\subsection{\texorpdfstring{The group $G$}{The group G}}\label{sec:gpG}
A core piece of our work concerns modular forms on a group $G$\index{$G$} that preserves (up to a similitude) a particular symplectic form $\langle, \rangle$ and a particular quartic form $Q$ on the $32$-dimensional $F$-vector space\index{$W_F$}
\begin{align*}
W_F:=F \oplus H_3(B) \oplus H_3(B) \oplus  F,\index{$W$}  
\end{align*}
with $B$ a quaternion algebra over a characteristic $0$ field $F$ like above.  Following \cite[Section 2.2]{pollack1}, we define the symplectic form $\langle, \rangle$ and the quartic form $Q$ by
\begin{align*}
\langle (a, b, c, d), (a', b', c', d') \rangle &:= ad'-\tr(b, c')+\tr(c, b')-da'\\
Q\left(\left(a, b, c, d\right)\right)&:= (ad-tr(b, c))^2+4aN(c)+4dN(b)-4\tr\left(b^\#, c^\#\right)
\end{align*}
for all $(a, b, c, d), (a', b', c', d')\in W_F$.  Our convention is for $GL\left(W_F\right)$ to act on $W_F$ on the right.  Then $G$ is the algebraic group whose points over each $F$-algebra $L$ are given by 
\begin{align*}
G(L) = \{(g, \nu(g))\in \GL\left(W_F\otimes_F L\right)\times \GL_1(L) \mid  &\langle ug, v g\rangle = \nu(g)\langle u, v\rangle\\
& \mbox{ and } Q(wg) = \nu(g)^2 Q(w) \forall u, v, w\in W_F\otimes_F L\}.
\end{align*}
That is, $G$ is a similitude group, and $\nu(g)$ is the similitude factor of each element $g$.  

\subsubsection{\texorpdfstring{Embedding of $\GSp_6$ in $G$ as algebraic groups over $\ZZ$}{Embedding of GSp6 in G as algebraic groups over Z}}\label{sec:GSp6embedding}
Going forward, we restrict ourselves to the case of $B$ a quaternion algebra over $F=\IQ$  that is ramified at $\infty$, and we fix a maximal order $B_0$\index{$B_0$} in $B$.  Under these conditions, we embed $\GSp_6$ inside $G$ and also realize $G$ as an algebraic group over $\ZZ$, as follows.  Let $W_6$\index{$W_6$} be the defining $6$-dimensional representation for $\GSp_6$ with invariant symplectic form $\langle, \rangle_6$.\index{$\langle , \rangle_6$}  When it is clear from context that we are working with $\GSp_6$ rather than $G$, we write omit the subscript $6$ from the subscript on $\langle, \rangle_6$.  We view $\GSp_6$ as acting on $W_6$ on the right, and we fix a symplectic basis $e_1, e_2, e_3, f_1, f_2, f_3$\index{$e_1, e_2, e_3, f_1, f_2, f_3$, symplectic basis for $W_6$} for $W_6$, i.e.\ $\langle e_i, f_i\rangle = \delta_{ij}$,\index{$\delta_{ij}$} where $\delta_{ij}$ denotes the Kronecker delta.  With respect to this ordered basis, $\GSp_6$ is the set of $g\in \GL_6$ such that $gJ { }^tg = \nu(g) J$ for $J=J_6=\begin{pmatrix}0& 1_3\\ -1_3 & 0\end{pmatrix}$, where $1_3$ denotes the $3\times 3$ identity matrix.\index{$J$}\index{$J_6$}\index{$1_3$}  We let $W_6(\ZZ)$\index{$W_6(\ZZ)$} denote the $\ZZ$-lattice in $W$ generated by this symplectic basis.  Following \cite[Section 5.1]{pollack1}, we identify $W_\IQ$ with its image in $\bigwedge^3W_6\otimes\nu^{-1}\otimes_\IQ B$ under the linear map
\begin{align*}
W_\IQ&\rightarrow\bigwedge^3W_6\otimes\nu^{-1}\otimes_\IQ B\\
(a, b, c, d)&\mapsto a e_1\wedge e_2 \wedge e_3 + \left( \sum_{i,j}b_{ij} e_i^* \wedge f_j\right) +  \left( \sum_{i,j} c_{ij}f_i^* \wedge e_j\right) + d f_1\wedge f_2 \wedge f_3,
\end{align*}
where $b=\left(b_{ij}\right)$, $c=\left(c_{ij}\right)$, and $e_i^*=e_{i+1}\wedge e_{i+2}$ with indices taken modulo $3$, and similarly for $f_i^*$.

\begin{lem}[Lemma 5.1 of \cite{pollack1}]\label{lem:pollack51}
The inclusion $W_\IQ\hookrightarrow \bigwedge^3W_6\otimes\nu^{-1}\otimes_\IQ B$ induces a similitude-preserving inclusion $\GSp_6\hookrightarrow G$, via the natural action of $\GSp_6$ on $\bigwedge^3W_6\otimes\nu^{-1}$.
\end{lem}

Setting
\begin{align*}
W:=W(\ZZ):= W_\IQ\cap\left(\left(\bigwedge^3W_6(\ZZ)\otimes\nu^{-1} \right)\otimes_\ZZ B_0\right),
\end{align*}
we can now realize $G$ as the algebraic group over $\ZZ$ whose $R$-points for a ring $R$ are given by
\begin{align*}
G(R) = \{(g, \nu(g))\in \GL\left(W\otimes_\ZZ R\right)\times \GL_1(R) \mid  &\langle ug, v g\rangle = \nu(g)\langle u, v\rangle\\
& \mbox{ and } Q(wg) = \nu(g)^2 Q(w) \forall u, v, w\in W\otimes_\ZZ R\}.
\end{align*}
Then similarly to Lemma \ref{lem:pollack51}, $\GSp_6$, viewed as an algebraic group over $\ZZ$ preserving $\langle, \rangle_6$, is a subgroup of $G$.

We denote by $G^1$ the elements of $G$ with similitude $1$, and we denote by \index{$G^+$}$G^+(\IR)\subseteq G(\IR)$ the connected component of the identity.  We write \index{$G^{1, +}$}$G^{1, +}(\IR)$ for the elements in the connected component of the identity that have similitude $1$.  We let $\GSp_6^+(\IR)$ denote the subgroup of $\GSp_6(\IR)$ of elements with positive similitude.

\subsubsection{\texorpdfstring{Embeddings $n$ and $\bar{n}$ of Hermitian matrices in $G$}{Embeddings n and bar{n} of Hermitian matrices in G}}
By \cite[Lemma 2.3]{pollack1}, we have injective group homomorphisms 
\begin{align*}
n: H_3(B)&\hookrightarrow G(\m Q)\\
\bar{n}: H_3(B)&\hookrightarrow G(\m Q)
\end{align*}
 defined as follows.  The embedding $n$ maps $X \in H_3(B)$ to the element \index{$n$, embedding of Hermitian matrices into $G$}$n(X) \in G(\m Q)$
that acts on elements $(a, b, c, d)\in W$ by
\begin{align*}
(a,b,c,d)n(X)=(a, b+aX,c+b\times X + a X^{\#},d+\mr{tr}(c,X)+\mr{tr}(b,X^{\#})+a\mr {N}(X)).
\end{align*}
The embedding $\bar{n}$\index{$\bar{n}$, embedding of Hermitian matrices into $G$} maps $X \in H_3(B)$ to $\bar{n}(X)\in G(\m Q)$ that acts on elements $(a, b, c, d)\in W$ by
\[
(a,b,c,d)\bar{n}(X)=(a+\mr{tr}(b,X)+\mr{tr}(c,X^{\#})+d N(X), b+c\times X + d X^{\#}, c+d X, d).
\]
Both $n(X)$ and $\bar{n}(X)$ have similitude $1$.  
Furthermore, since $n$ and $\bar{n}$ are group homomorphisms,
\begin{align}
n\left(X_1+X_2\right) = n\left(X_1\right)n\left(X_2\right)\label{equ:nembedding}\\
\bar{n}\left(X_1+X_2\right)= \bar{n}\left(X_1\right)\bar{n}\left(X_2\right)\nonumber
\end{align}
for all $X_1, X_2\in H_3(B)$.

\subsubsection{\texorpdfstring{Parabolic subgroup $P$ in $G$}{Parabolic subgroup P in G}} Similarly to \cite[p. 1406]{pollack1}, we define $P\subset G$ to be the (maximal) parabolic subgroup stabilizing the line $\IQ f\subset W_\IQ$, where
\begin{align*}
f:=(0, 0, 0, 1)\in W.
\end{align*} 
Observe that ${n}(X)$ belongs to $P$.  
By \cite[Lemma 4.5]{Tsao}, we have the Bruhat decomposition
\begin{align}\label{equ:bruhat}
G(\m Q) = \bigcup_{j=0}^3 P(\m Q) \iota_j P(\m Q),
\end{align}
where \index{$\iota_0$}$\iota_0 = \mr{Id}$, \index{$\iota_1$}$\iota_1=w_3$, \index{$\iota_2$}$\iota_2=w_2w_3$, and  \index{$\iota_3$}$\iota_3=w_1w_2w_3$, with $w_i$ representatives for the simple reflections in $(\m Z/2 \m Z)^{3}$ inside the Weyl group of $G$, which is of type $C_6$.  (N.B. The right hand side of Equation \eqref{equ:bruhat} could be written as disjoint union, but people usually write it how we have presented it here.)
Explicitly, they are the images (denoted by the same symbols) inside $G$ of the following elements of $\mr{GSp}_6$:
\begin{align*}
\iota_0=& \left(\begin{array}{ccc|ccc} 1&&&&&\\ &1&&&&\\ &&1&&& \\  \hline &&& 1&& \\&&& &1& \\ &&&&&1  \end{array}\right)\\
\iota_1=&\left(\begin{array}{ccc|ccc} 1&&& &&\\ &1&&&&\\ &&&&& 1 \\  \hline && &1 && \\&&& &1& \\ &&-1 &&&  \end{array}\right)\\
\iota_2=& \left(\begin{array}{ccc|ccc} 1&&& &&\\ & &&& 1&\\ &&&&& 1 \\  \hline && &1 && \\&-1&& && \\ &&-1 &&&  \end{array}\right)\\
\iota_3=& \left(\begin{array}{ccc|ccc} &&& 1&&\\ &&&&1&\\ &&&&&1 \\  \hline-1&&& && \\&-1&& && \\ &&-1&&&  \end{array}\right)=J.
\end{align*}

\subsubsection{Congruence subgroups}We fix a positive integer \index{$M$, a positive integer}$M$, and we denote by \index{$\Gamma_G^0(M)$}$\Gamma_G^0(M)$ the subgroup of $G(\ZZ)$ consisting of matrices reducing to elements of $P(\m Z/ M \m Z)$. We define 
\begin{align}\label{equ:Gammainfty}
\Gamma_{\infty}:=G(\mathbb Z)\cap P(\m Q).\index{$\Gamma_{\infty}$}
\end{align}
Note that in Section \ref{section:Siegelforms}, we also introduce a congruence subgroup $\Gamma^0(M)$ in $\GSp_6$, whose image lies inside this bigger group.  We define the Atkin--Lehner involution \index{$w_M$}$w_M$ for $G$ so that it normalizes the congruence subgroup $\Gamma^0(M)$ in $\GSp_6$.  The matrix defining the Atkin--Lehner involution for $G$ is the image in $G$ of the $6 \times 6$ matrix 
\begin{align*}
\left( \begin{array}{cc}
0 & M \\
-1 & 0
\end{array} 
\right).
\end{align*}

\subsection{The Hermitian symmetric space $\CH$ for $G(\IR)$}\label{sec:Hss} 
The Hermitian symmetric space for $G(\IR)$ is
\begin{align*}
\CH :=\set{Z=X+iY \in H_3\left(B_\IC\right) | Y \mbox{ is positive definite}},\index{$\CH$}
\end{align*}
where (following our conventions from Section \ref{sec:hermconv}), $B_\IC=B\otimes_{\IQ}\IC$. We stress that the $i$ appearing in $X+iY$ is the element $i \in \m Q \otimes_{\m Q} \m C$  in the center of  $B_\IC$.
This is also the same Hermitian space denoted by $\mathfrak{I} =\mc J + i \mc R $ in \cite[\S 2.1]{Tsao},  where $\mc J$ is the compact real Jordan algebra of type C in \cite[\S 3.1]{Tsao}, with $n=3$. This space is also introduced in \cite[Section 1]{ShimConfluent}, where it is called Case III.  Given $Z\in \CH$, instead of using the convention from Equation \eqref{equ:hermitianmxformat}, it is sometimes convenient to write
\begin{align}\label{equ:ZinHcoords}
Z=\begin{pmatrix}
Z_{11} & Z_3 & Z_2^*\\
Z_3^* & Z_{22} & Z_1\\
Z_2 &Z_1^* & Z_{33}
\end{pmatrix}.
\end{align}
For $Z = X+iY\in \CH$, i.e.\ $X, Y\in H_3(B_\IC)$ with $Y$ positive definite and $i\in\IC$, we put
\begin{align*}
\bar{Z} := X-iY.
\end{align*}
We write
\begin{align*}
\imaginary Z := Y.
\end{align*}

\subsubsection{Action of $G(\IR)$ on $\CH$}
Following the conventions of \cite[Section 6.2.1]{pollack1}, we explain how $G(\IR)$ acts on $\CH$.  In the vector space $W_\IQ$, we set \index{$e$, the vector $(1, 0,0,0)\in W$}$e=(1,0,0,0)$, like in \cite[Section 6.2.1]{pollack1}.  We realize an element $Z \in \mc H$ as a vector in $W_\IC=W_\IR \otimes_\IR \m C$ via
\begin{align}\label{enofZ}
r: \CH&\rightarrow W_\IC\\
Z &\mapsto r(Z):=en(Z)= (1,-Z,Z^{\#},-N(Z)).\index{$r$, embedding of $\CH$ into $W_\IR\otimes_\IR \IC$}
\end{align}
The action of $G$ on $\mc H$ and the factor of automorphy $j(g, Z):=j_G(g, Z)$\index{$j(g, Z)$}\index{$j_G(g, Z)$} for this action are simultaneously defined via
\begin{align}\label{equ:jdef}
r(Z)g^{-1}=j(g,Z)r(gZ).
\end{align}
Following the conventions of \cite[Section 7.1]{pollack1}, we write $i$ for $i 1_3\in \CH$, and we denote by $K_\infty$\index{$K_\infty$} the stabilizer of $i\in \CH$ for the $G^{1, +}(\IR)$-action on $\CH$.
It follows from the definition of $j(g, Z)$ and $f$ that
\begin{align*}
j(g, Z)=\langle r(Z)g^{-1}, f\rangle.
\end{align*}
As a consequence, we obtain the familiar cocycle relations for the factor of automorphy:
\begin{lem}\label{lemma:automorphyfactor}
For all $g, h\in G^+(\IR)$ and $Z\in \CH$, we have
$j(gh, Z) = j(g, hZ)j(h, Z)$.
\end{lem}
\begin{proof}
\begin{align*}
j(gh, Z) &= \langle r(Z)(gh)^{-1}, f\rangle = \langle r(Z)h^{-1}g^{-1}, f\rangle = \langle j(h, Z)r(hZ)g^{-1}, f\rangle=\\
&=j(h, Z)\langle r(hZ)g^{-1}, f\rangle=j(h, Z)j(g, hZ).
\end{align*}
\end{proof}

In Lemma \ref{twokeyequs}, we record some key facts that will employ later.
\begin{lem}\label{twokeyequs}
For any $\alpha \in G^+(\IR)$ and $Z, W\in\CH$, we have:
\begin{eqnarray}
N(\alpha Z - \alpha W)  & = & \nu(\alpha)^{-1} j(\alpha, z)^{-1}j(\alpha, W)^{-1} N(Z-W)\label{firstdiff}\\
N(\imaginary (\alpha Z)) & = &|j(\alpha, Z)|^{-2}\nu(\alpha)^{-1}N(\imaginary Z)\label{seconddiff}
\end{eqnarray}
\end{lem}
\begin{proof}
For each $Z, W\in \CH$, we have
\begin{align}
\langle r(Z), r(W)\rangle & = \langle en(-Z), en(-W)\rangle\nonumber\\
& = \langle e n(W-Z), e\rangle\nonumber\\
&=\langle r(Z-W), e\rangle.\nonumber\\
& = N(Z-W)\label{Nofdifference}
\end{align}
Recall from Equation \eqref{equ:jdef} that for any $\alpha \in G^+(\IR)$, 
\begin{align*}
r(Z) \alpha^{-1} = j(\alpha, Z) r(\alpha Z).
\end{align*}
So applying Equation \eqref{Nofdifference}, we obtain Equation \eqref{firstdiff}:
\begin{align*}
N(\alpha Z - \alpha W) &= \langle r(\alpha Z), r(\alpha Z)\rangle \\
&= \nu(\alpha)^{-1} j(\alpha, Z)^{-1}j(\alpha, W)^{-1}\langle r(Z), r(W) \rangle\\
& = \nu(\alpha)^{-1} j(\alpha, Z)^{-1}j(\alpha, W)^{-1} N(Z-W).
\end{align*}
Similarly, by Equation \eqref{Nofdifference}, we have
\begin{align*}
\langle r(\bar{Z}), r(Z) \rangle = N(Z-\bar{Z}) = N (2i Y) = 8i N(\imaginary Z).
\end{align*}
So we have
\begin{align*}
N(\imaginary Z) = \frac{1}{8i}\langle r(\bar{Z}), r(Z) \rangle.
\end{align*}
Note that $\overline{r(Z)} = r(\bar{Z})$.
Consequently, applying Equation \eqref{equ:jdef}, we obtain Equation \eqref{seconddiff}.
\end{proof}

\subsubsection{\texorpdfstring{Embedding of Siegel upper half-space in $\CH$}{Embedding of Siegel upper half-space in H}}

We denote by  \index{$\CH_3$}$\CH_3$ the Siegel upper half-space of degree $3$, defined by
\begin{align*}
\CH_3:=\left\{z\in M_3(\IC) | z=x+iy, x, y\in \Sym_3(\IR) \mbox{ and }  y>0\right\},
\end{align*}
where for any ring $R$, \index{$M_3$}$M_3(R)$ denotes $3\times 3$ matrices with entries in $R$ and \index{$\Sym_3$}$\Sym_3(R)$ denotes symmetric matrices with entries in $R$, and $y>0$ means $y$ is positive definite.
The space $\CH_3$ is embedded in $\CH$ via the embedding $\m Q\otimes \m C \rightarrow B \otimes \m C$ induced by the inclusion $\IQ\hookrightarrow B$ coming from the $\IQ$-algebra structure of $B$. 
By \cite[Proposition 6.4]{pollack1}, when restricted to the image of $\mr{GSp}_6$ and $\CH_3$, the action defined by Equation \eqref{equ:jdef} is the same is the same as the the classical one for $\GSp_6$, i.e. $gz= (Az+B)(Cz+D)^{-1}$ for $g =\begin{pmatrix}A& B\\ C&D\end{pmatrix}\in \GSp^+_6(\IR)$ and $z\in \CH_3$.  Furthermore, setting\index{$j_{\GSp_6}(g, z)$} 
\begin{align*}
j_{\GSp_6}(g, z):=\nu(g)^{-2}\det(Cz+D),
\end{align*}
for all $z\in \CH_3$ and $g =\begin{pmatrix}A& B\\ C&D\end{pmatrix}\in \GSp^+_6(\IR)$, we have
\begin{align}\label{equ:jdefn}
j(g,z)=j_G(g, z)=j_{\GSp_6}(g,z).
\end{align}

\begin{rmk}\label{rmk:jconventions}
Although some papers, e.g.\ \cite{AsgariSchmidt}, use $j(g, z)$ to denote $\det(Cz+D)$ in the context of Siegel modular forms, we follow the convention in \cite{pollack1} of including the similitude factor in $j(g, z)$ like above.  This ensures that Equation \eqref{equ:jdefn} holds.
\end{rmk}

\subsubsection{\texorpdfstring{Some useful observations about the action of $G$ on $\CH$}{Some useful observations about the action of G on H}}\label{sec:jwm}
Given a positive integer $M$, we have the following equations that will be useful to us later:
\begin{align*}
(a,b,c,d)w_M= & (-d,M c,-M^2b,M^3a); \\
 (a,b,c,d)w_M^{-1}= &(M^{-3}d,-M^{-2} c,m^{-1}b,-a);\\
w_M Z  =& -M\frac{Z^{\#}}{N(Z)}=-\frac{M}{Z};\\
j(w_M,Z)= &-\frac{N(Z)}{M^3};\\ 
j(\iota_1w_M,Z)=& \frac{Z^\#_{3,3}}{M^2} = \frac{N_2(Z)}{M^2};\\ 
j(\iota_2w_M,Z)= & \frac{Z_{1,1}}{M};\\
j(\iota_3w_M,Z)=& 1.
\end{align*}

\subsection{Normalization of the Archimedean integral}

The main result of \cite{pollack1} is only up to a constant. This constant comes from the Siegel integral on page 1415 of {\it loc. cit.}. We now make it explicit.
\begin{lem}\label{lemma:implicitconstant}
The implicit constant of \cite{pollack1} is
\[
\det(T)^3.
\]
\end{lem}
\begin{rmk}
Since $T$ is a half-integral matrix, $\det T\in \IQ$.
\end{rmk}
\begin{proof}[Proof of Lemma \ref{lemma:implicitconstant}]
We need to calculate the integral on \cite[p.~1415]{pollack1} 
\[
\int_{Y'} N(Y')^{s+r-2} e^{- 2 \pi \mathrm{tr} (Y')} \textup{d}^* Y'
\]
for $\textup{d}^* Y'$ an invariant measure on $\CH$. 
From \cite[Theorem VI.1.1]{AnalysisSymmCone} we have
\[
\int_x e^{-\mathrm{tr}(x)} N(x)^{s} N(x)^{-5} \textup{d} x = (2\pi )^{\frac{15-3}{2}}\Gamma(s)\Gamma(s-2)\Gamma(s-4).
\] 

(The rank in {\it loc. cit.} is $3$, the dimension $n$ is $15=3(6-1)$, and the dimension $d$ is $4$; see the table on page 97. Moreover, $\underline{s}=(s,s,s)$.)

Note that $N(x)^{-5} \textup{d} x$ is a $G$-invariant measure  (and it the same as in \cite[\S 8]{Tsao}).

Writing $x =2 \pi Y' $
\[
(2\pi)^{3s} \int_x e^{-2\pi\mathrm{tr}(Y)} N(Y')^{s} N(x)^{-5} (2\pi)^{-15} (2\pi)^{15}  \textup{d} x = (2\pi )^{\frac{15-3}{2}}\Gamma(s)\Gamma(s-2)\Gamma(s-4).
\]

Note that $\Gamma_{\mathbb{R}}(2s)=\pi^{-s}\Gamma(s)$, so to determine the constant is enough to understand the difference between the two measures.
Now, the measure $\textup{d} x $ is the one induced by the usual Euclidean measure on 
\[
B_0 \otimes_{\mathbb{Z}} \mathbb{R} \cong  \mathbb{R} \mb{1}\oplus \mathbb{R} \mb{i} \oplus \mathbb{R} \mb{j} \oplus \mathbb{R}\mb{k}. 
\]
The Archimedean measure of $B_0 \setminus B_0 \otimes_{\mathbb{Z}} \mathbb{R}$ is then the covolume of $B_0$ which by \cite[Lemma 7.6]{pollack1}  is $\det(T)=4^{-1}D_B$. 

We are left to specify what is the measure $\textup{d}^* Y'$ of \cite[p.~1415]{pollack1}; it is the one $\CH$ induced by the one on $B(\adeles)$ of \cite[Lemma 7.6]{pollack1} for which 
\[
\mu(B_0 \setminus B_0 \otimes_{\mathbb{Z}} \mathbb{R})=1.
\]
Hence, to calculate the covolume of $H_3(B_0)$ in $H_3(B_0 \otimes_{\mathbb{Z}} \mathbb{R})$ we just multiply three times the previous calculation (as we have exactly three linearly independent copies of $B_0$ in $H_3(B_0)$) and get 
\[
N(x)^{-5} \textup{d} x = \det(T)^3 \textup{d}^* Y'.
\]
\end{proof}

\section{\texorpdfstring{Modular forms on $G$ and $\mathrm{GSp}_6$}{Modular forms on G and GSp6}}\label{sec:mforms}
We will work with Siegel modular forms of degree $3$ (i.e.\ on $\GSp_6$) and with modular forms on $G$.  We review key facts about Siegel modular forms in Section \ref{section:Siegelforms}.  We introduce terminology for modular forms on $G$, which is not explicitly provided in prior literature, in Section \ref{sec:modularformsG}. 
\subsection{Siegel modular forms}\label{section:Siegelforms}
For a detailed treatment of Siegel modular forms, see, e.g., \cite{KlingenSiegel} or \cite[Section 4]{AsgariSchmidt}.
In this paper, we work with degree $3$ (i.e.\ genus $3$) Siegel modular forms of positive even scalar weight $2r$ (with $r$ a positive integer) and level $\Gamma$ for $\Gamma$ a congruence subgroup of $\GSp_6(\ZZ)$, i.e.\ holomorphic functions
\begin{align*}
f:\CH_3\rightarrow\IC
\end{align*}
such that for all $\gamma\in \Gamma$ and $z\in \CH_3$
\begin{align*}
\left(f|_{2r} \gamma\right)(z) = f(z),
\end{align*}
where 
\begin{align*}
\left(f|_{2r} g\right)(z):=\nu(g)^{-r} j(g, z)^{-2r}f(\gamma z),
\end{align*}
for all $g\in \GSp_6(\IR)$ and $z\in \CH_3$.
Siegel modular forms $f$ have Fourier expansions of the form
\begin{align*}
f(z) = \sum_{T \geq 0}a_f(T)e^{2\pi i\tr(Tz)},
\end{align*}
with $a_f(T)$ complex numbers supported on nonnegative definite, half-integral symmetric matrices $T$.  Given a Siegel modular form  $f$ of weight $2r$ on $\GSp_6$, we denote by $\varphi_f$\index{$\varphi_f$} its lift to $\mr{GSp}_{6}^+(\m R)$, i.e.
\begin{align*}
\varphi_f(g_\infty)  := \nu(g_\infty)^{-r}j(g_\infty, i)^{-2r}f(g_\infty i).
\end{align*}
Then 
\begin{align}\label{equ:mfonGSp}
\varphi_f(\gamma g k_\infty) = j(k_\infty, i)^{-2r}\varphi_f(g)
\end{align}
 for all $k_\infty\in K_\infty$ and $g\in \GSp_6^+(\IR)$.  So modular forms on $\GSp_6$ can be identified with the subspace of the functions on $\GSp_6^+(\IR)$ satisfying Equation \eqref{equ:mfonGSp} meeting a holomorphy condition and a moderate growth condition.

Later, we will work with the following congruence subgroups of  $\mr{GSp}_{6}(\m Z)$:
\begin{align*}
\Gamma^0(M):=&\left\{\gamma = \left(\begin{array}{cc}
A & B \\
C & D
\end{array} \right) \equiv \left( \begin{array}{cc}
\ast & 0 \\
\ast & \ast
\end{array} \right) \bmod M \right\};
\\
\Gamma^1(M):=&\set{\gamma = \left(\begin{array}{cc}
A & B \\
C & D
\end{array} \right) \equiv \left( \begin{array}{cc}
1_3 & 0 \\
\ast & 1_3
\end{array} \right) \bmod M}.
\end{align*}
 Note that the image of $\Gamma^0(M)$ in $G(\m Z)$ is contained in  $\Gamma_G^0(M)$.  Given a Dirichlet character $\varphi$\index{$\varphi$} modulo $M$, we view it as a character of $\Gamma^0(M)$ via $\varphi(\gamma)=\varphi(A)$.
 \begin{rmk}
 We work with lower, rather than upper, triangular subgroups in order to be consistent with the conventions in \cite{pollack1}.
 \end{rmk}
We will also consider Siegel modular forms of weight $2r$ and Nebetypus $\varphi$, i.e. Siegel modular forms $f$ for $\Gamma^1(M)$ such that
\begin{align*}
f(z) =\varphi(\gamma)\nu(g)^{-r} j(\gamma, z)^{-2r}f(\gamma z),
\end{align*}
for all $\gamma\in \Gamma^0(M)$ and $z\in \CH_3$.
\begin{rmk}
Note that our definition of Siegel modular form agrees with the one in \cite[Section 4.1]{AsgariSchmidt}, taking into account the difference in conventions for the notation $j$ described in Remark \ref{rmk:jconventions}.
\end{rmk}

\subsubsection{Cuspidal automorphic representation associated to a Siegel modular form}
Similarly to \cite[Section 1.1]{pollack1} and \cite[Section 4.1]{AsgariSchmidt}, we consider cuspidal automorphic representations associated to Siegel modular forms as follows.  Following the conventions of \cite[Section 1.1]{pollack1}, we denote the stabilizer of $i\in \CH_3$ in $\Sp_6(\IR)$ by $K_{\infty, \Sp_6}$.  So $K_{\infty, \Sp_6}$ is isomorphic to the definite unitary group $U(3)$.  We also let $K$ be an open compact subgroup of $\GSp_6(\adeles_f)$.
Similarly to \cite[Definition 1.1]{pollack1} (but, unlike \cite{pollack1}, allowing not only level $1$ but also other levels), we say that a cuspidal automorphic representation $\pi$ of $\GSp_6(\adeles)$ {\em is associated} to a level $\Gamma$ Siegel modular form of weight $2r$ if all of the following conditions are met:
\begin{itemize}
\item{There is a nonzero element $\phi\in \pi$ such that $\phi(gk_fk_\infty) = j(k_\infty, i)^{-2r}\phi(g)$ for all $g\in \GSp_6(\adeles)$, $k_f\in K$, and $k_\infty\in K_{\infty, \Sp_6}$.}
\item{The representation $\pi$ has trivial central character.}
\item{The function $f_\phi: \CH_3\rightarrow\IC$\index{$f_\phi$} defined by
\begin{align*}
f_\phi(g_\infty i) = \nu(g_\infty)^{r}j(g_\infty, i)^{2r}\phi(g_\infty)
\end{align*}
for all $g_\infty\in \GSp_6^+(\IR)$ is a Siegel modular form of weight $2r$ and level $\Gamma$.  W}
\end{itemize}
As noted in \cite[Section 1.1]{pollack1}, the final condition ensures that $f_\phi$ is cuspidal, i.e.\ its Fourier expansion is of the form
\begin{align*}
f_\phi(z)= \sum_{T> 0}a_{\phi}(T)e^{2\pi i\tr(Tz)}.
\end{align*}
We shall also refer to $f_\phi$ as the {\em Fourier expansion of} $\phi$, and similarly, we refer to the Fourier coefficients $a_{\phi}$ of $f_\phi$ as the {\em Fourier coefficients of} $\phi$.  When we need to specify the form to which $\pi$ is associated, we say that $\pi$ {\em is attached to} $\phi$.

Given a Siegel modular form $f$, we define the {\em adelisation} of $f$ by
\begin{align*}
\phi_f: \GSp_6(\IQ)\backslash\GSp_6(\adeles)/(KZ)\rightarrow \IC,\index{$\phi_f$}
\end{align*}
where $Z$\index{$Z$} denotes the center, by
\begin{align*}
\phi_f(g):=\left(f|_{2r} g_\infty\right)(i) = \phi_f(g_\infty),
\end{align*}
for $g= g_\IQ g_\infty k$ with $g_\IQ\in\GSp_6(\IQ)$, $g_\infty \in \GSp^+_6(\IR)$, and $k\in K$.
As noted in \cite[Section 4.1]{AsgariSchmidt} (which denotes by $\Phi_f$ the function we denote by $\phi_f$), the map $f\mapsto\phi_f$ gives an injection from the space of Siegel modular forms of weight $2r$ to the space of functions $\phi$ on $\GSp_6(\IQ)\backslash\GSp_6(\adeles)/(KZ)$ such that $\phi(gk_\infty) = j(k_\infty, i)^{-2r}\phi(g)$ for all $k_\infty\in K_\infty$.

\subsubsection{Petersson inner product for $\GSp_6$}\label{sec:petersson}
Let $f$ and $g$ be Siegel modular forms on $\GSp_6^+(\IR)$ of weight $2r$ (with $r$ a positive integer, as above) and level $\Gamma$, and suppose at least one of $f$ or $g$ is cuspidal.  The Petersson inner product of $f$ and $g$ is
\begin{align}\label{equ:Peterssondefn}
\langle f, g\rangle:=\int_{\Gamma \setminus \CH_3} f(z)\overline{g(z)} \mr{det}(\imaginary(z))^k d^*z,
\end{align}
where $d^*z$ denotes the $\GSp_6$-invariant volume form on $\CH_3$, i.e. $d^*z = \frac{dx dy}{\det \imaginary(z)^4}$ with $z=x+iy\in \CH_3$ and $y=\imaginary(z)$.   Note that Equation \eqref{equ:Peterssondefn} depends on $\Gamma$.  
If the forms are of a smaller level $\Gamma'$, the Petersson products will differ by the index of $\Gamma'$ in $\Gamma$.  While some references address this dependence by normalizing the integral by a volume factor, we follow the convention of \cite{hi85, BS, AsgariSchmidt} of omitting the volume factor.
Like in \cite[Equation (25)]{AsgariSchmidt}, we can reformulate the Petersson inner product adelically and obtain
\begin{align*}
\langle f, g\rangle=\langle \phi_f, \phi_g\rangle:=\int_{Z(\adeles)\GSp_6(\IQ)\backslash \GSp_6(\adeles)}\phi_f(h)\overline{\phi_g(h)} dh,
\end{align*}
where $dh$ denotes the appropriately normalized Haar measure on $Z(\adeles)\GSp_6(\IQ)\backslash \GSp_6(\adeles)$ that makes the equality hold (it depends on $\Gamma$).

The group $\Gal(\IC/\IQ)$ acts on the space of Siegel modular forms through its action on Fourier coefficients.  Following the conventions of \cite{GarrettEquivariance}, given a Siegel modular form $f$ and an automorphism $\sigma\in\Gal(\IC/\IQ)$, we denote by $f^\sigma$\index{$f^\sigma$} the form obtained by applying $\sigma$ to the Fourier coefficients of $f$, and we write $f^\natural$\index{$\natural$} for the form obtained via complex conjugation of the Fourier coefficients of $f$.  If $\phi =\phi_f$ is the adelisation of $f$, then we set $\phi^\sigma := \phi_{f^\sigma}$.\index{$\phi^\sigma$} If $\pi$ is a cuspidal automorphic representation associated to $\phi$, then we denote by $\pi^\sigma$\index{$\pi^\sigma$} the representation associated to $\phi^\sigma$.

In analogue with the equivariance property for Petersson products for modular forms in \cite[Lemma 4]{shimura-RS}, Proposition \ref{prop:equivariantPetersson} gives equivariance properties for ratios of Petersson products of Siegel modular forms.

\begin{prop}\label{prop:equivariantPetersson}
Let $f$ be a cusp form on $\GSp_6$ of even weight $k> 7$ that is a Hecke eigenfunction at all but at most finitely many primes.  Then for all weight $k$ Siegel modular forms $g$ on $\GSp_6$,
\begin{align}\label{equ:equivariantPetersson}
\frac{\langle g, f^\natural\rangle}{\langle f, f^\natural \rangle}\in \IQ(g, f),
\end{align}
where $\IQ(g, f)$ denotes the field generated by the Fourier coefficients of $g$ and $f$.
If, furthermore, there is a constant $c$ such that all the Fourier coefficients of $cf$ lie in a CM extension of $\IQ$, then we also have
\begin{align}\label{equ:equivariantPetersson2}
\frac{\langle g, f\rangle}{\langle f, f\rangle}\in \IQ(g, f).
\end{align}
\end{prop}

\begin{proof}
By \cite[Theorem 27.14]{shar}, we can decompose $g$ uniquely as 
\[
g= g^{\mr{cusp}}+g^{\mr{Eis}},
\]
where $g^{\mr{cusp}}$ is a cusp form and $g^{\mr{Eis}}$ is orthogonal to all the cusp forms. Hence $\langle g, f\rangle=\langle g^{\mr{cusp}}, f\rangle$.  As noted in \cite[Proof of Theorem 9]{bouganis}, this decomposition is stable under $\Gal(\IC/\IQ)$, i.e. $g^\sigma= \left(g^{\mr{cusp}}\right)^\sigma+\left(g^{\mr{Eis}}\right)^\sigma$ with $\left(g^{\mr{cusp}}\right)^\sigma$ a cusp form and $\left(g^{\mr{Eis}}\right)^\sigma$ orthogonal to all the cusp forms for all $\sigma\in\Gal(\IC/\IQ)$.  (Shimura only proved in \cite[Theorem 27.14]{shar} that this decomposition is stable under the action of $\Gal(\IC/\closure{\IQ})$.  The stronger decomposition follows from Harris's later work on Eisenstein series \cite[Lemma 3.2.1.2]{harrisannals}, together with the stability of the space of cusp forms under this larger group.) So the Fourier coefficients of $g^{\mr{cusp}}$ lie in the same field as those of $g$.  Equation \eqref{equ:equivariantPetersson} now follows from the main theorem in \cite{GarrettEquivariance}, which says that for all $\sigma\in\Gal(\IC/\IQ)$,
\begin{align}\label{equ:GarrettEquivariance}
\sigma\left(\langle h, f^\natural\rangle/\langle f, f^\natural\rangle\right) = \langle h^\sigma, \left(f^\sigma\right)^\natural\rangle/\langle f^\sigma, \left(f^\sigma\right)^\natural\rangle
\end{align}
for all cusp forms $h$ of weight $k>7$.  If, furthermore, there is a constant $c$ such all the Fourier coefficients of $cf$ lie in CM extensions of $\IQ$, then Equation \eqref{equ:equivariantPetersson2} holds by \cite[Corollary 3.19]{Saha}, which says that for all $\sigma\in\Gal(\IC/\IQ)$,
\begin{align}\label{equ:SahaEquivariance}
\sigma\left(\langle h, f\rangle/\langle f, f\rangle\right) = \langle h^\sigma, f^\sigma\rangle/\langle f^\sigma, f^\sigma\rangle,
\end{align}
which also implies Equation \eqref{equ:GarrettEquivariance}.  While Equation \eqref{equ:GarrettEquivariance} is proved in \cite{GarrettEquivariance} for $\Sp_{2n}$, Equation \eqref{equ:SahaEquivariance} is proved in \cite[Corollary 3.19]{Saha} for $\GSp_{2n}$.  How to move between these two groups is explained in \cite{HPSS}.
\end{proof}

\subsubsection{Nearly holomorphic Siegel modular forms}\label{sec:nearlyholoGsp}
Similarly to many constructions of $L$-functions, the construction employed in this paper requires that we expand our focus to nearly holomorphic Siegel modular forms.  We recall the key facts needed for this paper.  A more detailed treatment of nearly holomorphic forms is presented in, e.g., \cite[\S 13]{shar}.

Let $r$ be a nonnegative integer, and let $\Gamma$ be a congruence subgroup of $\GSp_6(\ZZ)$.
A {\em nearly holomorphic Siegel modular form} of weight $2r$ and level $\Gamma$ on $\GSp_6$ is a real-analytic function
\begin{align*}
f:\CH_3\rightarrow \IC
\end{align*}
such that for all $\gamma\in \Gamma$ and $z\in \CH_3$
\begin{align*}
f(z) =\nu(g)^{-r} j(\gamma, z)^{-2r}f(\gamma z),
\end{align*}
and such that $f$ can be written as a polynomial in the entries of ${\mr{Im}(z)}^{-1}$ with coefficients being holomorphic functions on $\CH_3$.  By a slight abuse of notation, we will denote the entries of ${\mr{Im}(z)}^{-1}$ by $Y_{i,j}^{-1}$ or $1/Y_{i,j}$, but we warn the reader that the entries of ${\mr{Im}(z)}^{-1}$ are not inverses of the entries of ${\mr{Im}(z)}$.  Every nearly holomorphic Siegel modular form $f$ of level $\Gamma^0(M)$ admits a polynomial $q$-expansion of the form
\begin{align*}
f = \sum_{T \geq 0}a_f(T)e^{2\pi i\tr(Tz)},
\end{align*}
with $a_f(T)$ are polynomials in the variables $Y_{i,j}^{-1}$, supported on positive definite, half-integral symmetric matrices $T$.  We say that a nearly holomorphic Siegel modular form $f$ is {\em defined over} a $\mathbb{Z}[1/M]$-algebra $R$ if the coefficients of all $a_f(T)$ are polynomials in the variables $Y_{i,j}^{-1}$ with coefficients in $R$.

\subsubsection{Holomorphic projection}  
The domain for the Petersson inner product introduced in Section \ref{sec:petersson} can be expanded to include nearly holomorphic forms.  In this context, we recall key facts about {\em holomorphic projection}.  For more details about holomorphic projection, see, e.g., \cite[\S 15]{shar}, \cite{sturm-critical}, \cite[\S 3]{shimura-half}, \cite[Section 2.4]{CouPan}, or \cite[Section 3.7]{ZLiuFour}.

The {\em holomorphic projection} $\holoprojop$ of a nearly holomorphic form $g$ is the holomorphic Siegel cusp form $Hg$\index{$H$, holomorphic projection} such that 
\begin{align}\label{equ:holoproj}
\langle f,g \rangle =\langle f,Hg \rangle
\end{align}
for all holomorphic cusp forms $f$ of weight $r$ and level $\Gamma$.
In order for $Hg$ to be unique, we need to require $Hg$ to be a cusp form.  Otherwise, we could add any Eisenstein series $E$ to $Hg$ and still get
\[
\langle f,g \rangle =\langle f,Hg +E \rangle.
\]

If $\set{f_i}$ is a basis of the space of  holomorphic Siegel cusp forms of weight $2r$ and level $\Gamma$ which is orthonormal for the Petersson product, then 
\[
Hg=\sum_i \overline{\langle f_i,g \rangle} f_i.
\]
There is an explicit recipe for the Fourier coefficients of $Hg$ \cite[Theorem 2.16]{CouPan}.  Specializing \cite[3.7]{ZLiuFour} to the case $\mathcal{U}=\set{2r}$, we have
\begin{lem}\label{lem:holoprojcoeffs}
If $g$ is a nearly holomorphic Siegel modular form defined over a $\mathbb{Q}$-algebra $R$, then the $q$-expansion coefficients of $Hg$ lie in $R$.
\end{lem}
  It is known that the same statement does not apply to $\mathbb{Z}$-algebras, because holomorphic projection introduces denominators.

\subsection{Modular forms on the group $G$}\label{sec:modularformsG}
Although modular forms on $G$ are not explicitly defined in the literature, it will be convenient for us to introduce this terminology.  For a more detailed introduction to modular forms in related settings, the reader might consult the survey article \cite[Section 6]{pollackNotices} or \cite[Section 3.4]{EischenAWS}.  For $g\in G^+(\IR)$ and functions $f:\CH\rightarrow\IC$, we define
\begin{align*}
(f|_{2r}g)(Z) = \nu(g)^{-r}j(g, Z)^{-2r}f(gZ)
\end{align*}
\begin{defi}\label{def:mformsonH}
Let $r$ be a nonnegative integer, and let $\Gamma$ be a congruence subgroup of $G(\ZZ)$.  A {\em modular form on $G$} of weight $2r$ and level $\Gamma$ is a holomorphic function
\begin{align*}
f: \CH\rightarrow \IC
\end{align*}
such that 
\begin{align}\label{equ:modularG}
(f|_{2r}\gamma)(Z)=f(Z)
\end{align}
 for all $Z\in \CH$ and $\gamma\in \Gamma$.  Following the conventions of \cite[Section 6]{pollackNotices}, we also require that the function $|\det(\imaginary(Z))^rf(Z)|$ is of moderate growth on $G(\IR)$.
\end{defi}
When a function $f$ on $G$ satisfies Equation \eqref{equ:modularG}, we call $f$ {\em modular} (even if $f$ is not holomorphic).  The modular forms on $G$ arising in this paper are Eisenstein series, which we introduce in Section \ref{sec:Eseries}.  It will be useful to reformulate Definition \ref{def:mformsonH} in terms of functions on $G$.  Similarly to the case of Siegel modular forms, we associate to each modular form $f$ on $G$ a function 
\begin{align*}
\varphi_f: G^+(\IR)\rightarrow \IC
\end{align*}
by
\begin{align*}
\varphi_f(g) = \nu(g)^{-r}j(g, i)^{-2r}f(gi).
\end{align*}
Then 
\begin{align}\label{equ:mfonG}
\varphi_f(\gamma g k_\infty) = j(k_\infty, i)^{-2r}\varphi_f(g)
\end{align}
 for all $k_\infty\in K_\infty$ and $g\in G^+(\IR)$.  So modular forms on $G$ can be identified with the subspace of the functions on $G^+(\IR)$ satisfying Equation \eqref{equ:mfonG} meeting a holomorphy condition and a moderate growth condition.
Similarly to the case of classical modular forms, it is also sometimes convenient to view modular forms on $G$ as functions on $G(\adeles)$.  This is the viewpoint taken in \cite[Section 7]{pollack1} (although without using the terminology {\em modular form}), which we sometimes employ in the discussion of Eisenstein series in Section \ref{sec:Eseries}.  In particular, given a modular form $f$ on $G$ and a compact open subgroup $K\subset G(\adeles_f)$, we define
\begin{align*}
\phi_f: G(\IQ)\backslash G(\IR)/Z\left(G\left(\adeles\right)\right)K\rightarrow \IC,
\end{align*}
with
\begin{align*}
\phi_f(g):=\left(f|_{2r} g_\infty\right)(i) = \varphi_f(g_\infty),
\end{align*}
for $g= g_\IQ g_\infty k$ with $g_\IQ\in G(\IQ)$, $g_\infty \in G^+_6(\IR)$, and $k\in K$.
Similarly to the case of Siegel modular forms discussed above, the map $f\mapsto\phi_f$ gives an injection from the space of Siegel modular forms of weight $2r$ to the space of functions $\phi$ on $G(\IQ)\backslash G(\IR)/Z\left(G\left(\adeles\right)\right)K$ such that $\phi(gk_\infty) = j(k_\infty, i)^{-2r}\phi(g)$ for all $k_\infty\in K_\infty$.

\begin{rmk}
By the properties given in Section \ref{sec:Hss}, we immediately have that the restriction to $\GSp_6$ of each modular form of weight $2r$ on $G$ is a Siegel modular form of degree $3$ and weight $2r$.  Similarly, if a function on $G$ satisfies the modularity property but is not necessarily holomorphic, the same is true of its restriction to $\GSp_6$.  Later, in Section \ref{sec:restrictionfromG}, we will see that if a modular form on $G$ has a Fourier expansion with coefficients in a ring $R$, then the same is true of its restriction to $\GSp_6$.  While we do not need a complete theory of nearly holomorphic modular forms on $G$ for this paper, we note that it also similarly follows from our work in Section \ref{sec:restrictionfromG} that if a function $f$ on $G$ satisfies the modularity property and can be expressed analogously to the nearly holomorphic forms defined over a ring $R$ in Section \ref{sec:nearlyholoGsp}, then the restriction of $f$ to $\GSp_6$ is a nearly holomorphic modular form on $\GSp_6$ defined over $R$.
\end{rmk}

\section{\texorpdfstring{Eisenstein series on $G$}{Eisenstein series on G}}\label{sec:Eseries}
We introduce certain Eisenstein series on the group $G$ that was defined in Section \ref{sec:gpG}.  In Section \ref{sec:EisDef}, we define the Eisenstein series with which we work.  In Section \ref{sec:FexpnG}, we compute their Fourier expansions.  In Section \ref{sec:restrictionfromG}, we describe the Fourier coefficients of the restriction to $\GSp_6$ of a modular form on $G$. In Section \ref{sec:diffop}, we summarize the effect of weight-raising differential operators.

\subsection{\texorpdfstring{Definition of certain Eisenstein series on $G$}{Definition of certain Eisenstein series on G}}\label{sec:EisDef}
In this section, we define Eisenstein series $E_{2r, \chi}$
that are one of the key ingredients in the the construction of the $L$-functions in this paper.  After some necessary setup, we define $E_{2r,\chi}$ precisely in Equation \eqref{equ:Eseriesdefn}.

\subsubsection{\texorpdfstring{Character $\chi$ of the parabolic subgroup $P$}{Character chi of the parabolic subgroup P}}\label{sec:charchiofP}

Let $\chi$\index{$\chi$} be a Dirichlet character  modulo the integer $M$.  Recall the element $f=(0, 0, 0, 1)\in W$.  We view $\chi$ as a character of the parabolic subgroup $P$ via $\chi(\gamma)=\chi(ad)$, where $d$, resp. $a$, is the scalar which gives the action of each element $\gamma\in P$ on $\IQ f$, resp $W/\left(f\right)^{\perp}$.  A straightforward calculation shows that $\nu(\gamma)=ad$ for each $\gamma\in P$ and that, furthermore, $\chi(\gamma)=\chi(a^2)$ for each element $\gamma$ in the center of $P$.  We will sometimes consider $\chi$ adelically and write it as  the restricted tensor product $\otimes' \chi_v$, for $v$ ranging over all the places of $\m Q$.

\subsubsection{\texorpdfstring{Eisenstein series as functions of $G(\adeles)$}{Eisenstein series as functions of G(A)}}\label{sec:AdelicEisensteinseries}
In this section, we introduce the Eisenstein series $E_{2r, \chi}^\ast$ that appears in Equation \eqref{equ:I2r}.  These are similar to the Eisenstein series $E_{2r}^\ast$ introduced in \cite{pollack1}, except that we also allow a character $\chi$ of $P$ as introduced in Section \ref{sec:charchiofP}, allow level $\Gamma^0(M)$ (instead of just level $1$ as in \cite{pollack1}),  and normalize the series slightly differently. 
From now on, we suppose that $\chi^2=1$; this ensure that the (restriction to $\mathrm{GSp}_6$ of the) Eisenstein series descends to $\mathrm{PGSp}_6$. These Eisenstein series will be defined in terms of particular induced sections $f_v(\gamma_v,s)$\index{$f_v(\gamma_v,s)$} in
$
\mr{Ind}_{P(\m Q_v)}^{G(\m Q_v)}(\chi_v(da)\vert da \vert^s)$:
\begin{itemize}\label{equ:fvintegral}
\item{For $v \nmid M\infty$, we take $f_v(\gamma_v,s)$ to be the unique unramified section with support on the maximal compact subgroup $G(\m Z_v)$.
Explicitly, we can write
\begin{align}
f_v(\gamma_v,s)= \vert \nu(\gamma_v) \vert^s \chi_v(\nu(\gamma_v)) \int_{\m Q_v^{\times} } \Phi_v(tf\gamma_v)\vert t \vert^{2s} \chi(t)^{2} \textup d t, 
\end{align}
for \index{$\Phi_v$}$\Phi_v$ as in \cite[\S 5]{pollack1}.  A calculation similar to the one in \cite[\S 7.2]{pollack1}, but taking into account the difference in normalization in \cite[p. 1407 and first paragraph of \S7.3]{pollack1}, shows that on the maximal compact $G(\m Z_v)$, this section takes  the value
\[
\frac{1}{(1-\chi(v)v^{-2s})}.
\]
}
\item{For $  v\mid M$, we define $f_v(\gamma_v,s)$ as in Equation \eqref{equ:fvintegral}, except that in this case we take $\Phi_v$ to be the characteristic function of $\m Z_v^{\times} f + v^{\mr{val}_v(M)} W \otimes_{\m Z}\m Z_v$, where $\mr{val}_v$\index{$\mr{val}_v$} denotes the usual $p$-adic valuation such that $\mr{val}_v(p)=1$ for each finite place $v=p$.
This implies that the section is supported only on elements of $G(\m Z_v)$ that reduce to elements of $P(\m Z / v^{\mr{val}_v(M)} \m Z)$.}
\item{For $v= \infty$, we take $f_{\infty,2r}$\index{$f_{\infty,2r}$} to be the unique section satisfies $f_{\infty,2r}(k, s) =j(k,i)^{2r}$ for all $k$ in the maximal compact $K_\infty$ of $G(\m R)$, similarly to \cite[first paragraph of \S 7.3]{pollack1}.}
\end{itemize}

Similarly to \cite[\S 7.3]{pollack1}, we let 
\begin{align*}\index{$J(\gamma,Z)$}
J(\gamma,Z):=\nu(\gamma)j(\gamma,Z),
\end{align*}
and we define the Eisenstein series $E^*_{2r,\chi}(g, s)$, for $g\in G$ and $s\in \IC$ with $\mathrm{Re}(s)>\!\!>0$, by
\begin{align*}\index{$E^*_{2r,\chi}$}
\frac{E^*_{2r,\chi}(g,s)}{D_{B}^s C_{\infty}}=& L^{(D_{B})}(\chi,2s-2)L(\chi,2s-4) \sum_{\gamma \in P(\m Q)\setminus G(\m Q)} {\vert \nu(\gamma g)\vert}^{s} \frac{{J(\gamma g,i)}^{2r}}{\vert {J(\gamma g,i)}\vert^{2s+2r}} \prod_{v \nmid \infty} f_v(\gamma g,s) \\
 =  &\sum_{\gamma \in  P(\m Q)\setminus G(\m Q)} {\nu(\gamma g)}^{r-s/2} {\overline{\nu(\gamma g)}}^{-r-s/2} {j(\gamma g,i)}^{r-s}{\overline{j(\gamma g,i)}}^{-s-r} \prod_{v \nmid \infty }f_v(\gamma g,s), 
\end{align*} 
where (applying the formula between \cite[Equations (7.4) and (7.5)]{pollack1})
\[ 
C_\infty=C_{\infty}(s,r) = \pi^{-(3s+3r-6)}\Gamma(s+r)\Gamma(s+r-2)\Gamma(s+r-4)\index{$C_\infty(s,r)$}
\]
and  $L^{(D_{B})}(\chi,s)$ denotes the partial $L$-function with factors removed at primes dividing $D_B$.

\begin{rmk}
Although we do not include Nebentypus here, the discussion could be amended to include $\Gamma_0$-Nebetypus $\varphi$ by replacing $\chi$ by the character $\gamma \mapsto \chi(ad)\varphi(ad^2)$.
\end{rmk}

\subsubsection{\texorpdfstring{Eisenstein series as functions of $Z$}{Eisenstein series as functions of Z}}\label{sec:EseriesZ}
Given $Z=X+iY\in \CH$, we write $Z=g_{\infty} i$ with $g_{\infty} \in G^1(\m R)$.  Applying Lemma \ref{twokeyequs}\eqref{seconddiff} to $i\in \CH$ and $g_\infty \in G^1(\m R)$, we obtain
\begin{align*}
\vert j(g_\infty,i) \vert^2 = N (Y)^{-1}.
\end{align*}
We transform the Eisenstein series  $E^*_{2r,\chi}$ into a function of the  complex variable $Z$.
\begin{lem}
With $\Gamma_\infty$ defined as in Equation \eqref{equ:Gammainfty} and $g_\infty$ the Archimedean component of $g\in G(\adeles)$, we have
\begin{align*}
{E^*_{2r,\chi}(g,s)}= E^*_{2r,\chi}(Z,s),
\end{align*}
where
\begin{align*}\index{$E^*_{2r,\chi}(Z,s)$}
E^*_{2r,\chi}(Z,s): = & {D_{B}^s C_{\infty}} L(\chi,2s)L^{(D_{B})}(\chi,2s-2)L(\chi,2s-4)\times \\
   & \times \sum_{\gamma \in  \Gamma_{\infty} \setminus \Gamma_G^0(M)} \chi^{-1}(\gamma) {\nu(\gamma)}^{-s} {\overline{j(\gamma, Z)}}^{-2r} \frac{N(Y)^{s-r}}{\vert {j(\gamma ,Z)}\vert^{2s-2r}}.
\end{align*}
\end{lem}
\begin{proof}
A calculation similar to the one in \cite[\S 3.1.7]{Urb} gives
\begin{align*}
\frac{E^*_{2r,\chi}(g,s)}{D_{B}^s C_{\infty}}=&  \overline{J(g_{\infty},i)}^{2r} \overline{\nu(g_{\infty})}^{-r} \sum_{\gamma \in P(\m Q)\setminus G(\m Q)} {\vert \nu(\gamma g_{\infty})\vert}^{s} \frac{{J(\gamma g_\infty,i)}^{2r}}{\vert {J(\gamma g_\infty,i)}\vert^{2s+2r}}\prod_{v \nmid \infty} f_v(\gamma ,s) \\
         =& \sum_{\gamma \in P(\m Q)\setminus G(\m Q)}  \vert {J(g_\infty,i)}\vert^{2r-2s}  \frac{ \nu(\gamma)^s {J(\gamma ,Z)}^{2r} }{ \vert {J(\gamma ,Z)}\vert^{2s+2r}} \prod_{v \nmid \infty} f_v(\gamma,s)\\
   =&L(\chi,2s)L^{(D_{B})}(\chi,2s-2)L(\chi,2s-4) \times\\ & \times \sum_{\gamma \in  \Gamma_{\infty} \setminus \Gamma_G^0(M)} \chi^{-1}(\gamma) {\nu(\gamma)}^{-s} {j(\gamma,Z)}^{r-s} {\overline{j(\gamma, Z)}}^{-s-r} N(Y)^{s-r}\\
   = &L(\chi,2s)L^{(D_{B})}(\chi,2s-2)L(\chi,2s-4)\times \\
   & \times \sum_{\gamma \in  \Gamma_{\infty} \setminus \Gamma_G^0(M)} \chi^{-1}(\gamma) {\nu(\gamma)}^{-s} {\overline{j(\gamma, Z)}}^{-2r} \frac{N(Y)^{s-r}}{\vert {j(\gamma ,Z)}\vert^{2s-2r}}.
\end{align*}
To change the index of summation, we use \cite[Lemma 7.3]{Tsao}, as well as the fact that $f_v$ is supported in the big cell at bad places $v$. 
\end{proof}
The function $E^*_{2r,\chi}(Z,s)$ is modular
of weight $2r$ and level $\Gamma_G^0(M)$ as a function of $Z$, and it has central character $\chi$.  When $s=r$, the form $E^*_{2r,\chi}(Z,r)$ is antiholomorphic as a function of $Z$.
We will especially work with the Eisenstein series $E_{2r,\chi}(Z,s)$ defined by
\begin{align}\label{equ:Eseriesdefn}\index{$E_{2r,\chi}$}
E_{2r,\chi}(Z,s) = & \overline{E^*_{2r,\chi}(Z,s)},
\end{align}
as well as the holomorphic Eisenstein series
\begin{align}\label{equ:holoEseries}
E_{2r, \chi}(Z)&:=E_{2r, \chi}\left(Z, r\right)\nonumber\\
&={D_{B}^s C_{\infty}(r,r)} L(\chi,2r)L^{(D_{B})}(\chi,2r-2)L(\chi,2r-4) \sum_{\gamma \in  \Gamma_{\infty} \setminus \Gamma_G^0(M)} \chi(\gamma)   \frac{\nu(\gamma)^{-r} }{{j(\gamma, Z)}^{2r}}.
\end{align}

\begin{rmk}\label{rmk:introducediffop}
In \cite[Section 7.3]{pollack1}, Pollack introduces a weight-raising differential operator $\mc D$\index{$\mc D$} that acts on modular forms on $G$. 
By  Theorem 7.7 of {\it loc. cit.} applied in the case $s=r$, 
\begin{align}\label{eq:MaassShimura}
\mc D^t  E^*_{2r,\chi}(Z,r) &= E^*_{2r+2t,\chi}(Z,r)\nonumber\\
\mc D^t  E_{2r,\chi}(Z,r) &= E_{2r+2t,\chi}(Z,r)
  \end{align}
  for all nonnegative integers $t$.  (Here, $\mc D^t$ denotes $\mc D$ applied $t$ times.)
Sections \ref{sec:FexpnG} and \ref{sec:restrictionfromG} establish the algebraicity of the Fourier coefficients of the holomorphic Eisenstein series $E_{2r,\chi}(Z)$ and of its restriction to $\GSp_6$, respectively. Section \ref{sec:diffop} then discusses algebraic properties of the restriction to $\GSp_6$ of the forms in the image of $\mc D^t$.
\end{rmk}

\subsection{\texorpdfstring{Main Results for Fourier expansions of Eisenstein series on the group $G$}{Main Results for Fourier expansions of Eisenstein series on the group G}}\label{sec:FexpnG}
In this section, we compute the Fourier expansion of the holomorphic Eisenstein series $E_{2r, \chi}$, and for convenience of notation we also work with
\begin{align}\index{$G_{2r, \chi}$}\label{equ:G2rchidefn}
G_{2r, \chi}(Z):=\frac{{E}_{2r,\chi}(Z)}{ {D_{B}^r C_{\infty}(r,r)}  L(\chi,2r)L^{(D_{B})}(\chi,2r-2)L(\chi,2r-4)}.
\end{align}
Our main result for this section is Theorem \ref{thm:Eseriescoeffs}.  
\begin{rmk}\label{rmk:Gausssumvalue}
We denote by $g(\chi)$\index{$g(\chi)$} the Gauss sum of $\chi$, i.e.
 \begin{align*}
g(\chi) := \sum_{n=1}^c\chi(n)e^{2\pi i n/c}
\end{align*}
for $c$ the modulus of $\chi$.  Many useful facts about Gauss sums are summarized in \cite[Chapter 9]{davenport}.  We briefly recall those needed here.  We denote by $c_\chi$\index{$c_\chi$} the conductor of the primitive character $\chi_0$\index{$\chi_0$} inducing $\chi$.  We have $g(\chi) = \mu(r)\chi_0(r)g(\chi_0),$ for $r$ the integer such that $rc_\chi$ is the conductor of $\chi$ and $\mu$ the M\"obius function (by \cite[p. 69]{davenport}).  If $\chi_0$ is nontrivial, we have $g(\chi_0)\overline{g(\chi_0)}=c_\chi$.  We also have $g(\overline{\chi_0})= \chi(-1)\overline{g(\chi_0)}$.  So in the special case where $\chi$ is quadratic (i.e. $\chi^2$ is trivial on $(\ZZ/c\ZZ)^\times$), we have that the value $g(\chi)^2$ is $c_\chi$ or $-c_\chi$.  When $\chi_0$ is trivial, we have $g(\chi_0)=-1$.  We set $c_\chi^\ast :=g(\chi)^2\in\ZZ$.    This is consistent with the notation used in the familiar setting of Legendre symbols $\left(\frac{\cdot}{p}\right)$ for an odd prime $p$, in which case we have $g(\chi)=\sqrt{p^\ast}$ with $p^\ast = (-1)^{\frac{p-1}{2}}p$.
\end{rmk}
\begin{thm}\label{thm:Eseriescoeffs}
For $2r >10$, the Fourier coefficients of $G_{2r, \chi}$ 
lie in $\IQ(\chi, g(\chi))$, or more precisely, in $\bigcup_{j=0}^3\left(\IQ(\chi)g(\chi)^{-j}\right)\subset\IQ(\chi, g(\chi))$, i.e. each Fourier coefficient is of the form $\alpha g(\chi)^{-j}$ with $j\in\{0,1,2,3\}$ and $\alpha\in\IQ(\chi)$.
\end{thm}

\begin{rmk}
In this paper, we only use the coarser statement that the Fourier coefficients lie in $\IQ(\chi, g(\chi))$.  We anticipate that the finer statements in Theorem \ref{thm:Eseriescoeffs} and related statements below about the coefficients lying in a particular subspace might be useful for applications to Deligne's conjectures and to Iwasawa theory, though, hence our recording them here.
\end{rmk}
\begin{rmk}
We have
\begin{align*}
G_{2r, \chi}(Z)&:=\sum_{\gamma \in  \Gamma_{\infty} \setminus \Gamma_G^0(M)} \chi(\gamma)   \frac{\nu(\gamma)^{-r} }{{j(\gamma , Z)}^{2r}}.
\end{align*}
So we need to study the Fourier coefficients of $\sum_{\gamma \in  \Gamma_{\infty} \setminus \Gamma_G^0(M)} \chi(\gamma)   \frac{\nu(\gamma)^{-r} }{{j(\gamma , Z)}^{2r}}$ for $2r>10$.  For the remainder of this section, we assume $2r>10$.
\end{rmk}

\subsubsection{Strategy}
Our proof of Theorem \ref{thm:Eseriescoeffs} has several steps, which we complete in Sections \ref{sec:decompositionGj} through \ref{sec:localfactorrk3}.  In Section \ref{sec:decompositionGj}, we decompose $G_{2r, \chi}(Z)$ as a sum of forms $G^{(j)}$, $0\leq j\leq 3$, each of which has a Fourier expansion of the form
$\sum_{h\in H_3(B)^\vee}a^{(j)}_hq^h,$
where $q^h=e^{2\pi i \tr(Z, h)}$, $H_3(B)^\vee$ is the $\ZZ$-dual of $H_3(B)$, and $a^{(j)}_h\in \IC$.
Our computation of the Fourier coefficients $a_h^{(j)}$  has two parts: an infinite part (completed in Section \ref{sec:factoratinfty}) and a finite part (begun in Section \ref{sec:finitepart}).  The computation of the finite part is further subdivided into four additional pieces: the computation of the Fourier coefficients at $h\in H_3(B)^\vee$ of rank $0, 1, 2,$ and $3$ (with ranks $<3$ handled in Section \ref{sec:localfactorrklessthan3} and rank $3$ in Section \ref{sec:localfactorrk3}).  In fact, similarly to the strategy in \cite{Tsao, KimExceptional}, each $G^{(j)}$ will be defined in such a way that $a_h^{(j)}=0$ whenever $h$ is not of rank $j$.  Furthermore, our computations will show that $a_h^{(j)}\in\IQ(\chi)g(\chi)^{-j}$.

\begin{rmk}
For proofs of algebraicity and $p$-adic interpolation, it is sufficient to show that the coefficients $a_h^{(j)}$ are polynomials in $h$ with coefficients in $\ZZ$.  Our approach, which is inspired by \cite{KarelFourier}, happens to give a more precise description of these polynomials than is strictly necessary for our applications to $L$-functions.
\end{rmk}

\subsubsection{\texorpdfstring{The forms $G^{(j)}$}{The forms G^(j)}}\label{sec:decompositionGj}
In this section, similarly to the strategy of \cite{Tsao, KimExceptional}, we decompose $G_{2r, \chi}(Z)$ as a sum of forms $G^{(j)}$ whose Fourier coefficients we will compute in the following sections.

We begin by introducing some conventions.  For $X \in H_3(B)$ and $v$ a rational prime, we write \index{$\mr{val}_{v}$}$\mr{val}_{v}(X) = \mr{min}(\mr{val}_{v}(x_{i,j}))$, and we define 
\index{$\kappa(X)$} \[ \kappa(X):= \prod_v  v^{-\mr{min}(0,\mr{val}_{v}(X), \mr{val}_{v}(X^{\#}), \mr{val}_{v}(N(X)) )}.\]
Recall from Equation \eqref{enofZ} that
$(1,0,0,0)n(X)=(1,X,X^{\#},N(X)).$ 
By \cite[Lemma 4.3]{KarelFourier}, we also have $\chi({n}(X)) = \chi (\kappa(X))$.
The factor $\kappa(X)$
defined in \cite[p.~522--523]{Baily}  as the determinant of the adjoint action of $n(X)$ on the unipotent of $P$ (see also \cite[\S 7.2]{Tsao}, where it is instead denoted by $c$, and \cite[\S 4]{KarelFourier}) coincides with ours by \cite[Lemma 4.3]{KarelFourier}. 
For each $a \in B$, we define an element $(a)_{i,j} \in \GL(H_3(B))$ ($i \neq j$) via
$$(a)_{i,j}.X = (\mr{Id}_3 + a^{*} e_{j,i})X (\mr{Id}_3 + a e_{i,j}),$$
where $e_{i, j}$\index{$e_{ij}$} denotes the elementary matrix with $1$ in the $ij$-position and $0$ everywhere else.
We let $L$\index{$L$, subgroup of $G$} be the subgroup of $G$ generated by all $\mu=(a)_{i,j}$ with $a \in B$, $  1\leq i \neq j \leq 3$. We embed it in $G$ via 
$(a,b,c,d)\mu=(a,\mu.b,\mu^*.c,d)$. This is (part of) the Levi of $P$. 
Similarly to the proof of \cite[Lemma 7.3]{Tsao}, every element of $P(\m Q)\setminus G(\m Q)$ is then represented (uniquely) by an element of the form $\iota_j n(X) \mu$ with $X$ in $H_3(B)$ and $\mu \in L(\m Q)$.

Similarly to the strategy of \cite{Tsao,KimExceptional}, we write
$$ 
G_{2r, \chi}(Z) = G_{2r, \chi}^{(0)}(Z)+G_{2r, \chi}^{(1)}(Z)+G_{2r, \chi}^{(2)}(Z)+G_{2r, \chi}^{(3)}(Z),
$$
where\index{$G_{2r, \chi}^{(j)}$} 
\begin{align}\label{equ:Gjsumdefn}
G_{2r, \chi}^{(j)}=\sum_{\gamma \in \Gamma_G^0(M), \gamma= \iota_j {n}(X) \mu } \chi(\gamma ) {\nu(\gamma)}^{-r} {j(\gamma, Z)}^{-2r},
\end{align}
with the sum only over the classes $\gamma$ represented in $P(\m Q)\setminus G(\m Q)$ by an element of the form $\iota_j n(X) \mu$ with $X$ in $H_3(B)$ and $\mu \in L(\m Q)$.   For ease of notation, we drop the subscript and write\index{$G^{(j)}$}
\begin{align*}
G^{(j)}:=G_{2r, \chi}^{(j)}
\end{align*}
for the remainder of the computations in this section.

Note that in the special case $j=0$, the sum in Equation \eqref{equ:Gjsumdefn} is over a single element, namely the identity matrix $\iota_0$.  We have
\begin{align}\label{equ:G0identically1}
G^{(0)}(Z):=\chi(\iota_0) {\nu(\iota_0)}^{-r} {j(\iota_0, Z)}^{-2r} = 1\times1\times1 = 1.
\end{align}
So $G^{(0)}$ is the constant function $1$.  To prove Theorem \ref{thm:Eseriescoeffs}, we are now reduced to determining the Fourier coefficients of $G^{(j)}$ for $j=1, 2, 3.$  Consequently, for the remainder of this section, we assume $j\neq 0$.

Passing back to the sum over rational matrices and using  the formula from \cite[Equation (2.1)]{KimExceptional} for $j \neq 0$
\[
j(\iota_j {n}(X) \mu , Z)=\pm N_{j}(\mu.Z+X)\kappa(X)
\] (and using the convention that $\chi (\kappa(X))$ is zero if $\kappa(X)$ is not coprime to $M$), we obtain
\begin{align}\label{equ:Gjdecomptemp}
G^{(j)}(Z)= & \sum_{X \in H_3(B), \mu  \in L(\m Q)} \chi(\kappa(X)) \kappa(X)^{-2r} {N_{j}(\mu.Z+X)}^{-2r}
\end{align}
for $j=1, 2, 3$.

Observe that each of the series $G^{(j)}$ is invariant by translation $Z \mapsto Z+Y $ for all $Y \in  H_3(B_0)$.  Thus, similarly to \cite[Equations (5) and (10) of Section 7]{Tsao}, $G^{(j)}(Z)$ has a Fourier expansion of the form
\begin{align}\label{expr:FexpnGj}
\sum_{h\in H_3(B_0)^\vee}a_h e^{2\pi i \tr(Z, h)},
\end{align}
where $H_3(B_0)^\vee$ denotes the $\mathbb{Z}$-dual of $H_{3}(B_0)$ and $a_h\in\IC$.

We want to perform further reductions. 
Let $\mathbb{H}_{j}(B_0)$\index{$\mathbb{H}_{j}(B_0)$} (resp. $\mathbb{H}_{j}(B)$) denote the group of elements of $H_3(B_0)$ (resp. $H_3(B)$) which have $0$'s outside the top left $j \times j$ block, and let $\mathbb{H}_{j}(B_0)^\vee$\index{$\mathbb{H}_{j}(B_0)^\vee$} denote the $\mathbb{Z}$-dual of $\mathbb{H}_{j}(B_0)$. 

\begin{lem}[Theorem 3.4.1 of \cite{Tsao}]\label{lem:Tsao341}
Each matrix $h \in H_3(B_0)^\vee$ is conjugate to an element of $\mathbb{H}_{j}(B_0)^\vee$ of maximal rank $j$, for $j$ the rank of $h$. Moreover, $h$ can be diagonalized by a unitary matrix $\mu$.
\end{lem}

Then, proceeding similarly to \cite[\S 8]{Tsao} and \cite[\S 3]{KimExceptional}, 
we furthermore have that $G^{(j)}(Z)$ has a Fourier expansion indexed by matrix of rank $j$.
\begin{lem}[Section 8.8 of \cite{Tsao}]\label{lem:FcoeffGjzero}
The $h$-th Fourier coefficient $a_h$ of $G^{(j)}$ is $0$ whenever $h$ is not conjugate to an element of $\mathbb{H}_{j}(B_0)^\vee$ of maximal rank $j$.
\end{lem}

For a fixed $\mu$ and for $X \in {H}_{3}(B)/{H}_{3}(B_0)$, we will study in Section \ref{sec:factoratinfty} the hypergeometric series
\[
 \sum_{S \in H_3(B_0)} N_{j}\left(\mu.Z+S+X\right)^{-2r}.
\]

Once we have dealt with the Fourier expansion of this series, we want to reduce to  a sum over $X \in \mathbb{H}_{j}(B)/ \mathbb{H}_{j}(B_0)$ rather than a generic lattice in $\mathbb{H}_{j}(B\otimes \mathbb{R})$. We proceed as in \cite[\S 7.6]{Tsao} to reduce ourselves to the case where $h \in \mathbb{H}_{j}(B_0)^\vee$.
Indeed, for any $h \in H_3(B_0)^\vee$ we have $\mu$ and  $h' \in \mathbb{H}_{j}(B_0)^\vee$ such that $\mu^*.h=h'$ (note that the action $\mu^*.h$ is given by conjugation). Then the $h$-th Fourier coefficient of $G^{(j)}(Z)$ will be the $h'$-th Fourier coefficient of $G^{(j)}(\mu.Z)$, as we have $e^{2 \pi i \mr{tr}(\mu.Z,h)}=e^{2 \pi i \mr{tr}(Z,\mu^*.h)}$  (one needs to use the formula $\mr{tr}(\mu.x,y)=\mr{tr}(x, \mu^*.y)$ \cite[\S 3.4]{Tsao}) and we obtain

\begin{align}\index{$C_{\infty,h}^{(j)}$}
G^{(j)}(Z)= &\sum_{h \in {H}_{3}(B_0)^\vee} C_{\infty,h}^{(j)}  \sum_{\mu, X \in \mathbb{H}_{j}(B)/\mathbb{H}_{j}(B_0)} \chi(n(X))  \kappa(X)^{-2s}    e^{2 \pi i \mr{tr}(\mu.Z + \mu.X,h)}\nonumber\\
 = & \sum_{h \in {H}_{3}(B_0)^\vee} C^{(j)}_{\infty,h}\nonumber \\
  & \times \sum_{\mu} \left( \sum_{X \in \mathbb{H}_{j}(B)/\mathbb{H}_{j}(B_0) }\chi(\kappa(X))   \kappa(X)^{-2r}  e^{2 \pi i \mr{tr}(X,h)}  \right) e^{2 \pi i \mr{tr}(Z,\mu^*.h)}\label{equ:GjFexpnexpression}
\end{align}
and we will make $C_{\infty,h}^{(j)}$ explicit in Section \ref{sec:factoratinfty}.  
This type of sum over $X$ is the focus of \cite[\S 9,10]{Tsao} and \cite[\S 4]{KimExceptional} (which make heavy use of the explicit calculations of \cite{KarelFourier}) and can be dealt with prime by prime.

\subsubsection{Infinite part}\label{sec:factoratinfty}
We begin by making $C_{\infty,h}^{(j)}$ explicit for $j=1, 2, 3$.  Since $s=r$ and $2r >10$, we have, by  \cite[Lemma 8.5]{Tsao}  (in the notation of {\it loc. cit.} $\rho=2r$ and $\mathbb{N}=10$), that  for all $h \in \mathbb{H}_{j}(B_0)^\vee$, the term $ C_{\infty,h}^{(j)}$ vanishes if $\mr{rank}(h)\neq j$.
In low weight, some of these factors might not vanish (analogously to how in the setting of $\GL_2$, the factor $\frac{1}{Y}$ appears for the $\mr{GL}_2$-Eisenstein series of weight $2$).  One can study the singular Fourier coefficients similarly to the approach in \cite{KimExceptional}.  
In particular, we may assume that $h \in  \mathbb{H}_{j}(B_0)^\vee$ has maximal rank $j$.
So for a fixed $\mu \in L(\m Q)$ we need to study the series 
\[
 \sum_{S \in H_3(B_0)} N_{j}\left(\mu.Z+X+S\right)^{-2r}.
\]
Shimura studied this type of series in his work on confluent hypergeometric functions \cite{ShimConfluent}.

Using \cite[Lemma 8.4]{Tsao} we get
\begin{align*}
\prod_{\iota=0}^{j-1} \Gamma\left( 2r -2\iota \right) \sum_{S \in \mathbb{H}_{j}(B_0)} N_{j}\left(Z-{S}\right)^{-2r} = & \mr{Vol}(\mathbb{H}_{j}(B_0)^\vee)(2 \pi i)^{2rj}\pi^{-j(j-1)}\\
 & \times \sum_{h \in \mathbb{H}_{j}(B_0)^\vee} N_{j}(h)^{2r-2j-1}e^{2 \pi i \mr{tr}(Z,h)},
\end{align*}
where $\mr{Vol}(\mathbb{H}_{j}(B_0)^\vee)$ is the volume of $\mathbb{H}_{j}(B_0)^\vee$. (Note: In the notation in \cite{Tsao}, $n$ is our $j$; $n_0$ is $4$, the dimension of $B$; $\m N=(n-1)n_0+2=4j-2$.)
Setting\index{$\Gamma_j$} 
\[
\Gamma_j(\alpha) =  \prod_{\iota=0}^{j-1} (2\pi i)^{\alpha} \pi^{-2\iota} \Gamma\left(\alpha-2{\iota}\right),
\]
we have\index{$C_{\infty,h}^{(j)}$}
\begin{align}\label{equ:Cinftyjexpression}
C_{\infty,h}^{(j)} = \frac{N_{j}(h)^{2r-2j-1}\mr{Vol}(\mathbb{H}_{j}(B_0)^\vee)}{\Gamma_{j}(2r)}.
\end{align}

We need to make the volume factors a bit more explicit. This is done in \cite[\S 11]{Tsao}
\begin{lem}\label{lemma:volumefactor}
The volume factors $\mr{Vol}(\mu.\mathbb{H}_{j}(B_0)^\vee)$ are rational numbers for every $\mu$.
\end{lem}
\begin{proof}
We start for $j=3$. We calculate the volume using the measure $\textup{d} x$ already used in Lemma \ref{lemma:implicitconstant}, the one induced by the usual Euclidean measure on 
\[
B_0 \otimes_{\mathbb{Z}} \mathbb{R} \cong  \mathbb{R} \mb{1}\oplus \mathbb{R} \mb{i} \oplus \mathbb{R} \mb{j} \oplus \mathbb{R}\mb{k}. 
\]
This gives that the covolume of $B_0$ is (again by \cite[Lemma 7.6]{pollack1})  $\det(T)=4^{-1}D_B$, for $T$ the matrix fixed in the paper, associated with $B$.
Hence, $mr{Vol}(\mathbb{H}_{3}(B_0))=\det(T)^3$, as 
\[
\mathbb{H}_{3}(B_0)\cong B_0^3 \oplus \mathbb{Z}^3.
\]
Now, the volume of every other lattice in $\mathbb{H}_{3}(B_0)\otimes_{\mathbb{Z}} \mathbb{R}$, may it be $\mathbb{H}_{j}(B_0)^\vee$ or $(\mu.\mathbb{H}_{j}(B_0))^\vee$, will differ by the volume of $\mathbb{H}_{3}(B_0)$ by the index, which is rational number.

For $j=2$ we have \[
\mathbb{H}_{2}(B_0)\cong B_0 \oplus \mathbb{Z}^2
\]
and its volume is $\det(T)$, and similarly for every commensurable lattice. 

For $j=1$, we \[
\mathbb{H}_{1}(B_0)\cong \mathbb{Z}
\]
which has volume $1$.
\end{proof}

\subsubsection{Finite part}\label{sec:finitepart}
We now need to study the series 
\begin{align}\label{finitesum}
\sum_{X \in \mathbb{H}_{j}(B)/\mathbb{H}_{j}(B_0)} \chi(\kappa(X))   \kappa(X)^{-2r}  e^{2 \pi i \mr{tr}(X,h')} .
\end{align}
arising in Equation \eqref{equ:GjFexpnexpression}.
We only consider $2r>10$.  So we only need to consider matrices $h'$ of rank $j$.
Since all terms in the sum are multiplicative, we have that the sum \eqref{finitesum} decomposes as
\begin{align*}
\sum_{X \in \mathbb{H}_{j}(B)/\mathbb{H}_{j}(B_0)} \chi(\kappa(X))   \kappa(X)^{-2r}  e^{2 \pi i \mr{tr}(X,h')}=\prod_{\ell } S^{(j)}_\ell(h'),
\end{align*}
with
\index{$S_\ell^{(j)}$}
\begin{align}
S^{(j)}_\ell(h')= & \sum_{m=0}^{\infty}\left (\sum_{X  \in \mathbb{H}_{j}(B \otimes \m Q_\ell)/\mathbb{H}_{j}(B_0 \otimes \m Z_\ell), \kappa(X) = \ell^m} e^{2 \pi i \mr{tr}(X,h')}\right) \chi(\ell^{m}) \ell^{-2mr} \label{equ:Sljdefn}\\
 = & (1-\chi(\ell)\ell^{-2r}) \sum_{m=0}^{\infty} \left(\sum_{\kappa(X) \leq \ell^m} e^{2 \pi i \mr{tr}(X,h')}\right)  \chi(\ell)^{m} \ell^{-2mr},\label{equ:Sljwithinteriorsum}
\end{align}
where the second equality holds if $j \neq 0$.
We study the sums $S^{(j)}_\ell$ in extensive and intricate computations in Sections \ref{sec:localfactorrklessthan3}, \ref{rank2}, and \ref{sec:localfactorrk3}.  

\begin{rmk}\label{rmk:reducetoblock}
By \cite[Lemma 7.6]{Tsao}, to prove rationality of each of the Fourier coefficients in Expression \eqref{expr:FexpnGj}, it suffices to prove rationality of the Fourier coefficients $a_{h'}$ for the matrices $h'$ whose entries are $0$ everywhere outside the upper left $j\times j$ block.  So we work with $h'$ of this simpler form in Sections \ref{sec:localfactorrklessthan3} and \ref{rank2}.  
 \end{rmk}
So that we may prove Theorem \ref{thm:Eseriescoeffs} without further delay, we summarize the main results here: 
\begin{itemize}
\item[($j=0$)] 
\[S^{(0)}_\ell(0) =1;\]
\item[($j=1$)] 
\[
S^{(1)}_\ell(h') =  (1- \chi(\ell)\ell^{-2r})P_\ell(h',\chi(l)l^{-2r});
\]
\item[($j=2$)]
\[
S^{(2)}_\ell(h') = (1- \chi(\ell)\ell^{-2r})(1- \chi(\ell)\ell^{2-2r})P_\ell(h', \chi(\ell)\ell^{-2r});
\]
\item[($j=3$)] 
\[
S^{(3)}_\ell(h') = (1- \chi(\ell)\ell^{-2r})(1- \chi(\ell)\ell^{2-2r})(1- \chi(\ell)l^{4-2r})P_\ell(h',\chi(\ell)\ell^{-2r}),
\]
\end{itemize}
with the factors $P_\ell(h',\chi(\ell)\ell^{-2r})$\index{$P_\ell$} denoting polynomials that depend only on the elementary divisors of $h'$ and that, furthermore, satisfy $P_\ell(h',\chi(\ell)\ell^{-2r})=1$ whenever $\ell \nmid N(h')$.\\

\begin{proof}[Proof of Theorem \ref{thm:Eseriescoeffs}]
The proof of Theorem \ref{thm:Eseriescoeffs} follows immediately from the expressions of the Fourier expansion of each $G^{(j)}$ in Equations \eqref{equ:G0identically1} and \eqref{equ:GjFexpnexpression}, the expression of the infinite part in Equation \eqref{equ:Cinftyjexpression},
and the algebraicity of the components of the finite part in Section \ref{sec:finitepart}.  In particular, for $j\neq 0$, we have
\begin{align}\label{equ:Cjproduct}
C_{\infty,h}^{(j)}\prod_\ell S_\ell^{(j)} = L(j)^{-1} N_{j}(h)^{2r-2j-1}\mr{Vol}(\mathbb{H}_{j}(B_0)^\vee)  \prod_{\ell\mid N(h')}P_\ell(h', \chi(\ell)\ell^{-2r}),
\end{align}
where $L(j)$ equals
\begin{itemize}
\item[($j=3$)] 
\[\frac{L(\chi,2r)L(\chi,2r-2)L(\chi,2r-4)\Gamma(2r)\Gamma(2r-2)\Gamma(2r-4)}{{(2i)}^{6r}\pi^{(6r-6)}};\]
\item[($j=2$)] 
\[\frac{L(\chi,2r-2) L(\chi,2r-4)\Gamma(2r-2)\Gamma(2r-4)}{{(2i)}^{4r}\pi^{(4r-6)}};\]
\item[($j=1$)]
\[\frac{ L(\chi,2r-4)\Gamma(2r-4)}{{(2i)}^{2r}\pi^{(2r-4)}}.\]
\end{itemize} 
(We consider $\chi$ as a character modulo $M$ so that $L(\chi,s)$ has no Euler factors at primes dividing $M$, or equivalently the Euler factor is $1$ at primes dividing $M$.)  
Recall that the generalized Bernoulli numbers are algebraic numbers in $\m Q(\chi) $ and coincide (up to some Euler factors if $\chi$ is not primitive, as in \cite[Equation (3)]{davenport}) with the values at negative integers of the Dirichlet $L$-function $L(\chi,s)$. Using the functional equation for Dirichlet $L$-functions (as in, e.g., \cite[Chapter 9, Equation (8)]{davenport}), we then get  \[ 
\frac{\Gamma(2r-2\iota)L(\chi,2r-2\iota)}{(2\pi i)^{2r-2\iota}} \in \m Q(\chi)g(\chi),
\]
where we recall $g(\chi)$ denotes the Gauss sum.
By \ref{lemma:volumefactor} the volume factors are rational numbers.
Thus Theorem \ref{thm:Eseriescoeffs} holds.
\end{proof}
\begin{cor}\label{coro:Eseriescoeffs}
For $2r >10$, the Fourier coefficients of $E_{2r, \chi}$ lie in $\IQ(\chi, g(\chi))$, or more precisely, in $\bigcup_{j=0}^3\left(\IQ(\chi)g(\chi)^j\right)\subset\IQ(\chi, g(\chi))$.
\end{cor}
\begin{proof}
Indeed, note that we have the relation \[ E_{2r, \chi}={ D_B^r C_{\infty}(r,r) L(\chi,2r)L^{(D_{B})}(\chi,2r-2)L(\chi,2r-4)}G_{2r, \chi} .\] 
Recall that
$ C_{\infty}(r,r) = \pi^{-(6r+6)}\Gamma(2r)\Gamma(2r-2)\Gamma(2r-4)$, then as in the end of the proof of Theorem \ref{thm:Eseriescoeffs}, the quantity $C_{\infty}(r,r) L(\chi,2r)L^{(D_{B})}(\chi,2r-2)L(\chi,2r-4)$ is essentially given by three generalised Bernoulli numbers, which are algebraic. This allows us to conclude the proof.
\end{proof}
In the Sections \ref{sec:localfactorrklessthan3}, \ref{rank2}, and \ref{sec:localfactorrk3}, we explicitly compute the sums $S_\ell^{(j)}$ defined in Equation \eqref{equ:Sljdefn} and summarized immediately preceding the proof of Theorem \ref{thm:Eseriescoeffs}.  We closely follow \cite{KarelFourier}, which treats the case when the sum is over Hermitian matrices over an octonion algebra.  

We fix an integer $j$ and a matrix $h \in \mathbb{H}_j(B_0)$ of rank $j$.  We also fix a prime number $\ell$ and we write \index{$B_\ell$}$B_\ell=B_0 \otimes_{\m Z} \m Q_\ell$, \index{$B_{0,\ell}$}$B_{0,\ell}=B_0 \otimes_{\m Z} \m Z_\ell$, and \index{$\Lambda_j$}$\Lambda_j=\mathbb{H}_{j}(B_{0,\ell})$.  As above, let $X \in \mathbb{H}_{j}(B)/\mathbb{H}_{j}(B_0)$.
We have (by definition) that $\kappa(X) \leq \ell^m $ if and only if $\ell^m X, \ell^m X^{\#} \in \Lambda_j$, and $\ell^m N(X) \in \m Z_\ell$. We rewrite the interior sum from Equation \eqref{equ:Sljwithinteriorsum}, i.e.\ the sum indexed by the inequality $\kappa(X) \leq \ell^m$,  doing the change of variables $X\mapsto\ell^mX$ to obtain
\begin{align}\label{equ:sumrewrite}
\sum_{\kappa(X) \leq \ell^m} e^{2 \pi i \mr{tr}(X,h)} =\left\{ \begin{array}{cc} 
 \sum_{ X\in \Lambda_1 / \ell^m \Lambda_1} \omega_m ({\mr{tr}(X,h)}), & j=1\\
 \sum_{ X\in \Lambda_2 / \ell^m \Lambda_2, X^{\#}  \in \ell^m \Lambda_2 } \omega_m ({\mr{tr}(X,h)}),& j=2\\
  \sum_{ X\in \Lambda_3 / \ell^m \Lambda_3, X^{\#}  \in \ell^m \Lambda_3, \ell^{2m} \mid N(X)} \omega_m ({\mr{tr}(X,h)}),& j=3
\end{array} 
\right.
\end{align}
where $\omega_m (z) = e^{2 \pi i z /\ell^m}$,\index{$\omega_m$} with $z \in \m Z_\ell$, is the usual additive character (the $\ell^m$ in the denominator of $\omega_m$ comes from the change of variables $X \mapsto l^m X$).  

We define, for any $\ell$-adic integer $ \lambda$, 
\index{$[\lambda]_m$}
\begin{align*}
[\lambda]_m:=\ell^{\mr{inf}(m,v_\ell(\lambda))}.
\end{align*}
 Note that $[\lambda]_m$ is the number of $x \in \set{1,\ldots,\ell^m}$ such that $\lambda x \equiv 0 \bmod \ell^m$.  The terms $[\lambda]_m$ are closely related to the polynomials arising in the Fourier expansions of our Eisenstein series.
We have the following lemma. 
\begin{lem}[Lemmas 2.1, 2.3, and 2.4 in \cite{KarelFourier}]\label{lem1}
Let $a_1$ be an element of $B_{0,\ell}$ such that $n(a_1) \equiv 0 \bmod \ell^m$ and let $A\left(a_1\right)$ be the number of elements $a_2 \in B_{0,\ell}/\ell^mB_{0,\ell}$ such that $a_1^*a_2 \equiv n\left(a_2\right) \equiv 0 \bmod \ell^m$.  Define $S(\lambda,a_1) = \sum_{a_2}\omega_m(\lambda n\left(a_2\right))$, where the sum runs over the elements $a_2$ satisfying this relationship with $a_1$.  If $f$ denotes the minimum between the $\ell$-adic valuations of $a_1$ and $\ell^{-m}n\left(a_1\right)$, we have
\begin{align}
\sum_{ x \in B_{0,l}/l^m B_{0,\ell}} \omega_m(\lambda n(x))= & \ell^{2m}[\lambda]^2_m \label{lem1.1}\\
S(\lambda,a_1) = & \ell^{2m}[\lambda]^2_f \label{lem1.2}\\
A(a_1) = & \ell^{2m}\left( \sum_{j=0}^f \ell^j -  \sum_{j=0}^{f-1} \ell^{j-1} \right). \label{lem1.3}
\end{align}
\end{lem}

With these lemmas at hand, we are redy to calculate the Fourier coefficients. The method is what Karel calls {\it Siegel's Babylonian reduction process};  we are not sure where the adjective ``Babylonian'' comes from. (In \cite{LatticePackings}, for a theorem on the $LDU$ decomposition of matrices that Weil calls Babylonian, the author guesses that the adjective comes from an analogy with completing the square for quadratic equation.)
The key idea is to repeatedly use the algebraic identity 
\[
\sum_{m=0}^{\infty} \left(\sum_{X | \gamma(X)=m}  a(X) \right)\ell^{-m} = (1-\ell^{-1})\sum_{m=0}^{\infty} \left( \sum_{X | \gamma(X)\leq m} a(X) \right)\ell^{-m} 
\]
(here $X$ ranges over a certain set, and we have two functions $\gamma$ and $a$ defined over this set, with $\gamma$ taking values in the non-negative integers) used to pass from Equation \eqref{equ:Sljdefn} to Equation \eqref{equ:Sljwithinteriorsum}, continuing rewriting the sum in parentheses until we are reduced to a finite sum over $m$, which will give a polynomial in $\ell$.

\subsubsection{\texorpdfstring{Local factors for Eisenstein series: ranks $<2$}{Local factors for Eisenstein series: ranks <2}}\label{sec:localfactorrklessthan3}
We now derive the expressions for the factors $S_\ell^{(j)}$ that were stated in Section \ref{sec:finitepart}.  Because the computations for $j=2, 3$ are particularly extensive, we handle $j=2, 3$ separately in Sections \ref{rank2} and \ref{sec:localfactorrk3} after handling $j=0, 1$ in the present section.

\paragraph{\bf Rank zero} By Equation \eqref{equ:G0identically1}, $G^{(0)}=1$, so we already know that the Fourier coefficients of $G^{(0)}$ are algebraic, even without computing $S_\ell^{(0)}$.  For the sake of completeness, though, we note that the definition of $S_\ell^{(j)}$ in Equation \eqref{equ:Sljdefn} immediately gives $S^{(0)}_\ell(h) =1$, since $0$ is the only rank $0$ matrix.

\paragraph{\bf Rank one} We identify  $\Lambda_1$ with $\ZZ_\ell$, by identifying the diagonal matrix $[a, 0, 0]$ with the element $a\in \ZZ_\ell$.  Thanks to Remark \ref{rmk:reducetoblock}, it suffices to consider those matrices $h$ whose entries are $0$ outside of the upper left $1\times 1$ block.  So we are reduced to considering diagonal matrices $h=[n, 0, 0]$ with $n\neq 0$.  We write $n=\ell^a n'$, with $\gcd(n',\ell)=1$. Then we get
\begin{align*}
\sum_{\kappa(X) \leq \ell^m} e^{2 \pi i \mr{tr}(X,h)} &= \sum_{ X \in \m Z_\ell/ \ell^m \m Z_\ell } e^{2 \pi i X \ell^a /\ell^m}\\
&=\begin{cases}
\ell^m& \mbox{if } a\geq m\\
0 & \mbox{else}
\end{cases}
\end{align*}
So the expression for $S^{(1)}_\ell$ from Equation \eqref{equ:Sljwithinteriorsum} becomes
\begin{align*}
S^{(1)}_\ell(h) & = \left(1- \chi(\ell)\ell^{-2r}\right)\left(1+\chi(\ell)\ell^{1-2r}+ \chi(\ell^2)\ell^{2-2\cdot 2r}+ \cdots + \chi(\ell^{\mr{val}_\ell(n)})\ell^{\mr{val}_\ell(n)-2r \mr{val}_\ell(n)}\right)\\
&=\left(1- \chi(\ell)\ell^{-2r}\right)P_\ell(h,\chi(\ell)\ell^{-2r}),
\end{align*}
where $P_\ell(h,T)$ is the polynomial defined by
\begin{align*}
P_\ell(h,T):=\sum_{m=0}^{a}\ell^mT^m.
\end{align*}

\subsubsection{\texorpdfstring{Local factors for Eisenstein series: rank $2$}{Local factors for Eisenstein series: rank 2}}\label{rank2}
Thanks to Remark \ref{rmk:reducetoblock}, it suffices to consider those matrices $h$ whose entries are $0$ outside of the upper left $2\times 2$ block.  Furthermore, by Lemma \ref{lem:Tsao341}, such matrices can be diagonalized by conjugation by a unitary matrix.   Since the trace of a Hermitian matrix is invariant under conjugation by a unitary matrix (see Remark \ref{rmk:traceinvariance}) and conjugation by a unitary matrix simply permutes the space of Hermitian matrices, we are reduced to the case where $h$ is a diagonal matrix of the form $[d_1,d_2, 0]$ $d_1 \mid d_2 \neq 0$.  Thus, we are reduced to considering Equation \eqref{equ:sumrewrite} in the case of $2\times 2$ matrices, with $h$ the diagonal matrix $[d_1, d_2]$ and $X$ in Equation \eqref{equ:sumrewrite} being written as  $X= \left( \begin{array}{cc}
c_1 & a_3 \\
a_3^* & c_2
\end{array}
\right)$.
For the remainder of the rank $2$ computation, we work exclusively with $2\times 2$ matrices.  We have $\tr(X, h) = c_1d_1+c_2d_2$, and adapting our conventions to the setting of $2\times 2$ matrices, we have $X^{\#}=\left( \begin{array}{cc}
0& 0 \\
0 & N(X)
\end{array}
\right)$ with $N(X)=c_1c_2-n(a_3)$.  The sum in Equation \eqref{equ:sumrewrite} can be expressed then as a sum over $X$ with $N(X)\in \ell^m\ZZ_\ell$.

Since $\ell^m\divides N(X)$, we have $ \sum_{\lambda \bmod \ell^m \m Z_\ell} \omega_m(\lambda N(X))=\ell^m$.  So 
\begin{align}
\ell^m \sum_{\kappa(X) \leq \ell^m} e^{2 \pi i \mr{tr}(X,h)} = & \sum_{X= \left( \begin{array}{cc}
c_1 & a_3 \\
a_3^* & c_2
\end{array}
\right), \lambda \bmod \ell^m} \omega_m(c_1d_1+c_2d_2)\omega_m(\lambda N(X))\nonumber \\
 = & \sum_{c_2 \bmod \ell^m} \omega_m(c_2d_2) \sum_{c_1,\lambda \bmod \ell^m} \omega_m(c_1(d_1+\lambda c_2)) \sum_{a_3 \in B_\ell / \ell^m B_\ell} \omega(-\lambda n(a_3)).\label{equ:2of3}
\end{align} 

Applying Lemma \ref{lem1} and noting that the second sum in Equation \eqref{equ:2of3} contributes $\ell^m$ if $d_1+\lambda c_2 \equiv 0 \bmod \ell^m$ and $0$ otherwise, we get
\begin{align*}
\ell^m \sum_{\kappa(X) \leq \ell^m} e^{2 \pi i \mr{tr}(X,h)} = & \ell^{2m+m} \sum_{c_2 \bmod \ell^m} \omega_m(c_2d_2) \sum_{\lambda \mbox{ s.t. } d_1+\lambda c_2 \equiv 0 \bmod \ell^m} [\lambda]^2_m.
\end{align*} 
So we need to calculate
\begin{align*}
 \sum_{\kappa(X) \leq \ell^m} e^{2 \pi i \mr{tr}(X,h)} = & \ell^{2m} \sum_{c_2 \bmod \ell^m} \omega_m(c_2d_2) \sum_{\lambda \mbox{ s.t. } d_1+\lambda c_2 \equiv 0 \bmod \ell^m} [\lambda]^2_m.
\end{align*}

Suppose first that $d_2$ (hence $d_1$) is prime to $\ell$. Given the condition $\lambda c_2 \equiv -d_1 \bmod \ell^m$,  the sum is only over $c_2$ and $\lambda$ coprime to $\ell$, and it reduces to $\ell^{2m} \sum_{(c_2,\ell)=1} \omega_m(c_2d_2)$. We can rewrite it as
\begin{align*}
 \sum_{\kappa(X) \leq \ell^m} e^{2 \pi i \mr{tr}(X,h)} = & \ell^{2m} \sum_{c_2} \omega_m(c_2d_2)(\beta(c_2)-\beta(c_2/\ell)).
\end{align*} 
where $\beta(x)=1$ if $x$ is in $\m Z_\ell$ and $0$ otherwise.  Distributing and using that $\beta(c_2/\ell)=1$ only if $\gcd(c_2,\ell) \neq 1$, we get 
\begin{align*}
 \sum_{\kappa(X) \leq \ell^m} e^{2 \pi i \mr{tr}(X,h)} = & \ell^{2m} \left( \sum_{c_2 \bmod \ell^m} \omega_m(c_2d_2)\beta(c_2)-\sum_{c_2 \bmod \ell^{m-1}} \omega_{m-1}(c_2d_2)\beta(c_2) \right).
\end{align*} 
Hence,
\begin{align*}
\frac{S_\ell^{(1)}(h)}{(1-\chi(\ell)\ell^{-2r})} & = \sum_{m=0} \ell^{2m} \left( \sum_{c_2 \bmod \ell^m} \omega_m(c_2d_2)\beta(c_2)-\sum_{c_2 \bmod \ell^{m-1}} \omega_{m-1}(c_2d_2)\beta(c_2) \right) \chi\left(
\ell^m\right)\ell^{-2rm} \\
 & =\left(1- \chi(\ell)\ell^{2-2r}\right) \sum_{m=0}  \left(\sum_{(c_2,\ell) =1} \omega_m\left(c_2d_2\right) \right)\chi(\ell^m) \ell^{2m-2rm}    \\
 & = \left(1- \chi(\ell)\ell^{2-2r}\right),
\end{align*}
as the last sum over $c_2$ is the sum of all the primitive $\ell^m$ roots of unit which is $1$ if $m=1$ and $0$ otherwise. \\

On the other hand, if $\ell \mid d_2$, then the sum
$ H_{c_2}:=\sum_{\lambda} [\lambda]^2_m$ can be written as 
\[
H_{c_2}= \beta_{m}(h,c_2)-\beta_{m-1}(h,c_2/\ell)
\]
for $\beta_{m}(h,c_2)=[d_1]_m^2 \sum_{k=0}^{\mr{inf}(v_\ell(d_1),v_\ell(c_2))}\ell^{-k}$, under the convention that a sum from $0$ to a negative number is $0$ (to handle the cases where $c_2/\ell$ is not an integer).

So 
\begin{align*}
\frac{S_\ell^{(1)}(h)}{\left(1-\chi(\ell)\ell^{-2r}\right)} & = \sum_{m=0} \ell^{2m} \left( \sum_{c_2 \bmod \ell^m} \omega_m(c_2d_2)\beta(c_2)-\sum_{c_2 \bmod \ell^{m-1}} \omega_{m-1}(c_2d_2)\beta(c_2) \right) \chi(\ell^m)\ell^{-2rm}\\
 & =\left(1- \chi(\ell)l^{2-2r}\right) \sum_{m=0}  \left(\sum_{c_2} \omega_m(c_2d_2) \beta_{m}\left(h,c_2\right)\right)\chi(\ell^m)\ell^{-2rm}  \\
 & = \left(1- \chi(\ell)l^{2-2r}\right)\sum_{m=0} \alpha'_m(h)\chi(\ell^m)\ell^{m(2-2r)},
\end{align*}
where \index{$\alpha'_m$}
\begin{align}\label{equ:alphaprimem}
\alpha'_m(h):= \left(\sum_{c_2 \bmod \ell^m} \omega_m(c_2d_2) [d_1]_m^2 \sum_{k=0}^{\mr{inf}(m,v_\ell(d_1),v_\ell(c_2))}\ell^{-j}\right).
\end{align}

Let $v:=\mr{inf}(m,\mr{val}_\ell(d_1))$ and $v':=\mr{inf}(m,\mr{val}_\ell(d_2))$. Swapping the sums in Equation \eqref{equ:alphaprimem}, we get that if $v < m-v' $ the sum of roots of unity is zero, hence
\begin{align*}
\alpha'_m(h)= &\ell^{2v}  \sum_{k=m-v'}^{v}\ell^{m-k}\ell^{-k},\\
\ell^{-m}\alpha'_m(h)= &\sum_{j=0}^{v+v'-m}\ell^{2j}.
\end{align*}

Let $\tau$ and $\tau'$ be the $\ell$-adic valuations of $d_1$ and $d_2$, we want to show that unless $j \leq \tau$ or $j \leq  m \leq \tau + \tau' -j$ the sum is empty. If $j \leq v+v'-m$, then $m \leq \tau + \tau' -j$, but as $v,v' \leq m$ then $j\leq m$. Using $v \leq \tau$ and $v' \leq m$, we get also $j \leq \tau$. This implies immediately 
\begin{align*}
\sum_{m=0} \alpha'_m(h)\chi\left(\ell^m\right)\ell^{m(2-2r)}= \sum_{j=0}^{\tau}\ell^{2j}\sum_{m=j}^{\tau+\tau'-j}\chi\left(\ell^m\right)\ell^{m(3-2r)}
\end{align*} 
which is the polynomial \[
P_\ell(h,T)= \sum_{j=0}^{\tau}\ell^{2j}\sum_{m=j}^{\tau+\tau'-j}\ell^{3m}T^m
\]
evaluated at $\chi(\ell)\ell^{-2r}$.

\subsubsection{\texorpdfstring{Local factors for Eisenstein series: rank $3$}{Local factors for Eisenstein series: rank 3}}\label{sec:localfactorrk3}

Inspired by \cite[Section 7]{KarelFourier}, we write 
\begin{align}\label{equ:decompX}
X=c_3U + W+Q,\index{$c_3$}
\end{align}
 where 
\begin{align*}
 U = & \left(  
\begin{array}{ccc}
0 & 0 & 0 \\
0 &0 &0 \\
0& 0& 1
\end{array} 
 \right),\\
 W = &
 \left(  
\begin{array}{ccc}
0       & 0 & a_2^* \\
0& 0     &a_1 \\
a_2   & a_1^*     & 0
\end{array} 
 \right),\\
\index{$Q$}Q=& \left(  
\begin{array}{ccc}
c_1 & a_3 & 0 \\
a_3^* &c_2 & 0 \\
0& 0 & 0
\end{array} 
 \right).
\end{align*}
Following the conventions of Equation \eqref{equ:hhashdefn}, we have 
$$ X^{\#}= c_3'U+W'+Q', $$ where
\begin{align*}\index{$c_3'$}
c_3'  = & c_2c_1-n(a_3),\\
\index{$W'$} W' = &
 \left(  
\begin{array}{ccc}
0       & 0  & a_3a_1-c_2a_2^* \\
0 & 0     &a_3^*a_2^*-c_1a_1 \\
a_3^*a_1^*-c_2a_2  & a_3a_2 -c_1a_1^*      & 0
\end{array} 
 \right),\\
\index{$Q'$}Q'=& \left(  
\begin{array}{ccc}
c_2c_3-n(a_1) & a_2^*a_1^*-c_3a_3 & 0 \\
a_2a_1-c_3a_3^* &c_1c_3-n(a_2) & 0\\
0& 0 & 0
\end{array} 
 \right).
\end{align*}
Note that 
\[
W^{\#} = 
 \left(  
\begin{array}{ccc}
-n(a_1)       & a_2^*a_1^*  & 0 \\
a_1a_2 & -n(a_2)     & 0 \\
0  & 0     & 0
\end{array} 
 \right).
\]
Applying \cite[Equation (2.4)]{pollack1} to the decomposition of $X$ from Equation \eqref{equ:decompX}, we obtain
\begin{align}\label{equ:NofX24}
N(X)= c_3N_2(Q)+\mr{tr}(W^{\#},Q)
\end{align}
Observe that $X^{{\#}} \equiv 0 \bmod \ell^{m}$  if and only if the following three conditionshold:  
\begin{itemize}
\item[i)]$N_2(Q)=c_3' \equiv 0 \bmod \ell^m$;
\item[ii)] $Q'\equiv 0 \bmod \ell^m$;
\item[iii)] $ W' \equiv 0 \bmod \ell^m$;
\end{itemize}
By Equation \eqref{equ:NofX24}, we also have that $N(X) \equiv 0 \bmod \ell^{2m}$  if and only if 
\begin{itemize}
\item[iv)] $c_3N_2(Q)+\mr{tr}(W^{\#},Q) \equiv 0 \bmod \ell^{2m}$.
\end{itemize}
These four conditions are analogues in our setting of the similarly number conditions in \cite[Section 8]{KarelFourier}, and we adapt the approach in \cite[Section 8]{KarelFourier}, which concerns an exceptional group, to our setting.

Similarly to \cite[Equations (8.2)]{KarelFourier}, Equation \eqref{equ:sumrewrite} can now be expressed as 
\begin{align*}
\sum_{\kappa(X) \leq l^m} e^{2 \pi i \mr{tr}(X,h)} = \sum_{N_2(Q) \equiv 0 \bmod \ell^m} \omega_m({\mr{tr}(Q,h)}) \sum_{c_3,W}\omega_m ({\mr{tr}(c_3U+W,h)}),
\end{align*}
where the inner sum is over $c_3$ and $W$ meeting conditions ii), iii), and iv).
Similarly to the conventions of \cite[Equation (8.3)]{KarelFourier}, we denote this inner sum by $\beta_m(Q,h)$\index{$\beta_m$}.
Note that there are {\it a priori} $\ell^m$ different values for $c_3$ and ${\ell^m}^{(4+4)}$ different values for $W$. Also note that if $h$ is diagonal, then $\mr{tr}(W,h)=0$.
By similar reasoning to \cite[p.~ 192]{KarelFourier}, we are reduced to considering the case where $h$ is the diagonal matrix $[d_1,d_2,d_3]$ with $0 \neq d_3 \mid d_2 \mid d_1$.\index{$d_1, d_2, d_3$}  We rewrite $\beta_m$ as
$$ \beta_m(Q,h) = \ell^{-9m} \sum_{c_3 \bmod \ell^{2m}} \omega_m(c_3d_3)C(c_3)$$
  where $C(c_3)$\index{$C(c_3)$} is the number of matrices $W$ modulo $\ell^{2m}$ which satisfy the conditions above for a given $c_3$. The factor $\ell^{-9m}$ in front accounts for passing from  $W \bmod \ell^{m}$ and $c_3 \bmod \ell^m$ to $W\bmod \ell^{2m}$ and $c_3 \bmod \ell^{2m}$.

We rewrite $C(c_3)$ as 
$$ C(c_3)= \left(\ell^{-2m}\sum_{\lambda=1}^{\ell^{2m}} \omega_{2m} (\lambda(c_3N_2(Q)+\mr{tr}(W^{\#},Q)) \right)\left(\ell^{-6m}\sum_{H}\omega_{2m} ({\mr{tr}(H,Q')}) \right),$$
where the sum is matrices $H$ modulo $\ell^{2m}$ of the type  $ \left(  
\begin{array}{ccc}
h_1 & h_3 & 0 \\
h^*_3 & h_2 &0 \\
0&0 & 0
\end{array} 
 \right)$ which vanish modulo $\ell$. The first sum is to keep track of Condition {\it ii)}, and the second is for Condition {\it iv)}.

We write $Q'$ as 
\begin{align*}
Q'=W^{\#}+c_3 \tilde{Q},
\end{align*} for \index{$\tilde{Q}$}$\tilde{Q}= \left(  
\begin{array}{ccc}
c_2 & -a_3 & 0 \\
-a_3^* &c_1 & 0 \\
0& 0 & 0
\end{array} 
 \right) $.
 Then
 \begin{align*}
 \beta_m(Q,h) = & \ell^{-17m}\sum_{c_3, H, \lambda} \omega_m(c_3d_3)\omega_{2m}({\mr{tr}(H,W^{\#}+c_3  \tilde{Q})}\omega_{2m} (\lambda(c_3 N_2(Q)+\mr{tr}(W^{\#},Q))) \\
 = & \ell^{-17m}\sum_{c_3, H, \lambda} \omega_m(c_3(d_3+\mr{tr}(H,\tilde{Q})+\lambda N_2(Q)) \omega_{2m}({\mr{tr}(H,W^{\#})})\omega_{2m} (\lambda \mr{tr}(W^{\#},Q)) \\
 = & \ell^{-15m}\sum_{ H, \lambda}  \omega_{2m}({\mr{tr}(W^{\#},H+\lambda Q)}),
 \end{align*}
 with the sum over $c_3$ modulo $\ell^{2m}$ and $H$ and $\lambda$ such that \begin{align}\label{eq:condition}
 \ell^m d_3+\mr{tr}(H,\tilde{Q})+\lambda N_2(Q) = \ell^m d_3 +h_1c_2+h_2c_1+\lambda c_1c_2 \equiv 0 \bmod \ell^{2m},
 \end{align} while still requiring $W' \equiv 0 \bmod \ell^m$, due to Condition iii).

Let $\tau=\mr{inf}(m,\mr{val}_\ell(d_3))$\index{$\tau$} and \index{$\zeta$}$\zeta=\mr{inf}(\mr{val}_\ell(Q),m)$, where $\val_\ell$ of a matrix denotes the minimum valuation of its entries.  Let 
\begin{align}\label{equ:xidefn}\index{$\xi=\xi(m, h, Q)$}
\xi =\Xi(m, h, Q):= \mr{inf}(m,\mr{val}_\ell(N_2(Q))-m, \tau, \zeta).
\end{align}
 Diagonalizing $Q=\left(  
\begin{array}{ccc}
c_1 & 0 & 0 \\
0& c_2 &0 \\
0&0 & 0
\end{array} 
 \right)$ so that $c_1\divides c_2$, we have that $\zeta=\mr{val}_\ell(c_1)$.  Inspired by \cite[Sections 9 and 10]{KarelFourier}, we break our computation of $\beta_m(Q, h)$ into two cases: $\zeta>\tau$ and $\zeta\leq \tau$.

\begin{itemize}
\item[($\zeta > \tau$)] Congruence \eqref{eq:condition} tells us that we must have $ \xi \leq \tau= \mr{val}_\ell(d_3)$, since otherwise the sum is over an empty set. Hence $\mr{val}_\ell(c_2)=(m+\xi)-\zeta \geq \zeta$, or equivalently $m+\xi \geq 2 \zeta$. 
The quantity $\mr{tr}(H,\tilde{Q})$ is divisible by $\ell^{m+\zeta}$, and by Congruence \eqref{eq:condition}, $\ell^m d_3+ \lambda N_2(Q)$ must also be divisible by $\ell^{m+\zeta}$. But since we are currently in the case where $\tau < \zeta$, we have that $\ell^md_3$ is not congruent to $0$ modulo $\ell^{m+\zeta}$, and so $\mr{val}_\ell(h)+m+\xi = m+ \tau$ and $\mr{val}_\ell(\lambda)=\tau - \xi$. 
Diagonalizing $H$, we get $\mr{val}_\ell(h_1+\lambda c_1)=\mr{val}_\ell(\lambda c_1)$, because $\mr{val}_\ell(h_1) \geq m$ and $\mr{val}_\ell(\lambda c_1)= (\tau - \xi)+\zeta < 2 \zeta - \xi \leq m + \xi -\xi$. Similarly $\mr{val}_\ell(h_2 + \lambda c_2)=\mr{val}_\ell(\lambda c_1)$. Hence, we obtain
\begin{align*}
\mr{val}_\ell(N(H+\lambda Q))=\mr{val}_\ell(\lambda c_1c_2)=& 2\tau - 2\xi + \zeta + (m+\xi)-\zeta = 2\tau + m - \xi \\
\mr{val}_\ell(H+\lambda Q)=&\tau + \zeta - \xi.
\end{align*} 
We write $$L = \ell^{-(\tau + \zeta - \xi)} (H+\lambda Q) =  \left(  
\begin{array}{ccc}
\lambda_1 & \lambda_3 & 0 \\
\lambda^*_3 & \lambda_2 &0 \\
0&0 & 0
\end{array} 
 \right),$$
with $\lambda_1 \not\equiv 0 \bmod \ell$. 
We calculate
\begin{align*}
\mr{tr}(W^{\#},H+\lambda Q) & =-\ell^{\tau + \zeta - \xi}(\lambda_1 n(a_1) + \lambda_2 n(a_2)-\mr{tr}(\lambda_3a_2a_1)) \\
&= -\ell^{\tau + \zeta - \xi}(\lambda_1 n(a_2 - \lambda_1^{-1}a_2 \lambda_3) + \lambda_1^{-1}N(L)n(a_2)).
\end{align*}
Note that after diagonalizing $Q$, the condition $W' \equiv 0 \bmod \ell^m$ reduces to $c_2a_2 \equiv c_1a_1^* \equiv 0 \bmod \ell^m$, or equivalently 
$ a_1 \equiv 0 \bmod \ell^{m- \zeta} $ and $a_2 \equiv 0 \bmod \ell^{\zeta-\xi}$.
Using $\mr{val}_\ell(N(L))=m -2\zeta+\xi$ and Equation \eqref{lem1.1} of Lemma \ref{lem1}, we obtain
\begin{align}\label{alpha&alpha'}
 \beta_m(Q,h) = & \ell^{-15m}\sum_{H,\lambda} \sum_{a_1 \equiv 0 \bmod \ell^{m- \zeta}, a_2 \equiv 0 \bmod \ell^{\zeta-\xi}} \omega_{2m-(\tau + \zeta - \xi)}(N(L)n(a_2)+n(a_1)) \\
 = &   \ell^{-15m}\sum_{H,\lambda} \ell^{4m}[\ell^{m-\zeta+\tau}]^2_{2(m-\zeta+\xi)} \ell^{4m}[\ell^{\zeta+\tau-\xi}]^2_{2\zeta}.
 \end{align}

Note that $[\ell^{m-\zeta+\tau}]_{2(m-\zeta+\xi)}=\ell^{m-\zeta+\tau}$ if and only if $2(m-\zeta+\xi) \geq m-\zeta+\tau$ or equivalently $m \geq \zeta +\tau -2\xi$.  But we have already $m \geq  \zeta + (\tau -\xi) $.

Also note that $[\ell^{\zeta+\tau-\xi}]_{2\zeta}=\ell^{\zeta+\tau-\xi}$ because of $\zeta > \tau$, so 
\[
2 \zeta > \zeta+\tau \geq \zeta+\tau-\xi.
\]
 For each fixed $H$, we are summing $[N_2(Q)]_{2m}=\ell^{m+\xi}$ choices for $\lambda$, and we have $\ell^{(4+2)m}$ choices for $H$, hence
\begin{align*}
 \beta_m(Q,h) = & \ell^{-7m}\sum_{H,\lambda} \ell^{2(m-\zeta+\tau +\zeta +\tau -\xi) } \\
 =& \ell^{-7m}\sum_{H,\lambda} \ell^{2m +4\tau -2\xi } \\
 =& \ell^{-7m} \ell^{7m + \xi}\ell^{2m +4\tau -2\xi } \\
= & \ell^{2m-\xi} \ell^{4 \mr{val}_\ell(h)}. 
 \end{align*}
 
\item[($\zeta \leq \tau$)]
We suppose now that $\tau=\mr{inf}(m,\mr{val}_\ell(d_3)) \geq \zeta=\mr{inf}(m,\mr{val}_\ell(Q))$ and that $\mr{val}_\ell(Q)=\mr{val}_\ell(c_1)$.  Given a pair $(\lambda,H_0)$ satisfying Congruence \eqref{eq:condition}, we have that $(\lambda,H)$ also satisfies Congruence \eqref{eq:condition}, for all $H:=H_0+h''U_2$ with $ h''c_1 \equiv 0 \bmod \ell^{2m}$ and $U_2 =  \left(  
\begin{array}{ccc}
0 & 0 & 0 \\
0 &1&0 \\
0& 0&0
\end{array} 
 \right)$.
We write 
\begin{align}\label{eq:trace}
 \mr{tr}(W^{\#},H+\lambda Q)= h'' n(a_2) + \ldots 
\end{align} where the rest of the right hand side doesn't depend on $h''$. 
 Note that $\sum_{ h''}  \omega_{2m}(h'' n(a_2))$ doesn't vanish if and only if $n(a_2) \equiv 0 \bmod c_1=\ell^{\zeta}$.  Hence, it follows from Congruence \eqref{eq:condition} that 
\begin{align}\label{eq:condition'}
 - c_1 h_2  n(a_2) \equiv (\ell^m d_3 + h_1c_2 + \lambda c_1 c_2) n(a_2) \bmod \ell^{2m+\zeta}.
 \end{align}
  Once we fix $\lambda$ and $h_1 \equiv 0 \bmod \ell^m$, there are either $\ell^{\zeta}$ or $0$ possible choices for $h_2$, depending on whether $\lambda \equiv 0 \bmod \ell^{\zeta-\xi}$ or not (as $d_3 \equiv 0 \bmod \ell^{\zeta}$).  
 If we multiply each side of Congruence \eqref{eq:condition} by $c_1$, 
 and then substitute the term with $c_1h_2n(a_2)$ with the right hand side of Congruence \eqref{eq:condition'}, we get that  \begin{align*}
 c_1\mr{tr}(W^{\#},H+\lambda Q) & \equiv c_1(-n(a_1)(h_1+\lambda c_1) -n(a_2)(h_2+\lambda c_2)+\mr{tr}(h_3 a_1^* a_2))
 \\
 & \equiv  n(a_2 )\left(\ell^m d_3 + h_1c_2 + \lambda c_1 c_2\right) -c_1n(a_1)h_1 -\lambda c_1^2n(a_1)-\lambda c_2c_1 n(a_2)  +c_1 \mr{tr}(h_3 a_1^* a_2)
\\
& \equiv   \ell^m d_3n(a_2 ) +h_1(c_2 n(a_2) -c_1n(a_1))-\lambda c_1^2n(a_1)+c_1 \mr{tr}(h_3 a_1^* a_2)
\end{align*}  doesn't depend on $h_2$.
So we get 
\begin{align*}
\sum_{\lambda,H,W}  \omega_{2m+\zeta}({c_1\mr{tr}(W^{\#},H+\lambda Q)}) & = \ell^{7m+\xi}\sum_{a_1,a_2}  \omega_{m+\zeta}(d_3 n(a_2)),
\end{align*}
where the conditions on $a_1$ and $a_2$  are such that
\begin{itemize}
\item $c_1 a_1 \equiv c_2 a_2 \equiv 0 \bmod \ell^m$ (from before);
\item $c_1^2 n(a_1) \equiv 0 \bmod \ell^{2m+\zeta}$;
\item $a_1^* a_2 \equiv 0 \bmod \ell^m$;
\item $c_1n(a_2) \equiv c_2n(a_2) \bmod \ell^{m+\zeta}$.
\end{itemize}
The power $\ell^{7m+\xi}$ comes from summing over all admissible $H$ and $\lambda$ modulo $\ell^{2m}$  under the conditions:  $h_1 \equiv 0 \bmod \ell^m$ (so $\ell^m$ choices); $\lambda \equiv 0 \bmod \ell^{\zeta-\xi}$ (so $\ell^{2m-\zeta+\xi}) $;  $h_3 \equiv 0$ (so $\ell^{4m}$), and we have only $\ell^\zeta$ choices for $h_2$ as noted above. 

Before substituting this into the expression for $\beta_m(Q,h')$, we make a change of variables $a_2 \mapsto \ell^{-m}c_2 a_2$ and $a_1\mapsto \ell^{-m}c_2 a_1$ (and recall that $\mr{val}_\ell(c_2)=\xi+m-\zeta$). We get 
\begin{align*}
\beta_m(Q,h')=\ell^{-8m+\xi}\sum_{a_2 \bmod \ell^{2m+\xi -\zeta}} \omega_{m+2\xi-\zeta}(d_3 n(a_2)) A(a_2),
\end{align*}
where $A(a_2)$\index{$A(a_2)$} is the number of $a_1 $ modulo $\ell^{m+\zeta}$ for which $a_1^* a_2 \equiv n(a_1) \equiv 0 \bmod \ell^{\xi}$.
To calculate $A(a_2)$, let $f=f(a_2)$ be the minimum between the $l$-adic valuation of $a_2$ and $\ell^{-\xi}n(a_2)$. By Equation \eqref{lem1.3} of Lemma \ref{lem1} (and \cite[Formula (II)]{KarelFourier}) we get 
\[
A\left(a_2\right) = \ell^{4((m+\zeta)-\xi)+2\xi}\left( \sum_{j=0}^f \ell^j -  \sum_{j=1}^{f} \ell^{j-2} \right).
\]
So 
\begin{align*}
\ell^{4m+\xi-4\zeta}\beta_m(Q,h) &  = \sum_{j=0}^{\xi}\ell^{j} \sum_{f=j}^{\xi} \sigma_f -  \sum_{j=1}^{\xi}\ell^{j-2} \sum_{f=j}^{\xi} \sigma_f,\\
&=\sum_{j=0}^{\xi}\ell^{j} \sigma'_j -  \sum_{j=1}^{\xi}\ell^{j-2} \sigma'_j,
\end{align*}
for \begin{align*}
\sigma_f = & \sum_{a_2 \; s.t. \; f(a_2)=f}\omega_{m+2\xi-\zeta}(d_3 n(a_2)),\\
\sigma'_j= & \sum_{f=j}^{\xi} \sigma_f= \sum_{a_2 \; s.t. \; f(a_2)\geq j}\omega_{m+2\xi-\zeta}(d_3 n(a_2)).
\end{align*}
Note that each $a_2$ in the sum for $\sigma'_j$ is divisible by $\ell^j$, hence
\begin{align*}
\sigma_j'= &\ell^{4j}\sum_{a_2\bmod \ell^{2m+\xi -\zeta-2j}}\omega_{m+2\xi-\zeta-2j}(d_3 n(a_2))\\
 = & \ell^{5j-\xi}\sum_{\lambda=1}^{\ell^{\xi-j}}\sum_{a_2 \bmod \ell^{2m+\xi-\zeta-2j}}\omega_{m+2\xi-\zeta-2j}((d_3+\lambda \ell^{m+\xi-\zeta-j}) n(a_2)).
\end{align*}
 \underline{If $d_3 \equiv 0 \bmod \ell^{m+\xi-\zeta-j}$} then $d_3+\lambda \ell^{m+\xi-\zeta-j}$ and $\lambda \ell^{m+\xi-\zeta-j}$ range over the same set, hence we can change the order of $\omega$ (and multiply by $\ell^{4j}$ to reduce the range of $a_2$), and we obtain
\begin{align*}
\sigma_j'= & \ell^{5j-\xi} \sum_{\lambda=1}^{\ell^{\xi-j}}\sum_{a_2 \bmod \ell^{2m+\xi-\zeta-2j}}\omega_{\xi-j}(\lambda n(a_2))\\
 =& \ell^{5j-\xi+4(2m+\xi -\zeta-2j -(\xi -j))}\sum_{\lambda=1}^{\ell^{\xi-j}}\sum_{a_2 \bmod \ell^{\xi-j}}\omega_{\xi-j}(\lambda n(a_2))\\
 =& \ell^{8m-4\zeta-\xi+j+2(\xi-j)}\sum_{\lambda=1}^{\ell^{\xi-j}}[\lambda]^2_{\xi-j}\\
 =& \ell^{8m-4\zeta+\xi-j}\sum_{k=0}^{\xi-j}\ell^{2k} \sum_{\lambda, [\lambda]_{\xi-j}=\ell^k} 1\\
 =&  \ell^{8m-4\zeta+\xi-j}\left(\ell^{2\xi-2j} + \sum_{k=0}^{\xi-j-1}\ell^{2k} (\ell-1)\ell^{(\xi-j)-k-1} \right)\\
 =&  \ell^{8m-4\zeta+\xi-j}\left(\ell^{\xi-j}\sum_{k=0}^{\xi-j}\ell^k  - \ell^{\xi-j-1}\sum_{k=0}^{\xi-j-1}\ell^{k}\right)\\
 =& \ell^{8m-4\zeta+\xi-j}(\mathbb{b}_{\xi-j}-\mathbb{b}_{\xi-j-1}),
\end{align*}
for $\mathbb{b}_{i}=\ell^{i}\sum_{k=0}^i \ell^k $.

 \underline{If $d_3 \not\equiv 0 \bmod \ell^{m+\xi-\zeta-j}$}, we simply get 
\begin{align*}
\sigma_j' = &\ell^{5j-\xi}\sum_{\lambda=1}^{\ell^{\xi-j}} \ell^{4(2m+\xi-\zeta-2j -(m+2\xi-\zeta-2j))} \ell^{2(m+2\xi-\zeta-2j)}[d_3]^2_{m+2\xi-\zeta-2j}\\
= &\ell^{5j-\xi}{\ell^{\xi-j}} \ell^{4m-4\xi} \ell^{2m+4\xi-2\zeta-4j}[d_3]^2_{m+2\xi-\zeta-2j}\\
= & \ell^{6m-2\zeta}[d_3]^2_{m+2\xi-\zeta-2j}.
\end{align*}
We let $m'=m-\zeta+\xi$.\index{$\m'$}  Recall that $[h]_m=\ell^{\mr{inf}(\mr{val}_\ell(d_3)=\tau,m)}$.  So if $d_3 \not\equiv 0 \bmod \ell^{m+\xi-\zeta-j}$ then $[d_3]_{m'-j}=[d_3]_{m+2\xi-\zeta-2j}$ for all $j$. We can hence find the following common expression in the two cases:
\begin{align*}
\sigma_j'=\ell^{6m-2\zeta}\left([d_3]^2_{m'-j}\sum_{k=0}^{\xi -j} \ell^{k} -[d_3]^2_{m'-1-j}\sum_{k=0}^{\xi -j-1} \ell^{k+1} \right).
\end{align*}
We plug this formula into the formula for $\beta_m$ and obtain
\begin{align*}
\ell^{m+\xi-4\zeta-(6m-2\zeta)}\beta_m(Q,h)  =&\sum_{j=0}^{\xi}\ell^{j} \left([d_3]^2_{m'-j}\sum_{k=0}^{\xi -j} \ell^{k} -[d_3]^2_{m'-1-j}\sum_{k=0}^{\xi -j-1} \ell^{k+1} \right) +\\
& - \sum_{j=1}^{\xi}\ell^{j-2} \left([d_3]^2_{m'-j}\sum_{k=0}^{\xi -j} \ell^{k} -[d_3]^2_{m'-1-j}\sum_{k=0}^{\xi -j-1} \ell^{k+1} \right),\\
\ell^{\xi-2\zeta-2m}\beta_m(Q,h)  = & [d_3]^2_{m'}\sum_{k=0}^{\xi} \ell^{k}-[d_3]^2_{m'-1}\sum_{k=1}^{\xi} \ell^{k-1}
\end{align*}
Indeed, when \underline{$d_3 \equiv 0 \bmod \ell^{m+\xi-\zeta-j}$}, we get
\begin{align*} 
&\ell^{2m'}\left( \sum_{j=0}^{\xi}\ell^{-j} \sum_{k=0}^{\xi -j} \ell^{k} - \sum_{j=0}^{\xi}\ell^{-j}\sum_{k=0}^{\xi -j-1} \ell^{k-1} \right. +\\
& \left.- \sum_{j=1}^{\xi}\ell^{-j-2} \sum_{k=0}^{\xi -j} \ell^{k} +\sum_{j=1}^{\xi}\ell^{-j-2}\sum_{k=0}^{\xi -j-1} \ell^{k-1} \right)=\\
=& \ell^{2m'}\left( \sum_{j=0}^{\xi}\ell^{-j} \sum_{k=0}^{\xi -j} \ell^{k} - \sum_{j=1}^{\xi}\ell^{-j}\sum_{k=0}^{\xi -j} \ell^{k} \right. +\\
& \left.- \sum_{j=1}^{\xi}\ell^{-j-2} \sum_{k=0}^{\xi -j} \ell^{k} +\sum_{j=2}^{\xi}\ell^{-j-2}\sum_{k=0}^{\xi -j} \ell^{k} \right).
\end{align*} 
So only the terms with $j=0$ in the first line and $j=1$ in the second line remain.
And \underline{if $d_3 \not\equiv 0 \bmod \ell^{m+\xi-\zeta-j}$} we replace all $[d_3]_{m'-j}$ by $[d_3]_{m'}$.
We re-expresse $\beta_m$ as
\begin{align}
\beta_m(Q,h)= & \ell^{2m}\left( [d_3]^2_{m'}\ell^{2\zeta-\xi}\sum_{k=0}^{\xi} \ell^{k} - [d_3]^2_{m'-1}\ell^{2\zeta-\xi} \sum_{k=0}^{\xi -1} \ell^{k-1} \right)\notag\\
= &\ell^{2m}\left( [d_3]^2_{m'}\ell^{2\zeta-\xi}\sum_{k=0}^{\xi} \ell^{k} - [d_3]^2_{m'-1}\ell^{2(\zeta-1)-(\xi-1)} \sum_{k=0}^{\xi -1} \ell^{k}\right) \notag\\
= & \ell^{2m}\left(\beta'_m(Q,h) - \beta'_{m-1}(\ell^{-1}Q,h)\right), \label{eq:beta}
\end{align}
with $\beta'_m(Q,h)$ defined to be the first sum in the parentheses.  (Note that that $\zeta$ for $m-1$ and $\ell^{-1}Q$ is $\zeta$ for $m,Q$ minus $1$; i.e. $\zeta (m-1, \ell^{-1}Q)=\zeta(m,Q)-1$ (similarly $\Xi (m-1,\ell^{-1}Q,\tau,\zeta(m-1, \ell^{-1}Q))=\Xi(m,Q,\tau,\zeta)-1$).)
\end{itemize}

We now uniformize the two cases.  That is, we address how to think about Equation \eqref{eq:beta} in the case where $\tau < \zeta$.
Note that $\beta'_m(Q,h)$ can be written as
\[\beta'_m(Q,h)= [d_3,c_1]_m^{2}[d_3,c_2]_m^{2}\sum_{k=0}^{\xi} \ell^{-k}, \]
where $[h,c_i]_m=\ell^{\mr{inf}(\mr{val}_\ell(d_3),\mr{val}_\ell(c_i),m)}$.
Hence,  \[
\Xi(m-1,h,\ell^{-1}Q) =\left\{ 
\begin{array}{cc}
\Xi(m,h,Q) & \mbox{ if }\tau < \mr{inf}(m,\zeta-m)\\
\Xi(m,h,Q) -1 & \mbox{ if }\tau \geq \mr{inf}(m,\zeta-m)
\end{array}. 
\right.
\]
So if $\tau < \mr{inf}(m,\zeta-m)$ then $\beta'_m(Q,h) =\beta'_{m-1}\left(\ell^{-1}Q,h\right)$, and $\beta_m$ vanishes; otherwise $\zeta$ does not appear in $\beta_m'$ and we get
\[ 
\beta_m(Q,h)=\left\{ 
\begin{array}{cc}
0 & \mbox{ if }\tau < \mr{inf}(m,\zeta-m)\\
\ell^{2m-\xi} \ell^{4 \mr{val}_\ell(h)} & \mbox{ if }\tau \geq \mr{inf}(m,\zeta-m)
\end{array}. 
\right.
\]
We define
\begin{align*}
\alpha_m(h) := \sum_{ N_2(Q) \equiv 0 \bmod \ell^m} \omega_m({\mr{tr}(Q,h)}) \beta_m(Q,h).
\end{align*} 
Since $\beta'_{m-1}(l^{-1}Q,h)=0$ if $Q \not\equiv 0 \bmod \ell$, we get 
\begin{align*}
 \sum_{ N_2(Q) \equiv 0 \bmod \ell^m} \omega_m({\mr{tr}(Q,h)}) \beta'_{m-1}(\ell^{-1}Q,h)=  \sum_{ N_2(Q) \equiv 0 \bmod \ell^{m-1}} \omega_{m-1}({\mr{tr}(Q,h)}) \beta'_{m-1}(Q,h).
\end{align*} 
Consequently, if we define 
$$\alpha'_m(h) :=   \sum_{ N_2(Q) \equiv 0 \bmod \ell^m} \omega_m({\mr{tr}(Q,h)}) \beta'_{m}(Q,h),$$
we get 
$$\alpha_m(h)=\ell^{2m} \left( \alpha'_m(h)  - \alpha'_{m-1}(h)  \right).$$
Telescopic summation gives us
\begin{align}\label{eq:twofactors}
 S_\ell^{(3)}(h)= (1- \chi(\ell)\ell^{-2r})(1- \chi(\ell)\ell^{2-2r})\sum_{m=0}^{\infty}\alpha'_m(h)\chi(\ell^m)\ell^{(2-2r)m}. 
\end{align}
We are left to calculate the internal sum.  We will now explain how to decompose it as a difference of the form $\ell^{2m}\left(\alpha''_m -\alpha''_{m-1}\right)$. This will give a polynomial times the factor $(1- \chi(\ell)\ell^{4-2r})$.

We begin by studying $\alpha'_m(h)$ when $m \leq \tau=\mr{val}_\ell(h)$. We have 
\begin{align*}
\alpha'_m(h)=& \sum_Q \omega_m({\mr{tr}(Q,h)}) \ell^{2m-2\zeta+2\xi+2\zeta-\xi}\sum_{k=0}^{\xi}\ell^k \\
=& \sum_{\xi=0}^{m}\ell^{2m} \mb{c}_{\xi}\ell^{\xi}\sum_{k=0}^{\xi}\ell^k,
\end{align*}
where $\mb{c}_\xi$ is the number of matrices $Q$ with $\Xi(m,Q)=\xi$. For each such $Q$, this implies that $Q$ is divisible by $\ell^{\xi}$, so $\Xi(m,Q)=\xi$ if and only if $\Xi(m-\xi,\ell^{-\xi}Q)=0$, so $N_2(Q)\equiv 0 \bmod \ell^{m-\xi}$, $Q \not\equiv 0 \bmod \ell$ or $N_2(Q) \not\equiv 0 \bmod \ell^{m-\xi+1}$. So we can write $\mb{c}_{\xi}=\mb{c}'_{m-\xi}-\mb{c}'_{m-\xi-1}$, where $\mb{c}'_{m-\xi}$ is the number of matrices $Q \bmod \ell^{m-\xi}$ with $N_2(Q)\equiv 0 \bmod \ell^{m-\xi}$.
This allows us to write 
$$\alpha'_m(h)=\ell^{2m} \left( \alpha''_m(h)  - \alpha''_{m-1}(h)  \right),$$
where $\alpha''_m(h)=\sum_{\xi =0}^m \mb{c}'_{m-\xi}\ell^{\xi}\sum_{k=0}^{\xi}\ell^k $. We want to show that $\alpha''_m(h)$  is equal to the quantity (denoted by the same symbol) $\alpha_m'(h'')$ defined in the rank 2 setting in Equation \eqref{equ:alphaprimem} of Section \ref{rank2}, where is  $h''$ the $2\times 2$-diagonal matrix $[d_1,d_2]$.

If $m > \tau=\mr{val}_\ell(h)$, we write \index{$\Psi_{j, \xi}$}$\Psi_{j,\xi}(h)=\sum_{Q \bmod \ell^m} \omega_m({\mr{tr}(Q,h)})$, where the sum ranges over the matrices $Q$ for which $\Xi(m,h,Q)$ has value $\xi$ and $j=\mr{inf}(m,\mathrm{val}_l(h),\mathrm{val}_l(Q))$.  We also write \index{$\Phi_{j, \xi}$}$\Phi_{j,\xi}(h) =\sum_{Q \bmod \ell^m} \omega_m({\mr{tr}(Q,h)})$, where the sum ranges over the matrices $Q$ for which $\Xi(m,h,Q) \geq \xi$ and $\mr{inf}(m,\mathrm{val}_l(h),\mathrm{val}_l(Q)) \geq j$. It follows from the definitions that $\Phi_{\xi-1,\xi}(h) = \Phi_{\xi,\xi}(h)$ (more generally $\Phi_{j,\xi}=\Phi_{\xi,\xi}$ if $j < \xi$).

We obtain 
\begin{align*}
\alpha'_m(h)  = & \sum_{ j=0}^{\tau}\ell^{2j} \sum_{\xi=0}^j \left(\sum_{i=0}^{\xi}\ell^{-i}\right) [d_3]^2_{m+\xi-j} \Psi_{j,\xi}(h) \\
\Psi_{j,\xi}(h)=& \Phi_{j,\xi}(h) - \Phi_{j,\xi+1}(h) -\Phi_{j+1,\xi}(h) + \Phi_{j+1,\xi+1}(h).
\end{align*}

We set {\index{$B_{j,\xi}$}$B_{j,\xi}:=   \left(\sum_{i=0}^{\xi}\ell^{-i}\right) [d_3]^2_{m+\xi-j}$.

We then have
\begin{align}\label{equ:alphamhtemp}
\alpha'_m(h)  = & \sum_{ j=0}^{\tau}\ell^{2j} \left( \sum_{\xi=0}^j B_{j,\xi}\Phi_{j,\xi}\right) -  \sum_{ j=0}^{\tau}\ell^{2j} \left( \sum_{\xi=0}^{j} B_{j,\xi}\Phi_{j,\xi+1}\right)\\
& - \sum_{ j=0}^{\tau}\ell^{2j} \left( \sum_{\xi=0}^j B_{j,\xi}\Phi_{j+1,\xi}\right) 
 + \sum_{ j=0}^{\tau}\ell^{2j} \left( \sum_{\xi=0}^{j} B_{j,\xi}\Phi_{j+1,\xi+1}\right).\nonumber
\end{align} 
 Replacing $j$ with $j-1$ in the third and fourth sum, and $\xi$ with $\xi-1$ in the second and fourth sum, and using $B_{j,-1}=0$ we get 
\begin{align*}
\alpha'_m(h)  =&  \sum_{j=0}^{\tau}\ell^{2j} \sum_{\xi=0}^j \left( B_{j,\xi}-B_{j,\xi-1}\right)\Phi_{j,\xi}  + \\
& -  \sum_{j=0}^{\tau}\ell^{2j} B_{j,j}\Phi_{j,j+1} + \\
& + \sum_{ j=1}^{\tau+1}\ell^{2j-2} \sum_{\xi=0}^j \left( B_{j-1,\xi}-B_{j-1,\xi-1}\right)\Phi_{j,\xi}+ \\
& +\sum_{j=0}^{\tau} \ell^{2j}B_{j,j}\Phi_{j+1,j+1}.
\end{align*} 
We can cancel the second and fourth terms  using $\Phi_{\xi-1,\xi} = \Phi_{\xi,\xi}$ to get
\begin{align}
\alpha'_m(h)  =  \sum_{ j=0}^{\tau}\ell^{2j} \sum_{\xi=0}^j \left( B_{j,\xi}-B_{j,\xi-1}\right)\Phi_{j,\xi}  - \sum_{ j=1}^{\tau+1}\ell^{2j-2} \sum_{\xi=0}^j \left( B_{j-1,\xi}-B_{j-1,\xi-1}\right)\Phi_{j,\xi}.
\end{align}
Note that $\mr{inf}(m,\mathrm{val}_l(h),\mathrm{val}_l(Q)) \leq \tau=\mathrm{val}_l(h)$ so $\Phi_{\tau+1,\xi}=0$ as $\mr{inf}(m,\mathrm{val}_l(h),\mathrm{val}_l(Q)) $ cannot be larger than $\tau$. Hence all the terms with $j=\tau+1$ are $0$.
Replacing $j$ with $j-1$ in the third and fourth sum, and $\xi$ with $\xi-1$ in the second and fourth sum, we get 
\begin{align}\label{equ:Binalphamprimetemp}
\alpha'_m(h)  =  \sum_{ j=0}^{\tau}\ell^{2j} \sum_{\xi=0}^j \left( B_{j,\xi}-B_{j,\xi-1}\right)\Phi_{j,\xi}  - \sum_{ j=1}^{\tau}\ell^{2j-2} \sum_{\xi=0}^j \left( B_{j-1,\xi}-B_{j-1,\xi-1}\right)\Phi_{j,\xi}.
\end{align} 
Whenever $\xi \leq j-m+\tau$, we have $[d_3]_{m+\xi-j}=\ell^{m+\xi-j}$, and we have $[d_3]_{m+\xi-j}=\ell^{\tau}$ otherwise.  So 
\begin{itemize}
\item $B_{j,\xi} - B_{j,\xi-1} = \ell^{-2}(B_{j-1,\xi} - B_{j-1,\xi-1} )$ if $\xi <  j-m+\tau $;
 \item $B_{j,j-m+\tau-1}=\ell^{-2}B_{j-1,j-m+\tau-1}$ if $\xi =  j-m+\tau $; 
 \item and $B_{j,\xi} - B_{j,\xi-1} = \ell^{2\tau-\xi}$ if $\xi >  j-m+\tau $.
\end{itemize}
So we rewrite $\alpha'_m(h)$ as 
\begin{align}\label{eq:rk3alpha}
\alpha'_m(h)  = & \ell^{2\tau}\left( \sum_{ j=0}^{\tau}\ell^{2j}  \sum_{\xi=\mr{sup}(0,j-m+\tau)}^j \ell^{-\xi} \Phi_{j,\xi} - \sum_{ j=1}^{\tau}\ell^{2j-2}  \sum_{\xi=\mr{sup}(0,j-m+\tau)}^j \ell^{-\xi} \Phi_{j,\xi} \right.\\ \notag
& + \sum_{ j=0}^{\tau}\ell^{2j} \left(\sum_{i=0}^{\mr{sup}(0,j-m+\tau)-1}\ell^{-i}\right) \Phi_{j,\mr{sup}(0,j-m+\tau)} +\\ 
& -  \left. \sum_{ j=1}^{\tau}\ell^{2j-2}  \left(\sum_{i=0}^{\mr{sup}(0,j-m+\tau)-1} \ell^{-i}\right) \Phi_{j,\mr{sup}(0,j-m+\tau)}\right). \notag
\end{align} 

When $\xi <  j-m+\tau $ the rest cancels out:
\begin{align*}
& \ell^{2j}\left( B_{j,\xi}-B_{j,\xi-1}\right)\Phi_{j,\xi}  - \ell^{2j-2} \left( B_{j-1,\xi}-B_{j-1,\xi-1}\right)\Phi_{j,\xi}=\\
=&   \ell^{2j} \left( B_{j,\xi}-B_{j,\xi-1} -\ell^{-2} \left( B_{j-1,\xi}-B_{j-1,\xi-1}\right) \right)\Phi_{j,\xi} =0.
\end{align*}
For $\xi =  j-m+\tau $ we use 
\begin{align*}
& [d_3]_{m+\tau-m}=\ell^2[d_3]_{m+\tau-m-1}=[d_3]_{m+\tau-m+1}= [d_3]_{m+\tau-m},\\
&\ell^{2j}(B_{j,j-m+\tau} -B_{j,j-m+\tau-1}) -\ell^{2j-2}(B_{j-1,j-m+\tau} -B_{j-1,j-m+\tau-1})=  \\
&=\ell^{2\tau}\left( \sum_{i=0}^{j-m+\tau}\ell^{2j-i} -\sum_{i=0}^{j-m+\tau-1}\ell^{2j-i-2}  - \sum_{i=0}^{j-m+\tau}\ell^{2j-i-2} + \sum_{i=0}^{j-m+\tau-1}\ell^{2j-i-2}\right)\\
&=\ell^{2\tau} \left( \ell^{2j- (j-m+\tau)} + \ell^{2j}\left(\sum_{i=0}^{\mr{sup}(0,j-m+\tau)-1}\ell^{-i}\right) - \ell^{2j-2} \left(\sum_{i=0}^{\mr{sup}(0,j-m+\tau)-1}\ell^{-i}\right) - \ell^{2j-2- (j-m+\tau)}\right) .
\end{align*}

We calculate $\Phi_{j,\xi}$:
\begin{align*}
\Phi_{j,\xi}(h) = & \sum_{Q \bmod \ell^m} \omega_m({\mr{tr}(Q,h)}), \left(\mbox{sum over } Q \equiv 0 \bmod \ell^j, N_2(Q) \equiv 0 \bmod \ell^{\xi +m} \right)\\
= & \sum_{Q \bmod \ell^{m-j}} \omega_{m-j}({\mr{tr}(Q,h)}), \left(\mbox{sum over }  N_2(Q) \equiv 0 \bmod \ell^{\xi +m-2j} \right)\\
 = & \ell^{6(j-\xi)}\sum_{Q \bmod \ell^{\xi +m-2j}} \omega_{\xi +m-2j}({\mr{tr}(Q,\ell^{\xi-j}h)}), \left(\mbox{sum over }  N_2(Q) \equiv 0 \bmod \ell^{\xi +m-2j} \right).
\end{align*}
This sum is nonzero only when ${\mr{tr}(Q,\ell^{\xi-j}h)} \equiv 0 \bmod \ell^{\xi +m-2j}$. We find again that $\Phi_{j,\xi}(h)$ is $$\ell^{6(j-\xi)}\alpha'_{m-2j+\xi}(\ell^{\xi-j}h'')$$  where this $\alpha'_{m-2j+\xi}$ is the one of Section \ref{rank2} and $h''$ the $2 \times 2$-diagonal matrix $[d_1,d_2]$.
So we can write 
\begin{align*}
\Phi_{j,\xi}(h) = & \ell^{6(j-\xi)}l^{2(m-2j+\xi)}\left( \alpha''_{m-2j+\xi}(\ell^{\xi-j}h'') -  \alpha''_{m-1-2j+\xi}(\ell^{\xi-j}h'') \right)\\
= & \ell^{2m-4\xi+2j}\left( \alpha''_{m-2j+\xi}(\ell^{\xi-j}h'') -  \alpha''_{m-1-2j+\xi}(\ell^{\xi-j}h'') \right)
\end{align*}
We claim that 
\begin{align*}
\Phi_{j,\xi}(h) = & \ell^{2m}\left( \alpha'''_{m}(h) -  \alpha'''_{m-1}(h) \right),
\end{align*}
where
\begin{align*}
 \alpha'''_{m}(h) = & \ell^{2\tau} \left( \sum_{j=0}^{\tau}\ell^{4j}\sum_{\xi=\mr{sup}(0,j-m+\tau)}^j \ell^{-5\xi}\alpha''_{m-2j+\xi}(\ell^{\xi-j}h'') + \right. \\
 & \left.- \sum_{j=1}^{\tau}\ell^{4j-2}\sum_{\xi=\mr{sup}(0,j-m+\tau)}^{j-1} \ell^{-5\xi}\alpha''_{m-2j+\xi}(\ell^{\xi-j}h'') \right).
\end{align*}
We recognize the first line of Equation \eqref{eq:rk3alpha} in $\alpha'''_{m}(h)$. Note that if $m \geq 2\tau$, then  $m-2j+\xi <0$ unless $m=2\tau$ and $j=\tau$, but then $\alpha''_{2\tau-2\tau-1}=0$.
So the rest of Equation \eqref{eq:rk3alpha} is 
\begin{align*}
R= & l^{4\tau} \left( \sum_{j=0}^{\tau}\ell^{4j}\sum_{\xi=\mr{sup}(0,j-m+\tau)}^j \ell^{-5\xi}\alpha''_{m-1-2j+\xi}(\ell^{\xi-j}h'') + \right. \\
 & \left.- \sum_{j=1}^{\tau}\ell^{4j-2}\sum_{\xi=\mr{sup}(0,j-m+\tau)}^{j-1} \ell^{-5\xi}\alpha''_{m-2j+\xi-1}(\ell^{\xi-j}h'') \right.\\
 & \left. (l^2-1)\sum_{j=m-\tau-1}^{\tau}\ell^{4m-4\tau-2} \sum_{i=0}^{\mr{sup}(0,j-m+\tau)-1}\ell^{-i}(\alpha''_{\tau-j}(\ell^{\tau-m}h'') -\alpha''_{\tau-j-1}(\ell^{\tau-m}h'') )\right).
\end{align*}
If $m\geq 2\tau$ the third line vanishes as the sum is empty, so we get $\alpha'''_{m-1}(h)$.
If $\tau < m < 2\tau$, then $\mr{sup}(0,j-m+\tau)=\mr{sup}(0,j-m+1+\tau)=0$ when $j < m- \tau$,  and $\mr{sup}(0,j-m+\tau)=\mr{sup}(0,j-m+1+\tau)-1$ otherwise. So the difference is 
\begin{align*}
R- \alpha'''_{m-1}(h)=&  \ell^{4m-2\tau-2}(\ell^2-1)\left(\sum_{j=m-\tau-1}^{\tau} \alpha''_{\tau-j}(\ell^{\tau-m}h'')\sum_{i=0}^{\mr{sup}(0,j-m+\tau)-1}\ell^{-i} \right. + \\
& -\sum_{j=m-\tau-1}^{\tau} \alpha''_{\tau-j-1}(\ell^{\tau-m}h'')\sum_{i=0}^{\mr{sup}(0,j-m+\tau)-1} \ell^{-i}+ \\
& -\left. \sum_{j=m-\tau}^{\tau} \ell^{-(j-m+\tau)}\alpha''_{\tau-j-1}(\ell^{\tau-m}h)  \right).
\end{align*}
Note that in the first and second sum the term $j=m-\tau-1$ or $j=m-\tau$ are $0$ as the inner sum is empty. 
The last two lines summed (with a change of variables $j+1 \mapsto j $, using that $ \alpha''_{-1}(\ell^{\tau-m}h'')=0$) give 
\[
\sum_{j=m-\tau-1}^{\tau} \alpha''_{\tau-j}(\ell^{\tau-m}h'')\left( \sum_{i=0}^{\mr{sup}(0,j-1-m+\tau)-1} \ell^{-i} + \ell^{-(j-1-m-\tau)} \right)
\]

 which is the opposite of the first line.  So we conclude that 
$$ S_\ell^{(3)}(h)= (1- \chi(\ell)\ell^{-2r})(1- \chi(\ell)\ell^{2-2r})(1- \chi(\ell)\ell^{4-2r})\sum_{m=0}^{\infty}\alpha'''_m(h)\chi(\ell^m)\ell^{(4-2r)m}. $$
To finish, we need to show that $\alpha'''_m(h)$ vanishes for $m$ large enough. 
We first observe that $\alpha''_{n}(h'') = 0 $ unless $d_1d_2 \equiv 0 \bmod \ell^n$. But  ${m-2j+\xi} \geq m-j \geq m-\tau$. So for $m > \tau$ all the terms of $\alpha'''_m$ vanish, hence we obtain our claim.    

\subsection{Restriction from $G$ to $\GSp_6$}\label{sec:restrictionfromG}
To prove our main results about algebraicity of values of Spin $L$-functions, we will need to understand the form of the Fourier coefficients of the restriction to $\GSp_6$ of our Eisenstein series.  Proposition \ref{prop:restrictionfromG} describes a helpful relationship between the Fourier coefficients of a modular form on $G$ and the Fourier coefficients of its restriction to $\GSp_6$.

\begin{prop}\label{prop:restrictionfromG}
Let $F$ be a modular form on $G$ of level $\Gamma^0(M)$, and let $f$ be the restriction of $F$ to $\GSp_6\subset G$.  Then each Fourier coefficient of $f$ is a finite $\ZZ$-linear combination of Fourier coefficients of $F$.
\end{prop}

\begin{proof}
To prove the proposition, it suffices to prove that for any given $t \in M_3(\IQ)$, finitely many of the terms $\exp(2\pi i \mathrm{tr}(Z,h))$ for $h \in H_3(B)$ in the Fourier expansion of $F$ restrict to $\exp(2\pi i \mathrm{tr}(zt))$ arising in the Fourier expansion of the restriction of $f$.  We first show that when we restrict the domain of $\exp(2\pi i \mathrm{tr}(Z,h))$ from $Z \in \mathcal{H}$ to $z \in \mathcal{H}_3 \subseteq \mathcal{H}$, we obtain the Fourier coefficient corresponding to $\frac{1}{2}(h+{}^th) \in M_3(\IQ)$.  To see this, write
\[h = \left(\begin{array}{ccc} c_1 & a_{12} & a_{13} \\ a_{12} & c_2 & a_{23}\\ a_{13} & a_{23} & c_3 \end{array}\right)\text{ and }z = \left(\begin{array}{ccc} z_1 & z_{12} & z_{13} \\ z_{12} & z_2 & z_{23}\\ z_{13} & z_{23} & z_3 \end{array}\right).\]
Recall that $\mathrm{tr}(z,h) = \frac{1}{2}(zh+hz)$.  Using the fact that $z$ is stable under $B$-conjugation in $B_\IC$, we compute that the first diagonal entry of $zh+hz$ is
\[\frac{1}{2}\left(c_1z_1 + a_{12}z_{12} + a_{13}z_{13}+z_1c_1+z_{12}\overline{a}_{12}+z_{13}\overline{a}_{13}\right),\]
which matches the first diagonal entry of $\frac{1}{2}(h+{}^th)z$.  The same is true for the remaining diagonals.

Next observe that $\frac{1}{2}(h+{}^th)$ has the same diagonals as $h$, so for a given $t \in M_3(\IQ)$, it suffices to show that finitely many $h \in H_3(B)$ with the same diagonals as $t$ can contribute to a Fourier expansion of an automorphic form on $G$.  We now apply the fact that if $\exp(2\pi i \mathrm{tr}(Z,h))$ appears with a nonzero coefficient in such a Fourier expansion, then $h$ must belong to the positive semi-definite cone.  (For a reference for this, see \cite[\S 4]{KarelFourier}, which notes that $h$ must belong to the positive cone of squares in $H_3(B)$ as a Jordan algebra.  By the spectral theorem for Hermitian forms over quaternion algebras \cite{fp}, which says that such a form $X$ can be written as $U^*DU$ for $U^*U=\mathrm{Id}$ and $D$ real diagonal, this is just the cone of positive semi-definite forms.)

The property of being positive semi-definite is stable upon restriction to a subspace of the Hermitian space.  It follows that the three principal $2 \times 2$ minors of $h \in H_3(B)$ must also be positive semi-definite, and in particular that their norms are nonnegative.  Since the diagonals of $h$ are fixed, these nonnegativity conditions bound the norms of the off-diagonal entries of $h$.  Finally, since we are assuming that the Fourier coefficients belong to the Fourier expansion of a form of level $\Gamma^0(M)$, the entries of $h$ have denominators dividing $M$.  There are finitely many elements of $B$ with bounded norm and denominator bounded by $M$.
\end{proof}
 
As an immediate consequence of Proposition \ref{prop:restrictionfromG} and its proof, we obtain Corollary \ref{cor:restrictioncor}.
\begin{cor}\label{cor:restrictioncor}
Let $F$ be a modular form on $G$.  If $F$ has algebraic Fourier coefficients, then its restriction to $\GSp_6\subset G$ also has algebraic Fourier coefficients.  If the Fourier coefficients of $F$ are polynomials in $h\in H_3(B)$ with coefficients in $\ZZ$, then the Fourier coefficients of its restriction to $\GSp_6$ are also polynomials in $h\in H_3(B)$ corresponding to $t\in M_3(\IQ)$, with coefficients in $\ZZ$.
\end{cor}
Applying our result to Theorem \ref{thm:Eseriescoeffs}, we obtain Corollary \ref{cor:restrictionEseries}.
\begin{cor}\label{cor:restrictionEseries}
For $2r >10$, the Fourier coefficients of the restriction of $G_{2r, \chi}$ to $\GSp_6$ lie in $\bigcup_{j=0}^3\left(\IQ(\chi)g(\chi)^{-j}\right)\subset\IQ(\chi, g(\chi))$,
 and they are polynomials in $h\in H_3(B)$ corresponding to $t\in M_3(\IQ)$.  Similarly, the Fourier coefficients of the restriction of $E_{2r, \chi}$ to $\GSp_6$ lie in $\bigcup_{j=0}^3\left(\IQ(\chi)g(\chi)^{j}\right)\subset\IQ(\chi, g(\chi))$,
 and they are polynomials in $h\in H_3(B)$ corresponding to $t\in M_3(\IQ)$.\end{cor}

\subsection{Effect of weight-raising differential operators on Eisenstein series on G}\label{sec:diffop}
As we noted in Remark \ref{rmk:introducediffop}, Pollack introduced a differential operator $\mathcal{D}$ \cite[Section 7.3]{pollack1} such that
$\mc D^t  E^*_{2r,\chi}(Z,r) = E^*_{2r+2t,\chi}(Z,r)$
 and $\mc D^t  E_{2r,\chi}(Z,r) = E_{2r+2t,\chi}(Z,r)$ for all nonnegative integers $t$.  The operator $\mc D$ is a Maass--Shimura operator defined Lie theoretically as an element of the complexified universal enveloping algebra of the Lie algebra of $G(\IR)$, analogous to the Maass--Shimura operators on automorphic forms discussed in, for example, \cite{maass71, shimura81, shar, shimura-RS, harris3lemmas, hasv, harrismaass}.  In this section, we make some observations about the operator $\mc D$.
 
\subsubsection{Necessary algebraicity properties for restriction of Eisenstein series from $G$ to $\GSp_6$} 
Our approach to proving algebraicity of critical values of Spin $L$-functions is inspired by the strategy pioneered by Shimura in \cite[p. 794-796]{shimura-RS}, where analogously obtained Eisenstein series are employed in the proof of algebraicity of values of Rankin--Selberg convolutions of elliptic modular forms.  In particular, we will need Lemma \ref{lem:resnhE}.  

\begin{lem}\label{lem:resnhE}
For each pair of integers $r\geq s\geq 6$, the restriction of $E_{2r, \chi}(g, s) = \mc D^{r-s} E_{2s, \chi}(g, s)$ to $\GSp_6$ is nearly holomorphic and defined over $\IQ(\chi, g(\chi))$.
\end{lem}

Lemma \ref{lem:resnhE} follows from the observations that the action of the Maass--Shimura operator $\mc D$ on a modular form $f(Z)$ of a variable $Z=X+iY\in \mc H$ with Fourier coefficients in a $\IQ$-algebra $R$ is a polynomial in the entries of $Y^{-1}$ whose coefficients are Fourier expansions in $Z$ with Fourier coefficients in $R$, together with a similar argument to the proof of Proposition \ref{prop:restrictionfromG} concerning restriction from $G$ to $\GSp_6$.  In the context of Corollary \ref{coro:Eseriescoeffs}, it suffices to apply these observations to $f=E_{2r, \chi}$ and $R=\IQ(\chi, g(\chi))$.

For this paper, Lemma \ref{lem:resnhE} turns out to be all that we need concerning algebraic aspects of $\mc D$.  We do not need a theory of nearly holomorphic forms on $G$, nor do we need a precise formula for the action of $\mc D$ on our Eisenstein series.  Nevertheless, in Section \ref{sec:diffopstronger}, we record some stronger statements that might satisfy readers interested in this topic and that also might be valuable for future work with modular forms on $G$.

\subsubsection{More comprehensive treatment of $\mc D$ and its effect on modular forms on $G$}\label{sec:diffopstronger}
In our more comprehensive treatment, we begin by reformulating $\mc D$ in terms of coordinates on $\CH$, i.e. as acting on functions $f:\CH\rightarrow \CH$ satisfying the modularity property for $G$.

For each $Z \in  \mc H$, we write $Z=X+iY$ like in Section \ref{sec:Hss}.  Following the convention introduced in Equation \eqref{equ:ZinHcoords}, we write the coordinates of $Z \in  \mc H$ explicitly as
\[
Z=\left( \begin{array}{ccc}
Z_{11} & Z_3 & Z_2^*\\
Z_3^* & Z_{22} & Z_1\\
Z_2 &Z_1^* & Z_{33}
\end{array} \right),
\]
for $Z_i =\sum_{\mb{u} \in \set{\mb{1}, \mb{i}, \mb{j}, \mb{k}}} Z_i(\mb{u})\mb{u}$ and 
$Z_i^*=\sum_{\mb{u} \in \set{\mb{1}, \mb{i}, \mb{j}, \mb{k}}} Z^*_i(\mb{u})\mb{u}$.  By a mild abuse of notation (to keep track of the relation between $Z_i(\mb{u})$ and $Z^*_i(\mb{u})$), we shall write $\frac{\partial }{\partial Z_{i}(\mb{u})}$ for the differential operator

\begin{align*}
\frac{1}{2}\left(
\frac{\partial}{\partial Z_{i}(\mb{u})} + \frac{\partial }{\partial Z^*_{i}(\mb{u})} \right)& \mbox{ if } \mb{u}= \mb {1}\\
\frac{1}{2}\left(
\frac{\partial }{\partial Z_{i}(\mb{u})} - \frac{\partial }{\partial Z^*_{i}(\mb{u})} 
\right) & \mbox{ if } \mb{u} \neq \mb{1}.
\end{align*}

We define\index{$\Delta$}
\begin{align*}
\Delta =\frac{\partial }{\partial Z_{11}}\frac{\partial }{\partial Z_{22}}\frac{\partial }{\partial Z_{33}} -\left(\sum_{i=1,2,3}\sum_{\mb{u} \in \set{\mb{1}, \mb{i}, \mb{j}, \mb{k}}} \frac{\partial }{\partial Z_{11}} \frac{\partial^2 }{\partial Z_{i}(\mb{u})^2}  \right) + 2 \frac{\partial }{\partial Z_{1}(\mb{1})} \frac{\partial }{\partial Z_{2}(\mb{1})} \frac{\partial }{\partial Z_{3}(\mb{1})}+\\
-2\left( \sum_{i=1,2,3}\sum_{\mb{u} \in \set{\mb{i}, \mb{j}, \mb{k}}} \frac{\partial }{\partial Z_{i}(\mb{1})} \frac{\partial }{\partial Z_{i+1}(\mb{u})} \frac{\partial }{\partial Z_{i+2}(\mb{u})}  \right)-2\mr{det}\left( \frac{\partial }{\partial Z_{i}(\mb{u})}\right)_{i=1,2,3; \mb{u} \in \set{\mb{i}, \mb{j}, \mb{k}}}.
\end{align*}
Note that in analogue with the case of unitary and symplectic groups developed in \cite[Section 9]{sheseries},  $\Delta = N(\partial/\partial z)$, where
\begin{align*}
\frac{\partial}{\partial z} :=\begin{pmatrix} \frac{\partial}{\partial Z_{11}} & \frac{\partial}{\partial  Z_3} & \left( \frac{\partial}{\partial  Z_2}\right)^\ast\\
\left(\frac{\partial}{\partial Z_3}\right)^\ast & \frac{\partial}{\partial Z_{22}} &\frac{\partial}{\partial Z_1}\\
\frac{\partial}{\partial Z_2} & \left(\frac{\partial}{\partial Z_1}\right)^\ast & \frac{\partial}{\partial Z_{33}}\end{pmatrix},
\end{align*}
with $\frac{\partial}{\partial Z_i} :=\sum_{\mb{u} \in \set{\mb{1}, \mb{i}, \mb{j}, \mb{k}}} \frac{\partial}{\partial Z_i(\mb{u})}\mb{u}$ for $i=1, 2, 3$ and
$\left( \frac{\partial}{\partial Z_i}\right)^\ast:=\sum_{\mb{u} \in \set{\mb{1}, \mb{i}, \mb{j}, \mb{k}}} \frac{\partial}{\partial  Z_i(\mb{u})}\mb{u}^\ast$ for $i=1, 2, 3$.

\begin{defi}
We say that a real analytic function on $\CH$ is nearly holomorphic if it is modular in the sense of Definition \ref{def:mformsonH} and has a Fourier expansion of the form $R\llbracket q^h \rrbracket[1/Y_{ii},1/Y_{i}(\mb{u})]$ where we write $Z=X+iY$, and by a slight abuse of notation we write $Y^{-1}$ for $-2i\cdot(Z-\bar{Z})^{-1}$ and   $1/Y_{ii},1/Y_{i}(\mb{u})$ for the entries of $Y^{-1}$.
\end{defi}

 We have the following theorem.
 \begin{thm}\label{thm:diffOpq-exp}
 Let $E$ be a nearly holomorphic form for $G$ and $R \subset \m C$. If  $E$ has $q$-expansion in $R\llbracket q^h \rrbracket[1/Y_{ii},1/Y_{i}(\mb{u})]$, then $\mc D E$ has  has $q$-expansion in $R[\frac{1}{2}]\llbracket q^h \rrbracket[1/Y_{ii},1/Y_{i}(\mb{u})]$.
 \end{thm}
 \begin{proof}
We closely follow the strategy of \cite[\S 3.3]{CouPan} (see especially the calculation of \cite[\S 3.2]{court2000}).
The idea is to study $\mc D^s (e^{2 \pi i \mr{tr}(Z,h)}) $.  Indeed, using induction on the degree of nearly-holomorphicity  we need to calculate $\Delta( Y^{-1}_{ii} e^{2 \pi i \mr{tr}(Z,h)}) $ or $\Delta( Y^{-1}_{i}(\mb{u}) e^{2 \pi i \mr{tr}(Z,h)})$, where $Y^{-1}_{ii}$ denotes the $ii$-entry of $Y^{-1}$, and same for $Y^{-1}_{i}(\mb{u})$. We get immediately that 
\[
\Delta(e^{2 \pi i \mr{tr}(Z,h)})=N(h)e^{2 \pi i \mr{tr}(Z,h)}.
\] 
 Note that $Y^{-1}=-2i\cdot(Z-\bar{Z})^{-1}$.  Since $\Delta$ is holomorphic, we can treat $\bar{Z}$ as constant. 
 So 
\[
 \frac{\partial }{\partial Z_{11}}\left(Z^{-1}\right)= \frac{\partial }{\partial Z_{11}}\frac{Z^{\#}}{N(Z)}
\] 
 which is a polynomial in the entries of $Z$ and of $N(Z)^{-1}$ with integer coefficients. 
 Using induction and the formula for the derivative of a product of functions, we conclude the proof.
 \end{proof}

\begin{rmk}
A similar argument to the proof of Proposition \ref{prop:restrictionfromG} shows that the restriction to $\GSp_6$ of a nearly holomorphic modular form on $G$ defined over a ring $R$ is a nearly holomorphic modular form on $\GSp_6$ defined over $R$.  In particular, by Theorem \ref{thm:diffOpq-exp}, if $E$ is a holomorphic modular form on $G$ defined over a ring $R[1/2]$, then the restriction of $\mc D E$ to $\GSp_6$ is a nearly holomorphic on $\GSp_6$ defined over $R[1/2]$.
\end{rmk}

Finally, to emphasize the parallels with the setting for modular forms in coordinates, we note that the Maass--Shimura operator that raises the weight by $2$ is the operator $\Diffop_k$\index{$\Diffop_k$} on $\IC$-valued $\ci$-functions $f(Z)$ on $\CH$
\begin{align*}
\Diffop_k f (Z) = N(Y)^{-k} \Delta (N(Y)^k f).
\end{align*}

\begin{prop}\label{dopeffect}
For any $\ci$ function $f:\CH\rightarrow \IC$,
\begin{align*}
\Diffop_k(f |_k\alpha) = (\Diffop_k f)|_{k+2}\alpha.
\end{align*}
\end{prop}
\begin{proof}
The proof follows from the two equations from Lemma \ref{twokeyequs}, and it is similar to the proof of the first equation in \cite[Proposition 12.10(2)]{shar}.
\end{proof}

As an immediate corollary of Proposition \ref{dopeffect}, we obtain Corollary \ref{automorphypreserved}.
\begin{cor}\label{automorphypreserved}
If $f$ is a nearly holomorphic modular form of weight $k$ on $G$, then $\Diffop_k f$ is a nearly holomorphic modular form on $G$ of weight $k+2$.
\end{cor}

\begin{rmk}
Like in \cite{kaCM, hasv, EDiffOps, EFMV, ZLiuFour}, it is possible to formulate these differential operators geometrically, for example by building on \cite{MilneAA}.  In this paper, however, we will not need such a formulation.  Our approach builds instead on the strategy introduced for modular forms in \cite{shimura-RS} and extended to Siegel modular forms in \cite{court2000, CouPan}.
 \end{rmk}

\section{\texorpdfstring{Proof of main theorem about algebraicity of Spin $L$-functions for $\GSp_6$}{Main theorem about algebraicity of Spin L-functions for GSp6}}\label{sec:mainthm}
The main goal of this section is to prove Theorem  \ref{thm:algprecise} about the algebraicity of values of Spin $L$-functions for $\GSp_6$.  Section \ref{sec:integralrepn} provides an integral representation for the Spin $L$-function, which is employed in the proof of Theorem \ref{thm:algprecise} in Section \ref{sec:proofofalgthm}.

\subsection{Integral representation}\label{sec:integralrepn}
This section introduces an integral representation, i.e. a Rankin--Selberg-style integral, that we will employ in the proof of Theorem \ref{thm:algprecise} in Section \ref{sec:proofofalgthm}.  In \cite[Proposition 6.2 and Theorem 6.3, summarized on p. 1393]{pollack1}, Pollack gives an integral representation for the Spin $L$-function of a cuspidal representation associated to a genus $3$ Siegel modular form of weight $2r$ and level $1$.  This section extends that result to genus $3$ Siegel modular forms of level $\Gamma^0(M)$ for positive integers $M$ divisible by the conductor of the Dirichlet character $\chi$ introduced in Section \ref{sec:charchiofP}.

Following the convention of \cite[\S 6.1]{pollack1}, for $\lambda$ a scalar and $m$ in $\mathrm{GL}_3$, we consider the Levi of $\GSp_6$ whose elements are expressed as 
\begin{align*}
(\lambda,m)= \left(\begin{array}{c|c} \
\lambda \mr{det}(m)^t m^{-1}&\\  \hline & m  
\end{array}
\right) \in \mr{GSp}_6,
\end{align*}
where the matrix is written with respect to the symplectic basis $e_1, e_2, e_3, f_1, f_2, f_3\in W_6$ in Section \ref{sec:GSp6embedding}. 
We also define $U_P$\index{$U_P$} to be the unipotent of $\mr{GSp}_6$ for this Levi and $U_0$\index{$U_0$} to be the subgroup of $U_P$ consisting of elements $u$ such that $\mr{Tr}(Tu)=0$.  Following the approach of \cite[\S 3]{pollack1}, we choose a half-integral symmetric matrix $T \in M_3(\m Z)$  such that $T$ corresponds to a maximal order in $B$ (for details see {\it loc. cit.}).  Recall that $T$ is selected to compatible with the maximal order $B_0$.  We associate $T$ with the rank one matrix $A(T) \in H_3(B)$ defined in \cite[Definition 3.6]{pollack1} and set $f_\mathcal{O} = (0,0,A(T),0)$.  In \cite[Lemma 3.8]{pollack1}, Pollack defines a set of $3\times 3$ Hermitian matrices $J_T$, which turns out to be the lattice $J_0\subset H_3(B)$ defined in \cite[Expression (5.4)]{pollack1} to consist of matrices 
\begin{align*}
\begin{pmatrix}
c_1 & x_3 & x_2^\ast\\
x_3^\ast & c_2 & x_1\\
x_2 & x_1^\ast & c_3
\end{pmatrix}
\end{align*}
with $c_i\in\ZZ$ and $x_i\in B_0$.

Define
\[\Xi(m) = \begin{cases} 1 & m \in M_3(\widehat{\ZZ}) \textrm{ and }m^{-1}Tc(m) \in M_3(\widehat{\ZZ}) \\ 0 & \textrm{otherwise}\end{cases}\]
and
\[\Psi(\lambda) = \prod_v \Psi_v(\lambda),\]
where for $v|M$ we have
\[\Psi_v(\lambda) = \begin{cases} (1-v^{-1})|\lambda|_v & \lambda \in \ZZ_v \\ -1 & \lambda \in v^{-1}\ZZ_v \setminus \ZZ_v \\ 0 & \lambda \notin v^{-1}\ZZ_v \end{cases}\]
and for $v \nmid M$ we have
\[\Psi_v(\lambda) = \begin{cases} |\lambda|_v & \lambda \in \ZZ_v \\ 0 & \lambda \notin \ZZ_v \end{cases}.\]

\begin{lem}
We have
\[
\int_{U_0(\adeles^\infty) \setminus U_P(\mathbb{A}^\infty)} \psi(\mathrm{Tr}(Tu))\Phi(f_{\mathcal{O}}u(\lambda,m))du = \Xi(M^{-1}m)\Psi(\lambda),
\]
for $\psi$ an additive character and $\Phi=\otimes'_v \Phi_v$, for $\Phi_v$ the functions defined in Section \ref{sec:EisDef} to define the Eisenstein series.  
\end{lem}

\begin{proof}
For $v \ndivides M$, this is \cite[Proposition 6.1]{pollack1}, so we assume $v|M$.  As computed there, $f_\mathcal{O}u(\lambda,m) = (0,0,m^{-1}A(T)c(m),\lambda^{-1}\mathrm{Tr}(Tu))$.

We first claim that $m^{-1}A(T)c(m)$ belongs to $M J_0$ if and only if both $M \mid m$ and $m^{-1}Tc(m) \in MM_3(\widehat{\ZZ})$.  To see this, first observe that since $(M^{-1}m)^{-1} = Mm^{-1}$ and $c(M^{-1}m)=M^{-2}c(m)$, the expression $m^{-1}A(T)c(m)\in M J_0$ is equivalent to
\[(M^{-1}m)^{-1}A(T)c(M^{-1}m)\in J_0.\]
By applying \cite[Lemma 3.8]{pollack1} with $M^{-1}m$ in place of $m$, we deduce that $m^{-1}A(T)c(m)\in M J_0$ if and only if both $M^{-1}m \in M_3(\widehat{\ZZ})$ and $(M^{-1}m)^{-1}Tc(M^{-1}m) \in M_3(\widehat{\ZZ})$.  Thus the integral is zero unless $\Xi(m)=1$, and when $\Xi(m)=1$, the integral is equal to
\begin{equation} \label{eqn:intwhenximis1}\int_{U_0(\adeles^\infty) \setminus U_P(\adeles^\infty)} \psi(\mathrm{Tr}(Tu)) \mathrm{Char}(\lambda^{-1}\mathrm{Tr}(Tu) \in \ZZ_v^\times)du.\end{equation}
We can write
\[\mathrm{Char}(\lambda^{-1}\mathrm{Tr}(Tu) \in \ZZ_v^\times) = \mathrm{Char}(\lambda^{-1}\mathrm{Tr}(Tu) \in \ZZ_v) - \mathrm{Char}(\lambda^{-1}\mathrm{Tr}(Tu) \in v\ZZ_v),\]
breaking Equation \eqref{eqn:intwhenximis1} into the difference of two integrals.  The first is evaluated as $|\lambda|_v$ for $\lambda \in \ZZ_v$ and $0$ otherwise in the proof of \cite[Proposition 6.1]{pollack1}.  The second is equivalent to the first after the substitution $\lambda \mapsto v\lambda$, so it is equal to $v^{-1}|\lambda|_v$ whenever $\lambda \in v^{-1}\ZZ_v$ and $0$ when $\lambda \notin v^{-1}\ZZ_v$.  The lemma follows.
\end{proof}

Let $\phi_f$ denote the adelisation of the form $f$.
Consider the integral 
\begin{align*}\index{$I_{2r, \chi}(\phi_f, s)$}
I_{2r, \chi}(\phi_f,s)= \int_{\mr{GSp}_6(\m Q) Z(\m A)\setminus \mr{GSp}_6(\m A)} \phi_f(g)E^*_{2r,\chi}(g,s) \textup{d}g.
\end{align*}
If  $\chi$ is trivial or quadratic (i.e. takes values in $\{0, 1, -1\}$), this integral is well-defined, as the integrand is invariant by $Z(\m A)$ (since the action of $\chi$ is as in Section \ref{sec:charchiofP}).
This is connected to the Spin $L$-function, thanks to a result of Evdokimov \cite[Theorem 3]{Evdo} (see also \cite[Theorem 4.1]{pollack1} for the explicit calculation of the ``error term'') which expresses the Spin $L$-function as an infinite sum.

\begin{thm}[Theorem 4.1 of \cite{pollack1}, Theorem 3 of \cite{Evdo}]\label{thm:Evdo}
If $T$ corresponds to a maximal order, then the infinite sum
\[
\sum_{\tiny{
\begin{array}{c}
\lambda \geq 0, (\lambda,M)=1, \\ 
m \in M^+_3(\m Z)/\mr{SL}_3(\m Z), m^{-1}Tc(m) \in M_3({\m Z})
\end{array}}}
\frac{a(\lambda m^{-1}T c(m))\chi((\lambda, m))}{\lambda^{s}\mr{det}(m)^{s-2r+3}}\]
is equal to
\[
a(T) \frac{L^{(M)}(s,\pi \otimes \chi,\mr{Spin})}{L(2s-6r+6,\chi)L_{D_{B}}(2s-6r+8,\chi)},
\]
for $L^{(M)}(s,\pi \otimes \chi,\mr{Spin})$ the Spin $L$-function of $\pi \otimes \chi$ with all Euler factors at primes dividing  $M$ removed.
\end{thm}

Going forward, we assume the Fourier coefficient of $f$ at $h=1_3$ is $1$.
We also write \index{$M'$}$M'$ for the squarefree part of $M$.  When $M=1$, Theorem \ref{thm:intrepn} coincides with \cite[Proposition 6.2 and Theorem 6.3, summarized on p. 1393]{pollack1}.  We set 
\begin{align}
\G_{\m C}(s) &:= 2(2\pi)^{-s}\Gamma(s)\nonumber\index{$\G_{\m C}$}\\
\G(s,\mr{Spin})&:= \G_{\m C}(s+r-4) \G_{\m C}(s+r-3) \G_{\m C}(s+r-2) \G_{\m C}(s+3r-5)\label{equ:GammaSpin}\index{$\G(s,\mr{Spin})$}
\end{align}

\begin{thm}\label{thm:intrepn}
Let $S$ be an integer whose support is contained in the support of $M$.  Suppose that $f$ is an eigenvector for the \index{$U_S$}$U_S$ operator

\[
f=\sum_{h}a(h)q^h \mapsto \sum_{h}a(Sh)q^{h}
\]
with eigenvalues $a(S)$. 
Suppose $\chi$ is trivial or quadratic.
 Then the integral $I_{2r, \chi}(\phi_f,s)$ unfolds  to give us
\[
I_{2r, \chi}(\phi_f,s)=\mathrm{det}(T)^3 a(T)\G(s,\mr{Spin}) L_M(s-2,f,\chi,\mr{Spin}) L^{(M)}(s-2, \pi \otimes \chi,\mr{Spin}),
\]
where $L_M(s,f,\chi,\mr{Spin})$ depends only on $f$, $\chi$, and the primes dividing $M$. If $M=M'$ and $U_S f=0$ for all $S \neq 1$, then $L_M(s-2,f,\chi,\mr{Spin})=(-1)^{\Omega(M)}M^{-2s}$ for \index{$\Omega(M)$}$\Omega(M)$ the number of prime divisors of $M$.
\end{thm}

\begin{rmk}
Following the conventions of \cite{pollack1}, we work with automorphic $L$-functions here.  For relationships with motivic $L$-functions, see Appendix \ref{appendix:Deligne}. The automorphic $L$-function in Theorem \ref{thm:intrepn} corresponds to the non-complete $L$-function associated with the Spin Galois representation of \cite{KretShin} with Hodge--Tate weight $\left\{0, 2r-3, 2r-2, \ldots,  6r-6 \right\}$.  It also coincides with the Euler product of \cite[(12)]{Evdo}.
\end{rmk}

\begin{proof}[Proof of Theorem \ref{thm:intrepn}]
We first consider the Eisenstein series 
$
E^{\Phi}(g,s)\index{$E^\Phi$}
$ defined exactly as on page 1407 of  \cite{pollack1}.
The first thing to notice is that while doing  the unfolding, thanks to our choice of section, the factor $\chi(\nu(\lambda, Mm))$  appears: just plug our choice of local section $\Phi_v$ in \cite[(5.8)]{pollack1}, i.e. the factor which equals $\chi(\lambda)\chi(\mr{det}(m))$ as $\nu(\lambda,m)=\lambda\mr{det}(m)$.) 

Note also that if $x=x_Mx^{(M)}$, with $x^{(M)}$ the prime-to-$M$ part of $x$, then $\chi(x)=\chi(x^{(M)})$.
 The same proof as  \cite[Proposition 6.2]{pollack1} gives us that the integral at finite places unfolds to
\begin{align*}
\sum_{\lambda' \in \frac{1}{M'}\mathbb{Z}, m \in M_3(\widehat{\ZZ}) \textrm{ and }m^{-1}Tc(m) \in M_3(\widehat{\ZZ})} \Psi(\lambda') \frac{a(\lambda' M m^{-1}Tc(m))\chi(\lambda',Mm)}{\mr{det}(Mm)^{s+r-2}{\lambda'}^{s+3r-6}}.
\end{align*}

For every $\lambda' \in \frac{1}{M'}\mathbb{Z}$, let $d$ be numerator of $\lambda'$, {\it i.e.} $\lambda'=\frac{\lambda}{d}$, with $\lambda \in \mathbb{Z}$ and $(\lambda,d)=1$. We collect in the sum the terms with the same denominator in $\lambda'$:

\begin{align*}
\sum_{d \mid M'} \sum_{\lambda \in \mathbb{Z}, m \in M_3(\widehat{\ZZ}) \textrm{ and }m^{-1}Tc(m) \in M_3(\widehat{\ZZ})} \Psi(\lambda/d) \frac{a(\lambda (M/d) m^{-1}Tc(m))\chi(\lambda,m)}{\mr{det}(m)^{s+r-2}\lambda^{s+3r-5}}.
\end{align*}

We fix $d$, and we have that the internal sum equals
\begin{align*}
a(M/d) d^{3r+s-6}\mr{det}(M)^{2-s-r} \prod_{v \mid d}(-1)\prod_{v\mid M'/d}(1-v^{-1})\times \\
\times \sum_{\substack{\lambda \in \mathbb{Z}, (\lambda,d)=1, m \in M_3(\widehat{\ZZ})\\ \textrm{ and }m^{-1}Tc(m) \in M_3(\widehat{\ZZ})}}  \frac{a(\lambda  m^{-1}Tc(m))\chi(\lambda,m)}{\mr{det}(m)^{s+r-2}\lambda^{s+3r-6}}.
\end{align*}

If  $M=M'$ and all $a(S)=0$ for all $S \neq 1$, then only the term with $d=M$ contributes. We use Theorem \ref{thm:Evdo} to get the desired form of the finite integral.

In general, we write
\begin{align*}
\sum_{\substack{\lambda \in \mathbb{Z}, (\lambda,d)=1, m \in M_3(\widehat{\ZZ})\\ \textrm{ and }m^{-1}Tc(m) \in M_3(\widehat{\ZZ})}}& \frac{a(\lambda  m^{-1}Tc(m))\chi(\lambda,m)}{ \mr{det}(m)^{s+r-2}\lambda^{s+3r-6}}\\
&=\sum_{v_1, \ldots, v_h \mid \frac{M}{d},  (v_i,d)=1}  \sum_{j_i=0}^{\infty} \sum_{\substack{\lambda \in \mathbb{Z}, (\lambda,M)=1, m \in M_3(\widehat{\ZZ})\\ \textrm{ and }m^{-1}Tc(m) \in M_3(\widehat{\ZZ})}} \frac{a(v_1^{j_1}\cdots v_h^{j_h}\lambda  m^{-1}Tc(m))\chi(\lambda,m)}{ \mr{det}(m)^{s+r-2}\lambda^{s+3r-6}(v_1^{j_1}\cdots v_h^{j_h})^{s+3r-6}}.
\end{align*}

We again use Theorem \ref{thm:Evdo} and the fact that $U_S f= a(S)f$ to obtain
\[
a(T)\frac{L^{(M)}(s-2,\pi \otimes \chi,\mr{Spin})}{L(2s-4,\chi)L^{(D_{B})}(2s-2,\chi)} \times \sum_{v_i} \sum_{j_i=0}^{\infty} \frac{a(v_1^{j_1}\cdots v_h^{j_h})}{(v_1^{j_1}\cdots v_h^{j_h})^{s+3r-6}}.
\]

The integral at infinity is given by \cite[Theorem 6.3]{pollack1}.  To conclude the proof, note that the terms by which $E^*_{2r, \chi}$ and $E^\Phi$ differ are exactly the two Dirichlet $L$-functions from \cite[p.~1407]{pollack1} and the extra terms appearing on the right hand side of the first formula in \cite[\S 6.2]{pollack1}. Also, the non-zero factor of {\it loc. cit.} is $\mathrm{det}(T)^3$.
\end{proof}

\subsection{Proof of main algebraicity theorem for Spin $L$-functions}\label{sec:proofofalgthm}
We are now in a position to prove our main algebraicity result, Theorem \ref{thm:algprecise}.
Consistent with the notation earlier in our paper, $\langle, \rangle$ denotes the Petersson inner product for $\GSp_6$, $\chi$ denotes a Dirichlet character, and $c_\chi^\ast = g(\chi)^2$.  Furthermore, we assume $\chi$ is of order dividing $2$.

\begin{rmk}
The center of symmetry for the automorphic $L$-function in Theorem \ref{thm:algprecise} is $1/2$, and the full set of critical points is $\left\{3-r, \ldots, r-2\right\}$.  Thus our result concerns points in the right half of the critical strip, omitting the points $s=1,2,3$ here, as addressed in Section \ref{sec:future}.  Appendix \ref{appendix:Deligne} discusses the relationship with the motivic setting.  The set of critical points of the corresponding motivic $L$-function is $\left\{2r, \ldots, 4r-5\right\},$ the center in that case is $\frac{6r-6}{2}+1$, and our algebraicity result in that setting concerns $\left\{3r+1, \ldots, 4r-5\right\}$.
\end{rmk}

\begin{proof}[Proof of Theorem \ref{thm:algprecise}]
Since $\phi$ satisfies Condition \ref{cond:Tmaxorder}, $\phi$ has a Fourier coefficient $a(T)\neq 0$ for some half-integral positive definite symmetric matrix $T$ corresponding to a maximal order in a quaternion algebra over $\IQ$.
By Theorem \ref{thm:intrepn}, we have
\begin{align*}
L^{(M)}(s-2, \pi \otimes \chi, \Spin) = \frac{I_{2r, \chi}(\phi, s)}{\det(T)^3a(T) L_M(s-2, f_\phi, \chi, \Spin)\Gamma(s,\Spin)}.
\end{align*}
Using the definition of $\Gamma(s,\Spin)$ from Equation \eqref{equ:GammaSpin}, we see that the denominator lies in $\pi^{-4s-6r+14} \IQ(\phi, \chi)$.  The numerator is expressed in terms of the Petersson product of the form $\phi$ on $\GSp_6$ and the restriction of our Eisenstein series to $\GSp_6$ as
\begin{align*}
I_{2r, \chi}(\phi, s) = \langle\phi, \overline{E^*_{2r,\chi}(g,s)}\rangle  = \langle \phi, E_{2r,\chi}(g,s)\rangle = \overline{\langle E_{2r,\chi}(g,s), \phi\rangle},
\end{align*}
where we recall from Section \ref{sec:EseriesZ} that
\begin{align*}
{E^*_{2r,\chi}(g,s)}= {D_{B}^s C_{\infty}}  E^*_{2r,\chi}(Z,s)
\end{align*}
for $Z=g_{\infty} i$.  
Applying Equations \eqref{equ:holoproj} and \eqref{eq:MaassShimura}, we have that for $r\geq s$,
\begin{align*}
\langle E_{2r, \chi}(g, s), \phi\rangle = \langle \mc D^{r-s} E_{2s}(g, s), \phi\rangle = \langle H\left(\mc D^{r-s} E_{2s}(g, s)\right), \phi\rangle,
\end{align*}
where $H\left(\mc D^{r-s} E_{2s}(g, s)\right)$ denotes the holomorphic projection of the restriction to $\GSp_6$ of $\mc D^{r-s} E_{2s}(g, s)$.  That is, although $\mc D^{r-s} E_{2s}(g, s)$ is defined on $G$, we first restrict $\mc D^{r-s} E_{2s}(g, s)$ to $\GSp_6$ and then apply the holomorphic projection operator for $\GSp_6$.  Applying Proposition \ref{prop:equivariantPetersson}, Lemma \ref{lem:holoprojcoeffs}, and Lemma \ref{lem:resnhE}, we have
\begin{align*}
\frac{\langle E_{2r, \chi}, \phi\rangle}{\langle\phi^\natural, \phi\rangle} \in\IQ(\phi, \chi, g(\chi)),
\end{align*}
and when there is a constant $c$ such that the Fourier coefficients of $c\phi$ lie in a CM field,
\begin{align*}
\frac{\langle E_{2r, \chi}, \phi\rangle}{\langle\phi, \phi\rangle}\in\IQ(\phi, \chi, g(\chi)).
\end{align*}
Since we have assumed $\chi$ is of order dividing $2$, we have that the values of $\chi$ lie in $\{0, 1, -1\}$ and that $c_\chi^\ast = g(\chi)^2$ is an integer.  So we have
\begin{align*}
\frac{\langle E_{2r, \chi}, \phi\rangle}{\langle\phi, \phi\rangle}\in\IQ(\phi).
\end{align*}
when $\chi$ is trivial and
\begin{align*}
\frac{\langle E_{2r, \chi}, \phi\rangle}{\langle\phi, \phi\rangle}\in\IQ\left(\phi, \sqrt{c_\chi^\ast}\right)
\end{align*}
for any $\chi$ of order dividing $2$.
Equations \eqref{equ:Lfcnsigmaequivariant} and \eqref{equ:LfcnsigmaequivariantCM} now follow from Equations \eqref{equ:GarrettEquivariance} and \eqref{equ:SahaEquivariance}, respectively.
\end{proof}

\appendix

\section{Relationship with the motivic setting: Deligne's conjecture and critical values}\label{appendix:Deligne}

\subsection{The Spin representation in genus 3}

Let $W_{\mr{st}}$\index{$W_{\mr{st}}$} be a $7$-dimensional orthogonal space, and we write $W_{\mr{st}}=V\oplus L \oplus V^*$ and let $v_3, v_2, v_1, l, v^*_1,v^*_2,v^*_3$ a basis. Here $V$ and $V'$ are hyperbolic subspaces and we suppose the matrix of the orthogonal form in this basis is
$$Q_{\mr{st}} = \left( \begin{array}{ccccccc}
 & & & & & & 1 \\ 
 & & & & & 1 &  \\ 
 & & & & 1& &  \\ 
 & & & 1& & &  \\ 
  & & 1& & & &  \\ 
 & 1& & & & &  \\ 
 1 & & & & & & 
\end{array} \right) $$\index{$Q_{\mr{st}}$}
so that $Q_{\mr{st}} 	$ identifies $V^*$ with the dual of $V$ (hence the notation). Note that $Q_{\mr{st}}(\ell,\ell)=1$.

We define $W_{\mr{spin}}:= \bigwedge^{\bullet} V = \bigoplus_{i=0}^3 \bigwedge^{i} V$\index{$W_{\mr{spin}}$}. We choose a basis \[
\left\{ w,v_3, v_2, v_1, v_1^*:=v_2 \wedge v_3, v_2^*:=v_3 \wedge v_1, v_3^*:=v_1 \wedge v_2,  w':=v_1 \wedge v_2 \wedge v_3 \right\}.
\] 
There is a transposition $\phantom{e}^t$\index{$\phantom{e}^t$ transposition on $W_{\mr{spin}}$} on $W_{\mr{spin}}$ defined as: 
\[ w^t=w, v_i^t=v_i, {(v_i)^*}^t=-v_i^*, {w'}^{t}=w' \]
There is a symmetric quadratic form defined by the following matrix
$$Q_{\mr{spin}} = \left( \begin{array}{cccccccc}
 & & & & & & &  1\\ 
& & & & & & 1 &  \\ 
& & & & & 1& &  \\ 
& & & & 1& & &  \\ 
&  & & 1& & & &  \\ 
& & 1& & & & &  \\ 
& 1 & & & & & & \\
1 & & & & & & & 
\end{array} \right). $$\index{$Q_{\mr{spin}}$}
Note that $\bigwedge^2 V$ is an isotropic $3$-dimensional space, dual to $V$. (Indeed, the symmetric form is built using the pairing $V \times V^*$ induced by $Q_{\mr{st}}$.)

We have then a map of orthogonal spaces 
\begin{align}\label{eq:std to spin}
 T_{\mr{st}}^{\mr{spin}}: W_{\mr{st}}\rightarrow W_{\mr{spin}}
\end{align}\index{$ T_{\mr{st}}^{\mr{spin}}$}
 sending $l \mapsto \frac{1}{\sqrt{2}}(w + w') $.
 We will use this to compare Spin and standard motives for $\mathrm{GSp}_6$.

\subsection{Deligne's critical values}

We recall the definition of Deligne's critical values and periods \cite{Deligne} for a motive $M$ (see also \cite{SchneiderBeilinsonC}): suppose we are given the following data
\begin{itemize}
\item a ($\mathbb{Q}$-rational) Betti realisation $M_B$ of $M$, which is a finite dimensional vector space  over $\mathbb{Q}$ equipped with an involution $F_{\infty}$ and a bigraduation $M^{p,q}$ on $M_B \otimes_{\mathbb{Q}} \mathbb{C}$ such that $F_{\infty}$ exchanges $M^{p,q}$  and $M^{q,p}$; 
\item a ($\mathbb{Q}$-rational) de Rham realisation $M_{\dR}$ of $M$, which is a vector space with a filtration $(F^iM_{\dR})_i$;
\item an isomorphism 
\[
I_\infty: M_B \otimes_{\mathbb{Q}} \mathbb{C} \cong M_\dR \otimes_{\mathbb{Q}} \mathbb{C}
\]
such  that $F^p M_{\dR}\otimes_{\mathbb{Q}} \mathbb{C}$ is the image under $I_{\infty}$ of the Hodge filtration  $\bigoplus_{p' \geq p} M^{p',q}$.
\end{itemize} 
We suppose that $M^{p,q} \neq 0$ only if $p+q=i$ and we say $M$ is pure of weight $i$.  If $i$ is even, we suppose furthermore that $F_{\infty}$ acts on $M^{i/2,i/2}$ by $1$ (this will be the case for the standard representation of $\mathrm{GSp}_{2g}$, see the calculation of the Gamma factors in \cite[Appendix]{BS}, so the hypothesis is harmless in our situation).

We define $M_B^{\pm }$ to be the subspace where $F_{\infty}$ acts as $\pm 1$, and we denote their dimension by $d^{\pm}$.  
We call $F^-M_{\dR}$ the piece of the filtration corresponding under $I_{\infty}$ to $\oplus_{p > q} M^{p,q}$ (if $i$ is even it corresponds $F^{i/2+1}M_\dR \otimes_{\mathbb{Q}} \mathbb{C}$) and $F^+M_\dR$ the piece of the filtration corresponding under $I_{\infty}$ to $\oplus_{p\geq q}H^{p,q}$  (if $i$ is even this corresponds to $F^{i/2}M_\dR \otimes_{\mathbb{Q}} \mathbb{C}$).
We define  $M_{\dR}^{\pm}=M_{\dR}/F^{\mp}M_\dR$.

\begin{defi}
We denote by $h^{pq}$ the dimension of $M^{p,q}$, for $p \neq q$.
We denote by $h^{p+}$ the dimension of $M^{p,p}$.
We define the Gamma factors for $M$, $\Gamma(M,s)$ as 
\[
\Gamma_{\mathbb{R}}(s-i/2)^{h^{p+}} \prod_{p < q } \Gamma_{\mathbb{C}}(s-p).
\]
We say that an integer $n$ is critical for $M$ if neither $\Gamma(M,j)$ nor $\Gamma(M,1+i-n)$ has a pole. 
\end{defi}

\begin{rmk}
 One can say that an integer $n$ is critical if and only if
\[
\dim M_B^{(-1)^{i+n}} =\dim F^{i+1-n}M_{\dR},
\]
see \cite[\S 2]{SchneiderBeilinsonC}
\end{rmk}

Let $f$ be a weight $2r$ Siegel form of genus $3$; suppose it belongs an irreducible cuspidal automorphic representation of $\mathrm{GSp}_6$.  There are conjecturally two motives attached to it: the standard motive  and the Spin motive. We  let $M_{\mr{st},B}$\index{$M_{\mr{st},B}$} and $M_{\mr{spin},B}$\index{$M_{\mr{spin},B}$} be their Betti realisation, with their $(p,q)$-decomposition over $\m C$, and $M_{\mr{st},\dR}$\index{$M_{\mr{st},\dR}$} and $M_{\mr{spin},\dR}$\index{ $M_{\mr{spin},\dR}$} their de Rham realisation, with their filtration.

The standard motive is pure of weight $0$ with Hodge(--Tate) weights 
\[\set{1-2r,2-2r,3-2r,0,2r-3,2r-2,2r-1},\] while the Spin has motivic weight $6r-6$ and the Hodge(--Tate) weights  are 
\[ \set{0,2r-3,2r-2,2r-1,4r-5,4r-4, 4r-3,6r-6}.\]
The set of critical values for the standard motive are well know, see for example  \cite[Appendix]{BS}. 
We calculate now the critical integers for the Spin motive:

\begin{prop}
The critical integers for the motivic $L$-function for the Spin motive are the integers in the strip $\{ 2r, \ldots, 4r-5\}$.
\end{prop}
\begin{proof}
The $\Gamma$-factors of  the completed $L$-function are 
\[
\Gamma(M_{\mr{spin}},s):=\Gamma(s)\Gamma(s-2r+3)\Gamma(s-2r+2)\Gamma(s-2r+1).
\]
(See also \cite[p.~1392]{pollack1} for the automorphic $L$-function.)
A critical value is such that neither $\Gamma(M_{\mr{spin}},s)$ nor $\Gamma(M_{\mr{spin}},w+1-s)$ have poles, for $w=6r-6$.
So we get the conditions $ s>2r-1$ and $(6r-6)+1-2r+1 > s$.
So we see that the critical values are the integers in the interval  
\[ 
 s \in \{2r, \ldots, 4r-5\}.
\]
\end{proof}
In the main body of the paper we interpolate the critical values to the right of the center of symmetry $3r-3+\frac{1}{2}$: from $3r-2$ to $4r-5$.

Note that we can renormalize the Spin $L$-function so that it has a functional equation sending $s$ to $1-s$ just by twisting by the weight \begin{align}
w= \frac{2}{8}\sum_{h \in {\mr{ HT weights}}} h = \frac{24r-24}{4}=6r-6, \label{eqn:motivicw}
\end{align}
 in this case the Hodge--Tate weights become \[ \set{-3r+3,-r,-r+ 1,-r+2 ,r -2, r-1, r, 3r-3}.\]

\subsection{Deligne's periods}

The remarkable conjecture of Deligne is that the whenever $n$ is a critical integer, the trascendental part of the motivic $L$-function $L(M,s)$, conjecturally associated with $M$, evaluated at $n$ is very explicit.

We shall now suppose that $0$ is a critical integer for $M$; if a critical integer exists, we can always reduce to this case by twisting $M$ by a suitable Tate twist.
All the hypotheses put to make $0$ a critical integer  ensure that $I_{\infty}$ induces two isomorphisms
\[
I^{\pm}_{\infty}: M_B^{\pm } \rightarrow M_{\dR}^{\pm}.
\]
\begin{defi}
The $\pm$-period $\Omega^{\pm}$ for  $M$ is defined as the determinant of $I^{\pm}_{\infty}$ in any $\Q$-rational basis of $M_B$ and $M_{\dR}$ (it is well-defined only up $\Q^{\times}$).
\end{defi}

Twisting by a Tate twist would modify the period by a power of $2 \pi i$, see \cite[(5.1.9)]{Deligne}.

We want to study the two periods for the Spin motive, $\Omega^+_{\mr{spin}}$\index{$\Omega^+_{\mr{spin}}$} and $\Omega^-_{\mr{spin}}$\index{$\Omega^-_{\mr{spin}}$}, and compare them with the period for the standard motive $\Omega^+_{\mr{st}}$\index{$\Omega^+_{\mr{st}}$}.

It is well know that the Petersson norm of $f$ is known to be the period  $\Omega^+_{\mr{st}}$ for the standard \cite{BS,liuJussieu,liu-rosso}. In Theorem \ref{thm:algprecise} we proved that the Petersson norm of $f$ (up to a power of $\pi$) is the transcendental part of the Spin $L$-function at critical integers; this seems to hint that $\Omega^\pm_{\mr{spin}}=\Omega^+_{\mr{st}}$, a very surprising result as in general there should not be relations between periods of different motives. We will explain these relations in what follows.

First recall that complex conjugation acts trivial on $ M_{\mr{st},B}^{0,0}$; see again  \cite[Appendix]{BS}.

For the standard $F^1M_{\mr{st},\dR}\otimes \m C= F^- M_{\mr{st},\dR} \otimes \m C= \oplus_{i \geq 1} M^{i,-i}_{\mr{st}}$ is three dimensional, and for the spin $F^{1}M_{\mr{spin},\dR} \otimes \m C = F^1 M_{\mr{spin},\dR} \otimes \m C= \oplus_{i \geq 1} M^{i,-i}_{\mr{spin}}$ is four dimensional. 

Note that the dimension of the $+$-part on $M$ is the same for both motives (namely, four) and  this is a consequence of the numerical coincidence $2^{g}/2=g+1$ for $g=3$, then genus of $f$.
The general framework of motives \cite[AM5]{DeligneP1} says that $F_{\infty}$ is the same as the image of complex conjugation in the corresponding $p$-adic realisation.
The matrix of  complex conjugation has been conjectured in  \cite[Conjecture 5.16]{BuzzGee} and this conjecture has been proved in \cite[Remark 10.8]{KretShin}.

Given the sketch of proof in {\it loc. cit.} we can assume that is (up to conjugation) the matrix of complex conjugation n the $8$-dimensional Galois representation and that it acts on $M_{\mr{spin},B}$ as $Q_{\mr{spin}}$ acts on $W_\mr{spin}$. 
Similarly, we can identify the action of complex conjugation on $M_{\mr{st},B}$ is given by $Q_{\mr{st}}$ acts on $W_\mr{st}$. 
We shall assume that we have a map $T_{B}:M_{\mr{st},B} \rightarrow M_{\mr{spin},B} $\index{$T_{B}$} which is equivariant for complex conjugation, induced by the map $ T_{\mr{st}}^{\mr{spin}}$ of \eqref{eq:std to spin}.
We shall also suppose that we have a similar map, always induced by $T$, $T_{\dR}:M_{\mr{st},\dR} \rightarrow M_{\mr{spin},\dR} $\index{$T_{\dR}$}, respecting the filtration, and these maps are compatible with the comparison isomorphisms $I_{\mr{spin}}$ and $I_{\mr{spin}}$.

\begin{prop}
If the maps $T_{\dR}$ and $T_{B}$ exist with the desired properties, and they are defined over $\mathbb{Q}$, then we have an equality $\Omega^{+}_{\mr{spin}} = \Omega^{+}_{\mr{st}}$.
\end{prop}
\begin{proof}
Using the map $T_B$ we can  identify  $M^+_{\mr{st},B}$ and $M^+_{\mr{spin},B}$, and using $T_{\dR}$ we identify $M^+_{\mr{st},\dR}$ and $M^+_{\mr{spin},\dR}$.
 Now the two comparison isomorphisms induce the following diagram
\begin{displaymath}
    \xymatrix{ 
    M^+_{\mr{spin},B}\ar[r]^{I^{+}_{\mr{spin}}} & M_{\mr{spin},\dR}/F^1 M_{\mr{spin},\dR}\ar[d]^{T_{\dR}^{-1}}  \\
               M_{\mr{st},B}^+ \ar[u]^{T_B}\ar[r]^{I^{+}_{\mr{st}}} & M_{\mr{st},\dR}/F^1 M_{\mr{st},\dR}
               }
\end{displaymath}
where all arrows are isomorphisms. Hence, as the $T$ maps are rationally defined, we get that  the $+$-period $\Omega^{+}_{\mr{spin}}$ for the Spin $L$-function coincides with the $+$-period $\Omega^{+}_{\mr{st}}$ for the standard $L$-function. 
\end{proof}
For larger genus, we still have the same diagram, but the vertical maps won't be isomorphisms, again we point out that  we rely on the numerical coincidence of before $g+1=2^{g-1}$, which holds only for $g=3$.

We are left to study why $\Omega^{+}_{\mr{spin}} =\Omega^{-}_{\mr{spin}} $. An attentive reader could remember that this happens also for the motive associated with the Rankin--Selberg product of two  modular forms (see \cite{hida88}). Indeed both for the Rankin--Selberg and the Spin representation of a genus $3$ Siegel modular form, the corresponding motives are orthogonal (see \cite[Lemma 0.1]{KretShin}), and this extra symmetry on the motive is a rational automorphism that swaps the $+$ and $-$ eigenspaces. 
This has been already observed in \cite[Theorem]{deligne2024motives}, which we summarize here:
\begin{prop}
If the Spin and standard motives are compatible under the map of orthogonal spaces $ T_{\mr{st}}^{\mr{spin}}$ of \eqref{eq:std to spin} than we have an equality of Deligne periods
\[
\Omega^{+}_{\mr{spin}} =\Omega^{-}_{\mr{spin}}.
\]
\end{prop}

\bibliographystyle{alpha} 
\bibliography{ERSbib} 
\printindex
\end{document}